\documentclass[10pt,reqno,b5paper]{amsart}%
\usepackage{anysize}
\marginsize{15mm}{15mm}{10mm}{10mm}
\usepackage[noadjust]{cite}
\usepackage[T1]{fontenc}
\usepackage[colorlinks,urlcolor=black,citecolor=blue,linkcolor=blue]{hyperref}
\usepackage{bm}
\usepackage{amsmath}
\usepackage{mathrsfs}
\usepackage{amsfonts}
\usepackage{amssymb}
\usepackage{amsthm}
\usepackage{graphicx}
\usepackage{cancel}
\usepackage{esint}
\DeclareSymbolFont{CM}{OMX}{cmex}{m}{n}
\DeclareMathSymbol{\sumop}{\mathop}{CM}{"50}
\renewcommand{\sum}{\sumop}

\allowdisplaybreaks[4]

\newtheorem{thm}{\sc Theorem}[section]

\newtheorem{lem}{\sc Lemma}[section]
\newtheorem{pro}{\sc Proposition}[section]
\theoremstyle{definition}
\newtheorem{defn}[thm]{\sc Definition}
\theoremstyle{remark}
\newtheorem{rem}{Remark}[section]
\newtheorem*{clm}{Claim}
\newtheorem*{pf}{Proof}
\newcommand{\Cref}[1]{\eqref{#1}}
\newcommand{\cref}[1]{(\ref{#1})}

\numberwithin{equation}{section}

\begin{document}
\title[Non-uniqueness for the compressible 3D MHD equations]
{Non-uniqueness for the Hypo-dissipative compressible 3D Magnetohydrodynamic Equations}
\author[H. B. Li \quad\&\quad P. Qu]
{Haobin Li $^a$, \quad Peng Qu $^b$}
\thanks{ E-mails: hbli25@m.fudan.edu.cn  (H. B. Li),\quad pqu@fudan.edu.cn (P. Qu).\\This research is supported in part by NSFC (Nos. 12431007, 62588101).}
\dedicatory{
\small
$^a$ School of Mathematical Sciences, \\Fudan University, Shanghai 200433, P. R. China \\
$^b$ School of Mathematical Sciences \\and Shanghai Key Laboratory for Contemporary Applied Mathematics, \\Fudan University, Shanghai 200433, P. R. China
}
\vspace{-5mm}
\begin{abstract}
We consider the compressible 3D magnetohydrodynamic (MHD) equations under general pressure laws. For all hypo-viscosities $(-\Delta)^{\alpha_1}$ and hypo-resistivity $(-\Delta)^{\alpha_2}$ with $\alpha_1,\alpha_2\in(0,1)$, we prove the non-uniqueness of weak solutions to 3D MHD equations which reveals that there exist infinitely many weak solutions with the same initial data.  Also, for the weak solutions in $C^{\widetilde{\beta}}_{t,x}$ to the compressible ideal MHD, where $\widetilde{\beta}>0$, we prove that they are the strong vanishing viscosity and resistivity limit  of the weak solutions to the hypo-dissipative compressible MHD.
\end{abstract}
\maketitle
\tableofcontents
\section{Introduction and main results}
\subsection{Background}
In this paper, we are concerned with the three-dimensional compressible hypo-dissipative magnetohydrodynamic (MHD in short) equations on the torus $\mathbb{T}^3:=[-\pi,\pi]^3$, which couples the compressible Navier-Stokes equations with the Faraday-Maxwell system via Ohm's law. Mathematically, this system reads:
\begin{equation}\label{eq1.1}
     \left\{
    \begin{aligned}
        &\partial_t\rho+{\rm div}(\rho u)=0,\\
        &\partial_t(\rho u)+{\rm div}(\rho u\otimes u)+\nabla P(\rho)={\rm div}\mathbb{S}(u)+{\rm div}\left(\mu\left(b\otimes b-\frac{1}{2}|b|^2\mathbb{I}\right)\right),\\
        &\partial_tb+{\rm div}(b\otimes u-u\otimes b)+\eta\mu^{-1}{\rm curl}({\rm curl}b)=0,\\
        &{\rm div} b=0,
    \end{aligned}
    \right.
\end{equation}
where $\rho(t,x):[0,T]\times\mathbb{T}^3\rightarrow(0,+\infty)$, $u(t,x):[0,T]\times\mathbb{T}^3\rightarrow\mathbb{R}^3$ and $b(t,x):[0,T]\times\mathbb{T}^3\rightarrow\mathbb{R}^3$ are the mass density, velocity field and magnetic field, respectively, the terms in the right-hand-side of the momentum equations $\eqref{eq1.1}_2$ are the Newtonian part of the viscous stress and Lorentz force. Furthermore, $\mathbb{S}(u)$ reads as follows:
\begin{equation}\label{eq1.2}
    \mathbb{S}(u)=2\nu^s\left({\rm D}(u)-\frac{1}{3}{\rm div}u\mathbb{I}\right)+\nu^b{\rm div}u\mathbb{I},
\end{equation}
where ${\rm D}(u)=\frac{1}{2}(\nabla u+\nabla^Tu)$ is the deformation tensor. The coefficients $\mu$, $\eta$, $\nu^s$ and $\nu^b$ are all constants. $\mu$ is the magnetic permeability, $\eta$ is the resistivity, $\nu^s$ is the shear viscosity coefficient and $\nu^b$ is the bulk viscosity coefficient satisfying the physical assumptions
\begin{align*}
    \nu^s>0,\quad\nu^b+\frac{1}{3}\nu^s\geq0.
\end{align*}

The pressure $P(\rho)$ is a function of the density in a compressible fluid. It is typically assumed to satisfy a monotonicity condition $P(\rho)>0$, $P^{\prime}(\rho)>0$ and $P^{\prime\prime}(\rho)>0$ in the study of well-posedness, as in the classical case of polytropic gases $P(\rho)=A\rho^{\gamma}$ with $A>0$ and $\gamma>1$. In the present work, however, these monotonicity and convexity assumptions are no longer required, and we only assume that
\begin{equation}\label{eq1.3}
    P(\rho)\text{ is $C^2$-regular with respect to $\rho$.}
\end{equation}

Furthermore, to investigate a broader class of dissipative effects in the MHD equations, we shall consider the fractional Laplacian operator $(-\Delta)^{\alpha}$, $\alpha\in(0,1)$, which is defined by the Fourier  transform
\begin{align*}
    \mathcal{F}[(-\Delta)^{\alpha}u](\xi)=|\xi|^{2\alpha}\mathcal{F}[u](\xi),\quad\xi\in\mathbb{Z}^3,
\end{align*}
which plays an important role in applied research and has given rise to numerous models.

Then, based on equations \eqref{eq1.1}, the Newtonian stress \eqref{eq1.2} and the fractional Laplacian operator, we can rewrite the equations in the following form in terms of density, momentum, and magnetic field:
\begin{equation}\label{eq1.4}
    \left\{
    \begin{aligned}
        &\partial_t\rho+{\rm div}m=0,\\
        &\partial_tm+\nu^s(-\Delta)^{\alpha_1}(\rho^{-1}m)-(\nu^b+\frac{1}{3}\nu^s)\nabla{\rm div}(\rho^{-1}m)+\nabla P(\rho)\\
        &\qquad\quad\qquad+{\rm div}\left(\rho^{-1}m\otimes m-\mu\left(b\otimes b-\frac{1}{2}|b|^2\mathbb{I}\right)\right)=0,\\
        &\partial_tb+\eta\mu^{-1}(-\Delta)^{\alpha_2}b+{\rm div}\left(\rho^{-1}\left(b\otimes m-m\otimes b\right)\right)=0,\\
        &{\rm div}b=0,
    \end{aligned}
    \right.
\end{equation}
where $m:=\rho u$.

Moreover, we note that when the viscosity and resistivity vanish $(\nu^s=\nu^b=\eta=0)$, \eqref{eq1.4} reduces to the following compressible ideal MHD equations:
\begin{equation}\label{eq1.5}
    \left\{
    \begin{aligned}
        &\partial_t\rho+{\rm div}m=0,\\
        &\partial_tm+\nabla P(\rho)+{\rm div}\left(\rho^{-1}m\otimes m-\mu\left(b\otimes b-\frac{1}{2}|b|^2\mathbb{I}\right)\right)=0,\\
        &\partial_tb+{\rm div}\left(\rho^{-1}\left(b\otimes m-m\otimes b\right)\right)=0,\\
        &{\rm div}b=0.
    \end{aligned}
    \right.
\end{equation}

For the mathematical theory of the MHD equations, the pioneering work is Duvaut-Lions \cite{Duvaut-Lions-1972-arma}, which is regarded as establishing the basic analytical framework for viscous, resistive incompressible MHD. The subsequent work Sermange-Temam \cite{ST-1983-cpam} refined and systematized the classical incompressible MHD theory, including local well-posedness, two-dimensional global behavior, and structural regularity properties. For the ideal (non-dissipative) incompressible MHD, classical local well-posedness holds in sufficiently smooth Sobolev spaces. Moreover, in analogy with the Beale-Kato-Majda criterion for Euler flows, blow-up criteria have been established in terms of the vorticity and current density, see, for instance, Caflisch-Klapper-Steele \cite{CKS-1997-cmp}. On the regularity criteria and partial regularity of weak solutions in three dimensions, He-Xin established Serrin-type regularity criteria for weak solutions to the incompressible MHD \cite{HX-2005-jde}  and proved partial regularity results \cite{HX-2005-jfa}. Furthermore, it was proved that the two-dimensional incompressible MHD system can exhibit global regularity by Cao-Wu \cite{CW-2011-advance}, even with reduced or fractional magnetic diffusion.

Regarding the compressible MHD, Kawashima \cite{kawashima-1984-japanjam} proved global smooth solvability and asymptotic stability for small perturbations of constant states within the hyperbolic-parabolic framework. Fan-Yu \cite{FY-2009-rwa} extended the local strong solutions theory to compressible MHD with vacuum.  On the weak solution theory, Hu-Wang  \cite{HW-2010-arma} proved the existence and large-time behavior of global finite-energy weak solutions for the three-dimensional isentropic compressible MHD in a bounded domain. The small-data classical solutions theory was refined by Chen-Tan \cite{CT-2010-nonlinear}, who obtained global smooth solutions and convergence rates near equilibrium, and  Li-Yu \cite{LY-2011-PRSESA} derived optimal decay rates for classical solutions. A further development was the blow-up criteria, see for \cite{LDY-2011-jmaa,Z-2015-jmaa}.

The research on the flexible theory of MHD is relatively scarce, even in the incompressible case. The successful application of the convex integration method in fluid mechanics has provided an important tool for the study of flexible theories in MHD. Since the groundbreaking works \cite{DS-2009-ann,DS-2010-arma} by De Lellis and Sz\'ekelyhidi on the existence of infinitely many solutions to incompressible Euler equations, there have been significant progress towards the non-uniqueness problem for various fluid models in the last decade. One milestone is the resolution of the
flexible part of Onsager’s conjecture ($C^{\frac{1}{3}-}_{t,x}$-dissipation) for incompressible Euler equations, by Isett \cite{Isett-2018-ann} and Buckmaster-De
Lellis-Sz\'ekelyhidi-Vicol \cite{BDS-2019-cpam}. Regarding the incompressible fluid models with viscosity, in the breakthrough work \cite{BV-2019-annmath} Buckmaster-Vicol proved the non-uniqueness of weak solutions to the 3D incompressible Navier-Stokes equations. Recently, Li-Qu-Zeng-Zhang \cite{LQZZ-2024-jmpa} proved the sharp non-uniqueness to 3D incompressible Navier-Stokes equation with hyper-dissipation ($\alpha\in[1,2)$) in supercritical space $L^{\gamma}_tW^{s,p}_x$, and Cheskidov-Dai-Palasek \cite{CDP-2026-arxiv} obtained instantaneous Type I blow-up and non-uniqueness of smooth solutions of the Navier-Stokes equations via inverse energy cascade.

For incompressible MHD, Bronzi-Lopes Filho--Nussenzveig \cite{BLN-2015-cms} firstly gave the weak solutions with non-trivial energy and vanishing magnetic helicity. Moreover, the Wild solutions with compact support in space-time were constructed by Faraco-Lindberg-Sz\'ekelyhidi \cite{FLS-2021-arma}, via $L^{\infty}_{t,x}$ convex integration scheme. Also, Faraco-Lindberg-Sz\'ekelyhidi \cite{FLS-2021-arma} built a weak solution whose magnetic helicity is conserved, but it violates the energy conservation. The magnetic helicity is commonly expected in the plasma physics to be conserved in the infinite conductivity limit, known as Taylor’s conjecture (see \cite{Taylor-1974}) which was proved by  Faraco-Lindberg \cite{FL-2020-cmp}.  For general weak solutions, Beekie-Buckmaster-Vicol \cite{BBV-2020-annpde} established a $C_tL^2_x$ weak solution violating the conservation of magnetic helicity via the intermittent flow. In addition, the non-uniqueness of dissipative MHD has been obtained by Li-Zeng-Zhang \cite{LZZ-2022-jmpa,LZZ-2024-jfa} and Nie-Ye \cite{NY-2025-jns}, and recently, the Onsager-type conjecture of ideal MHD has also been proved by Faraco-Lindberg-Sz\'ekelyhidi \cite{FLS-2024-cpam} and Miao-Nie-Ye \cite{MNY-2025-annpde}. A subsequent development is that \cite{EPP-2025-arxiv} established anomalous dissipation in ideal MHD with nonzero magnetic helicity when the H\"older exponent is less than $\frac{1}{200}$. Recently, Giardi-Sz\'ekelyhidi \cite{GS-2026-arxiv} have sharpened it to $C_{t,x}^{\frac{1}{5}-}$, representing a significant advance in the Onsager-type conjecture for non-vanishing magnetic helicity. More recently, for instance, Chen-Liu-Yin \cite{CLY-2026-arxiv} proved the non-uniqueness of smooth solutions of the 5D-MHD from critical data in $BMO^{-1}$, while Dai \cite{Dai-2026-arxiv} obtained instantaneous blow-up and non-uniqueness of smooth solutions.

Even more scarce is the research on the flexible theory of compressible MHD. It is because the relative rigidity of pressure $\nabla P(\rho)$ must be taken into account that the large errors cannot be absorbed. Currently, only Feireisl-Li \cite{FL-2020-nonlinearity} provided a result for inviscid case. 

To the best of our knowledge, the flexible theory of \eqref{eq1.4} remains open and Theorem \ref{thm-nonuniquenes-1} below gives the first result for the non-uniqueness of weak solutions to the 3D compressible MHD with both hypo-viscosity and hypo-resistivity.

\subsection{Main results}
To begin with, let us formulate precisely the definition of weak solutions in the distributional sense to the equations \eqref{eq1.4}.
\begin{defn}\label{def-weak-solutions}(Weak solutions). Let $T\in(0,+\infty)$. Given any initial datum $\rho_0\in L^{\infty}(\mathbb{T}^3)$, $\rho_0>0$, and $m_0,b_0\in H^{-1}(\mathbb{T}^3)$ with ${\rm div}b_0=0$, we say that $(\rho,m,b)\in L^{\infty}([0,T]\times\mathbb{T}^3)\times \left(L^{2}([0,T]\times\mathbb{T}^3)\right)^2$ is a weak solution for hypo-dissipative compressible MHD equations \eqref{eq1.4} if
\begin{itemize}
    \item $\rho\geq0$ a.e. and 
    \begin{align*}
        \int_{\mathbb{T}^3}\rho_0(x)\psi(0,x){\rm d}x=-\int_0^T\int_{\mathbb{T}^3}\rho\partial_t\psi+(m\cdot\nabla)\psi{\rm d}x{\rm d}t,
    \end{align*}
    for any test function $\psi\in C^{\infty}_0([0,T)\times\mathbb{T}^3;\mathbb{R})$.\\
    \item $m=0$ whenever $\rho=0$, and
    \begin{align*}
        \int_{\mathbb{T}^3}m_0\cdot\psi(0,x){\rm d}x&
        =-\int_0^T\int_{\mathbb{T}^3}m\cdot\partial_t\psi-\rho^{-1}m\cdot\left(\nu^s(-\Delta)^{\alpha_1}\psi-(\nu^b+\frac{1}{3}\nu^s)\nabla{\rm div}\psi\right)+P{\rm div}\psi{\rm d}x{\rm d}t\\
        &\quad-\int_0^T\int_{\mathbb{T}^3}\left(\rho^{-1}m\otimes m-\mu(b\otimes b-\frac{1}{2}|b|^2\mathbb{I})\right):\nabla\psi{\rm d}x{\rm d}t,
    \end{align*}
    for any test function $\psi\in C^{\infty}_0([0,T)\times\mathbb{T}^3;\mathbb{R}^3)$, where $m_0=\rho_0u_0$.\\
    \item $b$ is divergence-free and 
    \begin{align*}
        \int_{\mathbb{T}^3}b_0\cdot\psi(0,x){\rm d}x=-\int_0^T\int_{\mathbb{T}^3}b\cdot\partial_t\psi-\eta\mu^{-1}b\cdot\left((-\Delta)^{\alpha_2}\psi\right)+\rho^{-1}\left(b\otimes m-m\otimes b\right):\nabla\psi{\rm d}x{\rm d}t,
    \end{align*}
    for any test function $\psi\in C^{\infty}_0([0,T)\times\mathbb{T}^3;\mathbb{R}^3)$.
\end{itemize}
\end{defn}

The non-uniqueness of weak solutions for hypo-dissipative compressible MHD \eqref{eq1.4} is summarized in the following Theorem \ref{thm-nonuniquenes-1}.
\begin{thm}\label{thm-nonuniquenes-1}
    Let $\alpha_1,\alpha_2\in(0,1)$. Let $\widetilde{b}$ be any divergence-free vector field and $(\widetilde{\rho},\widetilde{m})$ be any smooth solution to the transport equation
    \begin{equation}\label{eq1.6s}
        \partial_t\widetilde{\rho}+{\rm div}\widetilde{m}=0,    \end{equation}
  on $\mathbb{T}^3$, such that $0<c_*\leq\widetilde{\rho}(t,x)\leq C_*$ for all $(t,x)\in[0,T]\times\mathbb{T}^3$, where $c_*$ and $C_*$ are both universal constants. Then, there exists $\beta^{\prime}\in(0,1)$, such that for any $\varepsilon_*>0$ and any exponents $(p,s)$ satisfying 
  \begin{equation}\label{eq1.7s}
      \alpha+s-\frac{2\alpha}{p}<0,\quad (p,s)\in[1,2]\times[0,1),\quad\alpha:={\rm max}\{\alpha_1,\alpha_2\},
  \end{equation}
  there exist $(\rho,m,b)$ to \eqref{eq1.4} such that the following holds:
 
(i) Weak solutions: $(\rho, m,b)$ is the weak solution to \eqref{eq1.4} in the sense of Definition \ref{def-weak-solutions} but with $m_0,b_0 \in H_x^{-1}$.

(ii) Regularity:
$$
\rho \in C_t C_x^1, \quad m,b \in H_t^{\beta^{\prime}} C_x \cap L_t^p C_x^s \cap C_t H_x^{-1} .
$$

(iii) Mass preservation:
$$
\int_{\mathbb{T}^3} \rho(t, x) \mathrm{d} x=\int_{\mathbb{T}^3} \widetilde{\rho}(t, x) \mathrm{d} x,\quad \forall t \in[0, T] .
$$

(iv) Small deviation of norms and relative magnetic helicity:
$$
\begin{aligned}
    &\|\rho-\widetilde{\rho}\|_{C_t C_x^1} \leq \varepsilon_*,\quad\|m-\widetilde{m}\|_{L_t^1 C_x \cap L_t^p C_x^s \cap C_t H_x^{-1}} \leq \varepsilon_*,\\
    &\|b-\widetilde{b}\|_{L_t^1 C_x \cap L_t^p C_x^s \cap C_t H_x^{-1}} \leq \varepsilon_*,\quad\|\mathcal{H}_{\rm rel}(b\mid \bar b)-\mathcal{H}_{\rm rel}(\widetilde{b}\mid \bar {\widetilde{b}})\|_{L^1_t} \leq \varepsilon_*,
\end{aligned}
$$
where relative magnetic helicity is defined as
$$
\mathcal{H}_{\rm rel}(b\mid \bar b)(t)
:=
\int_{\mathbb{T}^3}
\widehat{a}(t,x)\cdot
b^{\diamond}(t,x){\rm d}x,
\quad
 b^{\diamond}:=b-\fint_{\mathbb{T}^3}b{\rm d}x,\quad {\rm curl}\widehat{a}=b^{\diamond}.
$$

(v) Small deviation of temporal support: if $\widetilde{T}:=\inf _{t \in[0, T]}\{t \mid \nabla \widetilde{\rho}(t, \cdot) \neq 0, \widetilde{m}(t, \cdot)\neq 0, \widetilde{b}(t,\cdot)-b_* \neq 0\}>0$, then 
$$
\operatorname{supp}_t(\nabla \rho, m,b-b_*) \subseteq N_{\varepsilon_*}([\widetilde{T}, T]),
$$
where 
$$
b_*=\fint_{\mathbb{T}^3}b(0,x){\rm d}x.
$$
Note that, for any $A \subseteq[0, T]$ and $\varepsilon_*>0, N_{\varepsilon_*}(A)$ denotes the $\varepsilon_*$-neighborhood of $A$ in $[0, T]$, namely,
$$
N_{\varepsilon_*}(A):=\left\{t \in[0, T]: \exists s \in A \text {, s.t. }|t-s| \leq \varepsilon_*\right\} .
$$
\end{thm}

As a consequence, Theorem \ref{thm-nonuniquenes-1} gives the following result concerning the non-uniqueness of weak solutions
and the non-conservation of relative magnetic helicity.
\begin{thm}\label{thm-nonuniqueness-2}
    There exist smooth initial datum $\rho_0>0$, $m_0$ and $b_0$ such that for any exponents $(p,s)$ satisfying \eqref{eq1.7s}, there exist infinitely many weak solutions $(\rho,m,b)\in C_tC_x^1\times\left(L^p_tC^s_x\right)^2$ to the hypo-dissipative compressible MHD \eqref{eq1.4} with the same initial datum $(\rho_0,m_0,b_0)$ and do not conserve the relative magnetic helicity.
\end{thm}

Furthermore, we also consider the strong vanishing viscosity and resistivity limit, which relates the hypo-dissipative compressible MHD and ideal compressible MHD.
\begin{thm}\label{thm-vanishing limit}(strong vanishing viscosity and resistivity limit). Let $\alpha_1,\alpha_2\in(0,1)$ and $\rho,m,b\in C^{\widetilde{\beta}}_{t,x}$, $\widetilde{\beta}>0$, be any weak solution to the compressible ideal MHD \eqref{eq1.5}, such that $c_1\leq\rho\leq c_2$ for some constants $c_1,c_2>0$. Then there exist $\beta^{\prime}\in(0,\widetilde{\beta})$ and a sequence of weak solutions $(\rho^{(n)},m^{(n)},b^{(n)})\in C_{t,x}\times \left(H^{\beta^{\prime}}_tC_x\right)^2$ to the hypo-dissipative compressible MHD \eqref{eq1.4} with viscous coefficients $\kappa_n\nu^s$ and $\kappa_n\nu^b$ and resistivity $\kappa_n\eta$ such that
\begin{align}\label{eq1.8s}
    \rho^{(n)}\rightarrow\rho\text{ strongly in }C_{t,x},\text{ and }m^{(n)},b^{(n)}\rightarrow m,b\text{ strongly in }H^{\beta^{\prime}}_tC_x\text{ as }\kappa_n\rightarrow0.
\end{align}
\end{thm}

\noindent\textbf{Organization.} Now we introduce the organization of the paper.  In Section \ref{sec2}, we present the main iteration estimates, which play a crucial role in the subsequent proof of Theorem \ref{thm-nonuniquenes-1}. The mollification procedure is also employed to avoid the loss of derivatives. In Section \ref{sec3}, we construct the key building blocks and apply them to the construction of the momentum, magnetic field, and density perturbations. Corresponding estimates and certain cancellation identities are provided. In Section \ref{sec4}, we focus on the treatment of the Reynolds stress and magnetic stress, and prove the relevant inductive estimates. Finally, in Section \ref{sec5}, we establish the main results.
\\ \hspace*{\fill}\\
\noindent\textbf{Notations.}
We denote for $ p \in[1, \infty], s \in \mathbb{R}, N \in \mathbb{N}$ and $\eta \in(0,1)$,
$$
L_t^p:=L^p(0, T), \quad L_x^p:=L^p(\mathbb{T}^3), \quad C_x^N:=C^N(\mathbb{T}^3), \quad C_x^{N, \eta}:=C^{N, \eta}(\mathbb{T}^3), \quad W_x^{s, p}:=W^{s, p}(\mathbb{T}^3),
$$
where $W_x^{s, p}$ is the usual Sobolev space, $C_x^{N, \eta}$ is the H\"older space equipped with the norm
$$
\|u\|_{C_x^{N, \eta}}:=\sum_{0 \leq|\zeta| \leq N}\left\|\nabla^\zeta u\right\|_{C_x}+\max _{|\zeta|=N} \sup _{x \neq y \in \mathbb{T}^3} \frac{\left|\nabla^\zeta u(x)-\nabla^\zeta u(y)\right|}{|x-y|^\eta},
$$
and $\zeta=\left(\zeta_1, \zeta_2,\zeta_3\right)$ is the multi-index with $\nabla^\zeta:=\partial_{x_1}^{\zeta_1} \partial_{x_2}^{\zeta_2}\partial_{x_3}^{\zeta_3}$. When $N=0$, we denote $C_x^\eta:=C^{0, \eta}\left(\mathbb{T}^3\right)$ for brevity. We also use the shorthand notation $L_t^\gamma L_x^p$ to denote $L^\gamma\left(0, T ; L^p\left(\mathbb{T}^3\right)\right)$, where $p, \gamma \in[1, \infty]$. In particular, we write $L_{t, x}^p:=L_t^p L_x^p$ for short. Moreover, let
$$
\|u\|_{W_{t, x}^{N, p}}:=\sum_{0 \leq m+|\zeta| \leq N}\left\|\partial_t^m \nabla^\zeta u\right\|_{L_{t, x}^p}, \quad\|u\|_{C_{t, x}^N}:=\sum_{0 \leq m+|\zeta| \leq N}\left\|\partial_t^m \nabla^\zeta u\right\|_{C_{t, x}},
$$
Given any Banach space $X, C([0, T] ; X)$ denotes the space of continuous functions from $[0, T]$ to $X$, equipped with the norm $\|u\|_{C_t X}:=\sup _{t \in[0, T]}\|u(t)\|_X$.

We would also write $A \lesssim B$ to imply that $A \leq C B$ for some constant $C>0$ independent of the parameters $q$, $a$, $b$ and $\beta$.
\section{Main iteration and mollification}\label{sec2}
In this section, we will introduce the main iteration of the density, momentum, magnetic field and the Reynolds and magnetic stress, which is the heart of the proof of main results.

\subsection{Main iteration}
For each integer $q\in\mathbb{N}$, we consider the following relaxation system
\begin{equation}\label{eq2.1}
    \left\{
    \begin{aligned}
        &\partial_t\rho_q+{\rm div}m_q=0,\\
        &\partial_tm_q+\nu^s(-\Delta)^{\alpha_1}(\rho^{-1}_qm_q)-(\nu^b+\frac{1}{3}\nu^s)\nabla{\rm div}(\rho^{-1}_qm_q)+\nabla P(\rho_q)\\
        &\qquad+{\rm div}\left(\rho_q^{-1}m_q\otimes m_q-\mu\left(b_q\otimes b_q-\frac{1}{2}|b_q|^2\mathbb{I}\right)\right)={\rm div}R^m_q,\\
        &\partial_tb_q+\eta\mu^{-1}(-\Delta)^{\alpha_2}b_q+{\rm div}\left(\rho_q^{-1}\left(b_q\otimes m_q-m_q\otimes b_q\right)\right)={\rm div}R^b_q,\\
        &{\rm div}b_q=0,
    \end{aligned}
    \right.
\end{equation}
where $R^m_q$ and $R^b_q$ are the Reynolds stress ($3\times3$ symmetric matrix) and magnetic stress ($3\times3$ skew-symmetric matrix) respectively.

In order to measure the size of relaxed solutions $(\rho_q,m_q,b_q,R^m_q,R^b_q)$ accurately, we need to carefully select the frequency parameter $\lambda_q$ and amplitude parameter $\delta_q$. More precisely, let $\alpha={\rm max}\left\{\alpha_1,\alpha_2\right\}$. $a\in\mathbb{N}$ is a sufficiently large integer to be determined later, $\varepsilon\in\mathbb{Q}$ is sufficiently small such that
\begin{equation}\label{eq2.2s}
    \varepsilon\leq\frac{1}{20}{\rm min}\left\{1-\alpha,\alpha,\frac{2\alpha}{p}-\alpha-s\right\},\quad \varepsilon b\in\mathbb{N},
\end{equation}
with $b\in2\mathbb{N}$ is a large even number and $\beta>0$ is the regularity parameter such that
\begin{equation}\label{eq2.3s}
    b>\frac{1000}{\varepsilon},\quad0<\beta<\frac{1}{100b^2}.
\end{equation}
For $q\in\mathbb{N}$, the frequency parameter $\lambda_q$ and amplitude parameter $\delta_q$ are defined by
\begin{equation}\label{eq2.4}
    \lambda_q=a^{(b^q)},\quad \delta_q=\lambda_q^{-2\beta}\lambda_1^{3\beta},
\end{equation}
where the constant $a>0$ will be chosen later. Furthermore, we assume that the following crucial inductive estimates hold for the relaxed solutions to \eqref{eq2.1} at level $q\in\mathbb{N}$:
\begin{align}
    &C_1+\lambda_q^{-\beta}\leq\rho_q\leq C_2-\lambda_q^{-\beta},\label{eq2.5}\\
    &\|\partial_t^M\rho_q\|_{C_tC^N_x}\lesssim\lambda_q^{\frac{N\varepsilon}{4}},\label{eq2.6}\\
    &\|m_q\|_{C^N_{t,x}}+\|b_q\|_{C^N_{t,x}}\lesssim\lambda_q^{2N+2},\label{eq2.7}\\
    &\|R^m_q\|_{C^1_{t,x}}+\|R^b_q\|_{C^1_{t,x}}\lesssim\lambda_q^9,\label{eq2.8}\\
    &\|R^m_q\|_{L^1_tC_x}+\|R^b_q\|_{L^1_tC_x}\lesssim\delta_{q+1},\label{eq2.9}
\end{align}
where $1\leq N\leq4$, $0\leq M\leq1$, $C_1$ and $C_2$ are positive universal constants, and the implicit constants are independent of $q$.

The main iteration result is contained in the following theorem.
\begin{thm}\label{thm-main-iteration} (Main iteration). 
    Let $\alpha={\rm max}\left\{\alpha_1,\alpha_2\right\}$ and $(p,s)$ satisfy \eqref{eq1.7s}. Then, there exist $\beta\in(0,1)$ and $a_0$ large enough, such that for any integer $a\geq a_0$, the following holds:

    Suppose that $(\rho_q,m_q,b_q,R^m_q,R^b_q)$ solves \eqref{eq2.1} and satisfies \eqref{eq2.5}--\eqref{eq2.9}. Then, there exists a new relaxation solution $(\rho_{q+1},m_{q+1},b_{q+1},R^m_{q+1},R^b_{q+1})$ to \eqref{eq2.1} which satisfies \eqref{eq2.5}--\eqref{eq2.9} with $q+1$ replacing $q$ and the following estimates:
    \begin{align}
        &\int_{\mathbb{T}^3}\rho_{q+1}(t,x){\rm d}x=\int_{\mathbb{T}^3}\rho_{q}(t,x){\rm d}x,\quad\forall t\in[0,T],\label{eq2.10}\\
        &\|\rho_{q+1}-\rho_q\|_{C_tC_x^1}\lesssim\delta_{q+2}^{\frac{1}{2}},\label{eq2.11}\\
        &\|m_{q+1}-m_q\|_{L^2_tC_x}+\|b_{q+1}-b_q\|_{L^2_tC_x}\lesssim\delta_{q+1}^{\frac{1}{2}},\label{eq2.12}\\
        &\|m_{q+1}-m_q\|_{L^1_tC_x\cap L^p_tC^s_x\cap C_tH^{-1}_x}+\|b_{q+1}-b_q\|_{L^1_tC_x\cap L^p_tC^s_x\cap C_tH^{-1}_x}\lesssim\delta_{q+2}^{\frac{1}{2}},\label{eq2.13}
\end{align}
where the implicit constants are independent of $q$. In addition, concerning the temporal support of relaxed solutions, we denote
\begin{equation}\label{eq2.14}
    T_q:=\inf\left\{t\in[0,T]\mid \left(\nabla\rho_q,m_q,b_q-b_*,R^m_q,R^b_q\right)(t,\cdot)\neq0 \right\}.
\end{equation}
for all $q\in\mathbb{N}$ where $b_*=\fint_{\mathbb{T}^3}b(0,x){\rm d}x$. Then, if $T_0>0$, it holds that
\begin{equation}\label{eq2.15}
    T_{q+1}\geq T_q-\delta_{q+2}^{\frac{1}{2}}>0.
\end{equation}
\end{thm}

The heart of the proof of Theorem \ref{thm-main-iteration} is to construct suitable perturbations
\begin{align*}
    w_{q+1}\simeq m_{q+1}-m_q,\quad d_{q+1}\simeq b_{q+1}-b_q,\quad z_{q+1}\simeq\rho_{q+1}-\rho_q.
\end{align*}
In order to avoid the loss of derivatives in the convex integration scheme, we need to perform the mollification procedure to the relaxed system \eqref{eq2.1}.
\subsection{Mollification procedure}
Let $\phi_{\epsilon}$ and $\varphi_{\epsilon}$ be the standard mollifiers on $\mathbb{R}^3$ and $\mathbb{R}$ with ${\rm supp}\phi_{\epsilon}\subseteq B(0,\varepsilon)$ and ${\rm supp}\varphi_{\epsilon}\subseteq(-\epsilon,\epsilon)$ respectively. Then, the mollifications of $(\rho_q,m_q,b_q,R^m_q,R^b_q)$ in space and time are defined by
\begin{equation}\label{eq2.16}
    \begin{aligned}
        &\rho_{l}:=(\rho_q*_x\phi_l)*_t\varphi_l,\quad P_l:=(P(\rho_q)*_x\phi_l)*_t\varphi_l,\\
        &m_l:=(m_q*_x\phi_l)*_t\varphi_l,\quad b_l:=(b_q*_x\phi_l)*_t\varphi_l,\\
        &R^m_l:=(R^m_q*_x\phi_l)*_t\varphi_l,\quad R^b_l:=(R^b_q*_x\phi_l)*_t\varphi_l,
    \end{aligned}
\end{equation}
where the scale of mollification is given by
\begin{align}\label{eq2.17}
    l:=\lambda_q^{-30}.
\end{align}
Then, by equations \eqref{eq2.1}, $(\rho_l,m_l,b_l,R^m_l,R^b_l)$ satisfies
\begin{equation}\label{eq2.18}
    \left\{
    \begin{aligned}
        &\partial_t\rho_l+{\rm div}m_l=0,\\
        &\partial_tm_l+\nu^s(-\Delta)^{\alpha_1}(\rho^{-1}_lm_l)-(\nu^b+\frac{1}{3}\nu^s)\nabla{\rm div}(\rho^{-1}_lm_l)+\nabla P(\rho_l)\\
        &\qquad+{\rm div}\left(\rho_l^{-1}m_l\otimes m_l-\mu\left(b_l\otimes b_l-\frac{1}{2}|b_l|^2\mathbb{I}\right)\right)={\rm div}\left(R^m_l+R^m_{{\rm com}}\right),\\
        &\partial_tb_l+\eta\mu^{-1}(-\Delta)^{\alpha_2}b_l+{\rm div}\left(\rho_l^{-1}\left(b_l\otimes m_l-m_l\otimes b_l\right)\right)={\rm div}\left(R^b_l+R^b_{{\rm com}}\right),\\
        &{\rm div}b_l=0,
    \end{aligned}
    \right.
\end{equation}
where the symmetric commutator stress $R^m_{{\rm com}}$ and the  skew-symmetric commutator stress $R^b_{{\rm com}}$ are of form
\begin{align}
    &\begin{aligned}
        R^m_{{\rm com}}:=
        &\nu^s\mathcal{R}^m(-\Delta)^{\alpha_1}\left(\rho_l^{-1}m_l-((\rho_q^{-1}m_q)*_x\phi_l)*_t\varphi_l\right)+\mathcal{R}^m\nabla\left(P(\rho_l)-P_l\right)\\
        &-(\nu^b+\frac{1}{3}\nu^s)\mathcal{R}^m\nabla{\rm div}\left(\rho_l^{-1}m_l-((\rho_q^{-1}m_q)*_x\phi_l)*_t\varphi_l\right)\\
        &+\mathcal{R}^m{\rm div}\left(\rho_l^{-1}m_l\otimes m_l-\mu(b_l\otimes b_l-\frac{1}{2}|b_l|^2\mathbb{I})\right)\\
        &-\mathcal{R}^m{\rm div}\left(\left(\rho_q^{-1}m_q\otimes m_q-\mu(b_q\otimes b_q-\frac{1}{2}|b_q|^2\mathbb{I})\right)*_x\phi_l\right)*_t\varphi_l,
    \end{aligned}\label{eq2.19}\\
    & R^b_{{\rm com}}:=\mathcal{R}^b{\rm div}\left(\rho_l^{-1}(b_l\otimes m_l-m_l\otimes b_l)-\left(\left(\rho_q^{-1}(b_q\otimes m_q-m_q\otimes b_q)\right)*_x\phi_l\right)*_t\varphi_l\right),\label{eq2.20}
\end{align}
with $\mathcal{R}^m$ and $\mathcal{R}^b$ being the symmetric and skew-symmetric inverse-divergence operator respectively, given by \eqref{eq4.1} and \eqref{eq4.2} below.

By the standard mollification estimates and inductive estimates \eqref{eq2.5}--\eqref{eq2.9}, for any integers $1\leq N\leq4$ and $0\leq M\leq1$,
\begin{align}
    &C_1+\frac{1}{2}\lambda_q^{-\beta}\leq\rho_l\leq C_2-\frac{1}{2}\lambda_q^{-\beta},\label{eq2.21}\\    &\|\partial^M_t\rho_l\|_{C_tC^N_x}\lesssim\|\partial^M_t\rho_q\|_{C_tC^N_x}\lesssim\lambda_q^{\frac{N\varepsilon}{4}}, \label{eq2.22}\\
    &\|\rho_l-\rho_q\|_{C_tC^{N-1}_x}\lesssim l\|\rho_q\|_{C^1_tC^{N-1}_x}+l\|\rho_q\|_{C_tC^{N}_x}\lesssim l\lambda_q^{\frac{N\varepsilon}{4}}, \label{eq2.23}\\
    &\|m_l\|_{C^N_{t,x}}+\|b_l\|_{C^N_{t,x}}\lesssim l^{-N+1}\left(\|m_q\|_{C^1_{t,x}}+\|b_q\|_{C^1_{t,x}}\right)\lesssim l^{-N+1}\lambda_q^4,\label{eq2.24}\\
    &\|m_l-m_q\|_{C_tC^{N-1}_x}+\|b_l-b_q\|_{C_tC^{N-1}_x}\lesssim l\left(\|m_q\|_{C^N_{t,x}}+\|b_q\|_{C^N_{t,x}}\right)\lesssim l\lambda_q^{2N+2}, \label{eq2.25}\\
    &\|R^m_l\|_{C^N_{t,x}}+\|R^b_l\|_{C^N_{t,x}}\lesssim l^{-N+1}\left(\|R^m_q\|_{C^1_{t,x}}+\|R^b_q\|_{C^1_{t,x}}\right)\lesssim l^{-N},\label{eq2.26}\\
    &\|R^m_l\|_{L^1_tC_x}+\|R^b_l\|_{L^1_tC_x}\lesssim \|R^m_q\|_{L^1_tC_x}+\|R^b_q\|_{L^1_tC_x}\lesssim\delta_{q+1}.\label{eq2.27}
\end{align}
Moreover, concerning the temporal support of $\nabla\rho_l$, $m_l$, $b_l-b_*$, $R^m_l$ and $R^b_l$, we have
\begin{align}\label{eq2.28}
    {\rm supp}_t(\nabla\rho_l,m_l,b_l-b_*,R^m_l,R^b_l)\subseteq N_l\left({\rm supp}_t(\nabla\rho_q,m_q,b_q-b_*,R^m_q,R^b_q)\right).
\end{align}
\section{Perturbations}\label{sec3}
The aim of this section is to construct appropriate momentum, magnetic and density perturbations, such that the corresponding inductive estimates in Theorem \ref{thm-main-iteration} propagate through in the convex integration scheme.

The fundamental building blocks in the convex integration will be indexed by the following parameters
\begin{align}\label{eq3.1}
    \lambda:=\lambda_{q+1},\quad\tau:=\lambda_{q+1}^{2\alpha-10\varepsilon},\quad\sigma:=\lambda_{q+1}^{15\varepsilon},
\end{align}
where $\alpha={\rm max}\{\alpha_1,\alpha_2\}\in(0,1)$, and $\varepsilon$ is the small constant satisfying \eqref{eq2.2s}.
\subsection{Spatial building blocks}
To begin with, let us recall two geometric lemmas in \cite{BBV-2020-annpde} and \cite{LZZ-2022-jmpa} which will be used to construct the basic spatial building blocks, namely, the Mikado flows.

\begin{lem}\label{lem-1st-geometric}(First Geometric Lemma). There exists a finite set $\Lambda_b\subseteq\mathbb{S}^2\cap\mathbb{Q}^3$ consisting of vectors $k$ with orthonormal bases $(k,k_1,k_2)\in \mathbb{S}^2\cap\mathbb{Q}^3$ such that, for some $\varepsilon_b>0$, and smooth positive functions $\gamma_{(k)}:B_{{\rm skew},3}(0,\varepsilon_b)\rightarrow\mathbb{R}$, where $B_{{\rm skew},3}(0,\varepsilon_b)$ is the ball of radius $\varepsilon_b$ centered at $0$ in the space of $3\times3$ skew-symmetric matrices, such that for any $A\in B_{{\rm skew},3}(0,\varepsilon_b)$ we have the following identity:
\begin{align}\label{eq3.2}
    A=\sum_{k\in\Lambda_b}\gamma^2_{(k)}(A)(k_2\otimes k_1-k_1\otimes k_2).
\end{align}
\end{lem}
\begin{lem}\label{lem-2nd-geometric}(Second Geometric Lemma). There exists a finite set $\Lambda_m\subseteq\mathbb{S}^2\cap\mathbb{Q}^3$ consisting of vectors $k$ with orthonormal bases $(k,k_1,k_2)\in \mathbb{S}^2\cap\mathbb{Q}^3$ such that, for some $\varepsilon_m>0$, and smooth positive functions $\gamma_{(k)}:B_{{\rm sym},3}(\mathbb{I},\varepsilon_m)\rightarrow\mathbb{R}$, where $B_{{\rm sym},3}(\mathbb{I},\varepsilon_m)$ is the ball of radius $\varepsilon_m$ centered at the identity matrix $\mathbb{I}$ in the space of $3\times3$ symmetric matrices, such that for any $S\in B_{{\rm sym},3}(\mathbb{I},\varepsilon_m)$ we have the following identity:
\begin{align}\label{eq3.3}
    S=\sum_{k\in\Lambda_m}\gamma^2_{(k)}(S)(k_1\otimes k_1).
\end{align}
Furthermore, we may choose $\Lambda_m$ such that $\Lambda_b\cap\Lambda_m=\varnothing$.
\end{lem}
As pointed out in \cite{BBV-2020-annpde}, there exists $N_{\Lambda}\in\mathbb{N}$ such that
\begin{align}\label{eq3.4}
    \{N_{\Lambda}k,N_{\Lambda}k_1,N_{\Lambda}k_2\}\subseteq N_{\Lambda}\mathbb{S}^2\cap\mathbb{Z}^3,\quad \forall k\in\Lambda_m\cup\Lambda_b.
\end{align}
Furthermore, we denote by $M_*$ the geometric constant such that
\begin{align}\label{eq3.5}
    \sum_{k\in\Lambda_m}\|\gamma_{(k)}\|_{C^4(B_{{\rm sym},3}(\mathbb{I},\varepsilon_m))}+\sum_{k\in\Lambda_b}\|\gamma_{(k)}\|_{C^4(B_{{\rm skew},3}(0,\varepsilon_b))}\leq M_*.
\end{align}
This parameter is universal and will be used later in the estimates of the size of perturbations.

Then, let $\Phi:\mathbb{R}\rightarrow\mathbb{R}$ be a smooth cut-off function supported on $[-1,1]$, and then normalize $\Phi$ such that $\phi:=-\frac{{\rm d}^2}{{\rm d}x^2}\Phi$ satisfies
\begin{align}\label{eq3.6}
    \frac{1}{2\pi}\int_{\mathbb{R}}\phi^2(x){\rm d}x=1.
\end{align}
Under circumstances that do not cause confusion, we periodize $\phi$ and $\Phi$ so that they are treated as periodic functions defined on $\mathbb{T}^3$.

Inspired by \cite{BV-2019-emssms,BBV-2020-annpde,LQZZ-2022-arxiv,LZZ-2022-jmpa}, we choose the Mikado flows as the basic spatial building blocks, defined by
\begin{align*}
        &W_{(k)}:=\phi(\lambda N_{\Lambda}k\cdot x)k_1,\quad k\in\Lambda_m\cup \Lambda_b,\\
        &D_{(k)}:=\phi(\lambda N_{\Lambda}k\cdot x)k_2,\quad k\in\Lambda_b,
\end{align*}
where $W_{(k)}$ and $D_{(k)}$ are the momentum and magnetic Mikado flows respectively. For ease of notation, we set
\begin{equation}\label{eq3.7}
    \begin{aligned}
        &\phi_{(k)}(x):=\phi(\lambda N_{\Lambda}k\cdot x),\\
        &\Phi_{(k)}(x):=\Phi(\lambda N_{\Lambda}k\cdot x),
    \end{aligned}
\end{equation}
and thus, the Mikado flows can be reformulated as 
\begin{equation}\label{eq3.8}
    \begin{aligned}
        &W_{(k)}=\phi_{(k)}k_1,\quad k\in\Lambda_m\cup \Lambda_b,\\
        &D_{(k)}=\phi_{(k)}k_2,\quad k\in\Lambda_b.
    \end{aligned}
\end{equation}
Next we present the $C^N_x$-estimates of Mikado flows in the following lemma, see \cite{BBV-2020-annpde} where $r=1$.
\begin{lem}\label{lem-mikado-flows}(Estimates of Mikado flows). For $N\geq0$ and $k\in\Lambda_m\cup\Lambda_b$, we have
\begin{align}\label{eq3.9}
\|\nabla^N\phi_{(k)}\|_{C_x}+\|\nabla^N\Phi_{(k)}\|_{C_x}\lesssim\lambda^N,
\end{align}
where the implicit constants are independent of $\lambda$ but may depend on $N$.
\end{lem}
\subsection{Temporal building blocks}
Since the Mikado flows chosen as in \eqref{eq3.8} provide no intermittent effect, and thus cannot control the dissipation caused by $(-\Delta)^{\alpha_1}$ and $(-\Delta)^{\alpha_2}$ in \eqref{eq2.1}. We proceed to construct building blocks with the temporal intermittency. This idea was used in \cite{CL-2021-annpde,CL-2022-invent,CL-2023-annpde,LQZZ-2024-jmpa,LZZ-2024-jfa} which obtain sharp non-uniqueness results. However, unlike those works, because of spatial interactions of Mikado flows, inspired by \cite{LQZZ-2022-arxiv,DL-2025-JLMS}, suitable shifts should also be taken into account in the temporal building blocks, so that the supports of different temporal building blocks are disjoint.

More precisely, let $\left\{g_k\right\}_{k \in \Lambda_m\cup\Lambda_b}\subseteq C_c^{\infty}([0, T])$ be cut-off functions such that ${\rm supp}(g_k)\cap{\rm supp}(g_k^{\prime})=\emptyset$ if $k \neq k^{\prime}$, and
\begin{align*}
    \fint_0^Tg^2_k(t){\rm d}t=1
\end{align*}
for all $k \in \Lambda_m\cup\Lambda_b$. The existence of such a family $\left\{g_k\right\}_{k \in \Lambda}$ is guaranteed by the fact that there are finitely many wavevectors in $k \in \Lambda_m\cup\Lambda_b$, e.g., $g_k=g\left(t-\alpha_k\right)$, where $g \in C_c^{\infty}([0, T])$ with  small support and $\left\{\alpha_k\right\}_{k \in \Lambda_m\cup\Lambda_b}$ are the temporal shifts such that the supports of $\left\{g_k\right\}$ are disjoint.

Then, for each $k \in \Lambda_m\cup\Lambda_b$, we rescale the cut-off function $g_k$ by
\begin{align}\label{eq3.10}
    g_{k, \tau}(t)=\tau^{\frac{1}{2}} g_k(\tau t),
\end{align}
where the concentration parameter $\tau$ is given by \eqref{eq3.1}. Then, we periodize $g_{k,\tau}$ such that the resulting functions (by an abuse of notation, still denoted by $g_{k,\tau}$) are periodic functions defined on $[0,T]$.

In order to balance the high temporal oscillations arising from the concentration functions $g_{k,\tau}$, we need the functions $h_{k,\tau}:[0,T]\rightarrow\mathbb{R}$, defined by
\begin{align}\label{eq3.11}
    h_{k, \tau}(t):=\int_0^t\left(g_{k, \tau}^2(s)-1\right) d s, \quad t \in[0, T].
\end{align}

Set
\begin{align}\label{eq3.12}
    g_{(k)}(t):=g_{k, \tau}(\sigma t), \quad h_{(k)}(t):=h_{k, \tau}(\sigma t) .
\end{align}
Then, we obtain
\begin{align}\label{eq3.13}
    \partial_t\left(\sigma^{-1} h_{(k)}\right)=g_{(k)}^2-1=g_{(k)}^2-\fint_0^T g_{(k)}^2(s) \mathrm{d} s,
\end{align}
where $\sigma$ is given by \eqref{eq3.1}.

Now we introduce the crucial temporal intermittent estimates of $g_{(k)}$ and $h_{(k)}$ as in \cite{LQZZ-2022-arxiv} which are summarized in the following lemma.
\begin{lem}\label{lem-temporal-intermittency}(Estimates of temporal intermittency). For any $p\in[1,+\infty]$, $M\in\mathbb{N}$, we have
\begin{align}\label{eq3.14}  \|\partial_t^Mg_{(k)}\|_{L^p_t}\lesssim\sigma^M\tau^{M+\frac{1}{2}-\frac{1}{p}},
\end{align}
where the implicit constants are independent of $\sigma$ and $\tau$ but may depend on $M$. Moreover, we have the uniform bound
\begin{align}\label{eq3.15}
    \|h_{(k)}\|_{C_t}\leq1.
\end{align}
\end{lem}
\begin{rem}
On the one hand, we need to control the $L^1_tC_x$-norm of Reynolds and magnetic error, for which the spatial intermittency would cause even larger error. This suggests that the suitable building blocks should be spatially homogeneous. On the other hand, in order to control the error caused by the dissipation, an extra intermittency (temporal intermittency) should be exploited from the building blocks.
\end{rem}
\subsection{Magnetic and momentum perturbations}
This section is devoted to the construction of the key magnetic and momentum perturbations, including the
principal parts, the almost incompressibility correctors and the temporal correctors. 

To begin with, let us first define suitable amplitudes of perturbations, which play the key role in the cancellation between the low-frequency part of the nonlinearity and the old Reynolds and magnetic stresses.
\subsubsection{Amplitudes}
Let $\chi:[0,+\infty)\rightarrow\mathbb{R}$ be a smooth spatial cut-off function satisfying
\begin{align}\label{eq3.16}
    \chi(z)=\left\{
    \begin{aligned}
     & 1,\quad0\leq z\leq1,\\  
     &z,\quad z\geq2,
    \end{aligned}
    \right.
\end{align}
and
\begin{align}\label{eq3.17}
    \frac{1}{2}z\leq\chi(z)\leq2z,\quad 1<z<2.
\end{align}

Set
\begin{align}\label{eq3.18}
\varrho_b(t)=2\varepsilon_b^{-1}\delta_{q+1}\chi\left(\frac{\|R^b_l(t,\cdot)\|_{C_x}}{\delta_{q+1}}\right),
\end{align}
where $\varepsilon_b>0$ is the small constant in the first geometric lemma, Lemma \ref{lem-1st-geometric}. Then, by \eqref{eq2.8}, \eqref{eq2.9}, \eqref{eq2.26}, \eqref{eq3.16}--\eqref{eq3.18} and standard H\"older estimates, for $1\leq N\leq4$, we have
\begin{align}
    &\left|\frac{R^b_l(t,x)}{\varrho_b(t)}\right|\leq\varepsilon_b,\quad\varrho_b(t)\geq\varepsilon_b^{-1}\delta_{q+1},\quad\forall t\in[0,T],x\in\mathbb{T}^3,\label{eq3.19}\\   &\|\varrho_b\|_{L^p_t}\leq4\varepsilon_b^{-1}\left(T^{\frac{1}{p}}\delta_{q+1}+\|R^b_l\|_{L^p_tC_x}\right),\quad\forall p\in[1,+\infty],\label{eq3.20}\\
    &\|\varrho_b\|_{C_t}\lesssim l^{-1},\quad\|\varrho_b\|_{C^N_t}\lesssim l^{-2N},\label{eq3.21}\\
    &\|\varrho^{\frac{1}{2}}_b\|_{C_t}\lesssim l^{-1},\quad\|\varrho^{\frac{1}{2}}_b\|_{C^N_t}\lesssim l^{-2N},\label{eq3.22}\\
    &\|\varrho^{-1}_b\|_{C_t}\lesssim \varepsilon_b\delta^{-1}_{q+1},\quad\|\varrho_b^{-1}\|_{C^N_t}\lesssim l^{-2N},\label{eq3.23}
\end{align}
where the implicit constants are independent of $q$.

Then, the smooth temporal cut-off function $f_b(t)$ is defined by
\begin{itemize}
    \item $0\leq f_b\leq1\text{ and } f_b\equiv1\text{ on }{{\rm supp}_t(R^b_l)}$;
     \item ${\rm supp}_t(f_b)\subseteq N_l({\rm supp}_t(R^b_l))$;
     \item $\|f_b\|_{C^N_t}\lesssim l^{-N},1\leq N\leq5$.
\end{itemize}

The amplitudes of the magnetic perturbations are defined by
\begin{align}\label{eq3.24}
    a_{(k)}(t,x):=f_b(t)\varrho^{\frac{1}{2}}_b(t)\rho^{\frac{1}{2}}_l(t,x)\gamma_{(k)}\left(\frac{-R^b_l(t,x)}{\varrho_b(t)}\right),\quad k\in\Lambda_b,
\end{align}
where $\gamma_{(k)}$ and $\Lambda_b$ are given in Lemma \ref{lem-1st-geometric}, and $\rho_l$ is the mollified density given by \eqref{eq2.16}.

Note that, by virtue of Lemma \ref{lem-1st-geometric}, the identity \eqref{eq3.2} and the expression \eqref{eq3.24} of $a_{(k)}$, the following identity holds:
\begin{equation}\label{eq3.25}
    \begin{aligned} &\quad\sum_{k\in\Lambda_b}\rho_l^{-1}a^2_{(k)}g^2_{(k)}\left(D_{(k)}\otimes W_{(k)}-W_{(k)}\otimes D_{(k)}\right)\\
  &=-R^b_l+\sum_{k\in\Lambda_b}\rho_l^{-1}a^2_{(k)}g^2_{(k)}\mathbb{P}_{\neq0}\left(D_{(k)}\otimes W_{(k)}-W_{(k)}\otimes D_{(k)}\right)\\
  &\quad+\sum_{k\in\Lambda_b}\rho_l^{-1}a^2_{(k)}\left(g^2_{(k)}-1\right)\fint_{\mathbb{T}^3}\left(D_{(k)}\otimes W_{(k)}-W_{(k)}\otimes D_{(k)}\right){\rm d}x,
    \end{aligned}
\end{equation}
where $\mathbb{P}_{\neq0}$ denotes the spatial projection onto nonzero Fourier modes.

Moreover, the analytic estimates of magnetic amplitudes are included in the following lemma, whose proof is similar to \cite[Lemma 4.1]{LZZ-2022-jmpa}, and so the proof is omitted here.
\begin{lem}\label{lem-magnetic-amplitude-estimate}(Estimates of magnetic amplitude). For any $1\leq N\leq4$, $k\in\Lambda_b$, we have
\begin{align}
    &\|a_{(k)}\|_{L^2_tC_x}\lesssim\delta_{q+1}^{\frac{1}{2}},\label{eq3.26}\\
    &\|a_{(k)}\|_{C_{t,x}}\lesssim l^{-2},\quad\|a_{(k)}\|_{C^N_{t,x}}\lesssim l^{-5N-1},\label{eq3.27}
\end{align}
where the implicit constants are independent of $q$.
\end{lem}

Next we will focus on the momentum amplitudes. Unlike the previous magnetic case, because of the strong coupling between the momentum and magnetic fields, that is, sparked by \cite{BBV-2020-annpde,LZZ-2022-jmpa},  a matrix $G^b$ is indispensable in order to maintain the cancellation between the momentum perturbations and old Reynolds stresses, where
\begin{align}\label{eq3.28}  G^b:=\sum_{k\in\Lambda_b}a^2_{(k)}\left(\rho_l^{-1}\fint_{\mathbb{T}^3}W_{(k)}\otimes W_{(k)}{\rm d}x-\mu\fint_{\mathbb{T}^3}D_{(k)}\otimes D_{(k)}-\frac{1}{2}|D_{(k)}|^2\mathbb{I}{\rm d}x\right).
\end{align}
In view of estimates \eqref{eq2.21}, \eqref{eq2.22}, \eqref{eq3.26} and \eqref{eq3.27}, we have that for $N\geq1$,
\begin{align}\label{eq3.29}
    \|G^b\|_{C_{t,x}}\lesssim l^{-2},\quad\|G^b\|_{C^N_{t,x}}\lesssim l^{-5N-2},\quad\|G^b\|_{L^1_tC_x}\lesssim\delta_{q+1}.
\end{align}

Set
\begin{align}\label{eq3.30} \varrho_m(t):=2\varepsilon_m^{-1}\delta_{q+1}\chi\left(\frac{\|R^m_l(t,\cdot)+G^b(t,\cdot)\|_{C_x}}{\delta_{q+1}}\right),
\end{align}
where $\varepsilon_m>0$ is the small constant in the second geometric lemma, Lemma \ref{lem-2nd-geometric}. Then, by \eqref{eq2.8}, \eqref{eq2.9}, \eqref{eq2.26}, \eqref{eq3.16}--\eqref{eq3.18}, \eqref{eq3.29} and standard H\"older estimates, for $1\leq N\leq4$,
\begin{align}
    &\left|\frac{R^m_l(t,x)+G^b(t,x)}{\varrho_m(t)}\right|\leq\varepsilon_m,\quad\varrho_m(t)\geq\varepsilon_m^{-1}\delta_{q+1},\quad\forall t\in[0,T],x\in\mathbb{T}^3,\label{eq3.31}\\   &\|\varrho_m\|_{L^p_t}\leq4\varepsilon_m^{-1}\left(T^{\frac{1}{p}}\delta_{q+1}+\|R^m_l+G^b\|_{L^p_tC_x}\right),\quad\forall p\in[1,+\infty],\label{eq3.32}\\
    &\|\varrho_m\|_{C_t}\lesssim l^{-2},\quad\|\varrho_m\|_{C^N_t}\lesssim l^{-5N-3},\label{eq3.33}\\
    &\|\varrho^{\frac{1}{2}}_m\|_{C_t}\lesssim l^{-1},\quad\|\varrho^{\frac{1}{2}}_m\|_{C^N_t}\lesssim l^{-5N-3},\label{eq3.34}\\
    &\|\varrho^{-1}_m\|_{C_t}\lesssim \varepsilon_m\delta^{-1}_{q+1},\quad\|\varrho_m^{-1}\|_{C^N_t}\lesssim l^{-5N-3},\label{eq3.35}
\end{align}
where the implicit constants are independent of $q$.

We choose the smooth temporal cut-off function $f_m(t)$ such that
\begin{itemize}
    \item $0\leq f_m\leq1\text{ and } f_m\equiv1\text{ on }{{\rm supp}_t(R^m_l)\cup{\rm supp}_t(G^b)}$;
     \item ${\rm supp}_t(f_m)\subseteq N_l({\rm supp}_t(R^m_l)\cup{\rm supp}_t(G^b))\subseteq N_{2l}({\rm supp}_t(R^m_l)\cup{\rm supp}_t(R^b_l))$;
     \item $\|f_m\|_{C^N_t}\lesssim l^{-N},1\leq N\leq5$.
\end{itemize}

The amplitudes of the momentum perturbations are defined by
\begin{align}\label{eq3.36}
    a_{(k)}(t,x):=f_m(t)\varrho^{\frac{1}{2}}_m(t)\rho^{\frac{1}{2}}_l(t,x)\gamma_{(k)}\left(\mathbb{I}-\frac{R^m_l(t,x)+G^b(t,x)}{\varrho_m(t)}\right),\quad k\in\Lambda_m,
\end{align}
where $\gamma_{(k)}$ and $\Lambda_m$ are given in Lemma \ref{lem-2nd-geometric}, and $\rho_l$ is the mollified density given by \eqref{eq2.16}.

Similar to \eqref{eq3.25}, we can obtain the following identity by Lemma \ref{lem-2nd-geometric} and \eqref{eq3.36},
\begin{equation}\label{eq3.37}
    \begin{aligned}
        \sum_{k\in\Lambda_m}\rho_l^{-1}a^2_{(k)}g^2_{(k)}W_{(k)}\otimes W_{(k)}=&\varrho_m(t)f_m^2(t)\mathbb{I}-R^m_l-G^b+\sum_{k\in\Lambda_m}\rho_l^{-1}a^2_{(k)}g^2_{(k)}\mathbb{P}_{\neq0}\left(W_{(k)}\otimes W_{(k)}\right)\\
        &+\sum_{k\in\Lambda_m}\rho_l^{-1}a^2_{(k)}\left(g^2_{(k)}-1\right)\fint_{\mathbb{T}^3}W_{(k)}\otimes W_{(k)}{\rm d}x.
    \end{aligned}
\end{equation}

Next, the analytic estimates of momentum amplitudes are collected in Lemma \ref{lem-momentum-amplitude-estimate}, whose proof is similar to \cite[Lemma 4.2]{LZZ-2022-jmpa}.
\begin{lem}\label{lem-momentum-amplitude-estimate}(Estimates of momentum amplitudes). 
    For any $1\leq N\leq4$, $k\in\Lambda_m$, we have
\begin{align}
    &\|a_{(k)}\|_{L^2_tC_x}\lesssim\delta_{q+1}^{\frac{1}{2}},\label{eq3.38}\\
    &\|a_{(k)}\|_{C_{t,x}}\lesssim l^{-2},\quad\|a_{(k)}\|_{C^N_{t,x}}\lesssim l^{-10N-5},\label{eq3.39}
\end{align}
where the implicit constants are independent of $q$.
\end{lem}
\subsubsection{Principal parts of perturbations}
The principal parts of the momentum and magnetic perturbations are defined by
\begin{align}    &w_{q+1}^{(p)}:=\sum_{k\in\Lambda_m\cup\Lambda_b}a_{(k)}g_{(k)}W_{(k)},\label{eq3.40}\\    &d_{q+1}^{(p)}:=\sum_{k\in\Lambda_b}a_{(k)}g_{(k)}D_{(k)}.\label{eq3.41}
\end{align}
Since the temporal supports of $g_{(k)}$ are pairwise disjoint, all cross-interaction terms vanish. By \eqref{eq3.13} and the algebraic identity \eqref{eq3.25},
\begin{equation}\label{eq3.42}
    \begin{aligned}
        &\quad\rho_l^{-1}\left(d_{q+1}^{(p)}\otimes w_{q+1}^{(p)}-w_{q+1}^{(p)}\otimes d_{q+1}^{(p)}\right)+R^b_l\\       &=\sum_{k\in\Lambda_b}\rho_l^{-1}a^2_{(k)}g^2_{(k)}\left(D_{(k)}\otimes W_{(k)}-W_{(k)}\otimes D_{(k)}\right)+R^b_l\\
        &=\sum_{k\in\Lambda_b}\rho_l^{-1}a^2_{(k)}g^2_{(k)}\mathbb{P}_{\neq0}\left(D_{(k)}\otimes W_{(k)}-W_{(k)}\otimes D_{(k)}\right)\\
  &\quad+\sum_{k\in\Lambda_b}\rho_l^{-1}a^2_{(k)}\partial_t\left(\sigma^{-1}h_{(k)}\right)\fint_{\mathbb{T}^3}\left(D_{(k)}\otimes W_{(k)}-W_{(k)}\otimes D_{(k)}\right){\rm d}x
    \end{aligned}
\end{equation}
Moreover, for the nonlinearity in the momentum equations $\eqref{eq2.1}_2$, by \eqref{eq3.13}, \eqref{eq3.28} and \eqref{eq3.37},
\begin{equation}\label{eq3.43}
    \begin{aligned}
        &\quad\rho_l^{-1}w_{q+1}^{(p)}\otimes w_{q+1}^{(p)}-\mu\left(d_{q+1}^{(p)}\otimes d_{q+1}^{(p)}-\frac{1}{2}\left|d_{q+1}^{(p)}\right|^2\mathbb{I}\right)+R^m_l\\
        &=\sum_{k\in\Lambda_m}\rho_l^{-1}a^2_{(k)}g^2_{(k)}W_{(k)}\otimes W_{(k)}+R^m_l\\
   &\quad+\sum_{k\in\Lambda_b}a^2_{(k)}g^2_{(k)}\left(\rho_l^{-1}W_{(k)}\otimes W_{(k)}-\mu\left( D_{(k)}\otimes D_{(k)}-\frac{1}{2}|D_{(k)}|^2\mathbb{I}\right)\right)\\
   &=\varrho_m(t)f_m^2(t)\mathbb{I}+\sum_{k\in\Lambda_m}\rho_l^{-1}a^2_{(k)}g^2_{(k)}\mathbb{P}_{\neq0}\left(W_{(k)}\otimes W_{(k)}\right)\\
        &\quad+\sum_{k\in\Lambda_m}\rho_l^{-1}a^2_{(k)}\partial_t\left(\sigma^{-1}h_{(k)}\right)\fint_{\mathbb{T}^3}W_{(k)}\otimes W_{(k)}{\rm d}x\\
        &\quad+\sum_{k\in\Lambda_b}a^2_{(k)}g^2_{(k)}\left(\rho_l^{-1}\mathbb{P}_{\neq0}\left(W_{(k)}\otimes W_{(k)}\right)-\mu\mathbb{P}_{\neq0}\left( D_{(k)}\otimes D_{(k)}-\frac{1}{2}|D_{(k)}|^2\mathbb{I}\right)\right)\\
        &\quad+\sum_{k\in\Lambda_b}a^2_{(k)}\partial_t\left(\sigma^{-1}h_{(k)}\right)\left(\rho_l^{-1}\fint_{\mathbb{T}^3}W_{(k)}\otimes W_{(k)}{\rm d}x-\mu\fint_{\mathbb{T}^3}D_{(k)}\otimes D_{(k)}-\frac{1}{2}|D_{(k)}|^2\mathbb{I}{\rm d}x\right).
    \end{aligned}
\end{equation}
\begin{rem}
 Note that, in view of \eqref{eq3.31}, $|\varrho_m|$ may not be less than $|R^m_l|$, up to a constant, the term $\varrho_m(t)f_m^2(t)\mathbb{I}$ is a large error and would ruin the inductive estimates \eqref{eq2.9} for the Reynolds stress at level $q+1$. In the incompressible case, this will not cause any difficulty due to its divergence can be absorbed into the pressure term $\nabla P$. However, this strategy fails in the compressible case, as the pressure depends on the density which is constrained by the transport equation $\eqref{eq2.1}_1$. Inspired by \cite{LQZZ-2022-arxiv,DL-2025-JLMS}, we use the time-dependence of both $\varrho_m$ and $\varrho_b$ to eliminate this trouble term by taking the spatial divergence operator.
\end{rem}
\subsubsection{Incompressible correctors}
Since the amplitude functions $\left\{a_{(k)}\right\}_{k\in\Lambda_m\cup\Lambda_b}$ depend on spatial variables, the principal parts of the perturbations are not divergence-free in general. This leads to the definition of the incompressible correctors
\begin{align}
    &w_{q+1}^{(c)}:=\sum_{k\in\Lambda_m\cup\Lambda_b}g_{(k)}\left({\rm curl} \left(\nabla a_{(k)}\times W^c_{(k)}\right)+\nabla a_{(k)}\times{\rm curl}W^c_{(k)}\right),\label{eq3.44}\\
    &d_{q+1}^{(c)}:=\sum_{k\in\Lambda_b}g_{(k)}\left({\rm curl} \left(\nabla a_{(k)}\times D^c_{(k)}\right)+\nabla a_{(k)}\times{\rm curl}D^c_{(k)}\right),\label{eq3.45}
\end{align}
where $W^c_{(k)}$ and $D^c_{(k)}$ are defined by
\begin{align}\label{eq3.46}
    W^c_{(k)}:=\frac{1}{N_{\Lambda}^2\lambda^2}\Phi_{(k)}k_1,\quad D^c_{(k)}:=\frac{1}{N_{\Lambda}^2\lambda^2}\Phi_{(k)}k_2.
\end{align}
Then, it follows that,
\begin{align}
    &w_{q+1}^{(p)}+w_{q+1}^{(c)}={\rm curl}{\rm curl}\left(\sum_{k\in\Lambda_m\cup\Lambda_b}a_{(k)}g_{(k)}W^c_{(k)}\right),\label{eq3.47}\\
    &d_{q+1}^{(p)}+d_{q+1}^{(c)}={\rm curl}{\rm curl}\left(\sum_{k\in\Lambda_b}a_{(k)}g_{(k)}D^c_{(k)}\right).\label{eq3.48}
\end{align}
In particular,
\begin{align}\label{eq3.49}
    {\rm div}\left(w_{q+1}^{(p)}+w_{q+1}^{(c)}\right)={\rm div}\left(d_{q+1}^{(p)}+d_{q+1}^{(c)}\right)=0.
\end{align}
\begin{rem}
    We have noticed that this type of incompressible corrector is frequently employed in convex integration schemes for incompressible fluid, see for instance \cite{BV-2019-emssms,BBV-2020-annpde}. But for compressible MHD equations, although the magnetic field requires the incompressible correctors, it seems not necessary to keep the momentum perturbations incompressible. However, inspired by \cite{LQZZ-2022-arxiv}, we still use this corrector to keep part of momentum perturbations incompressible, which reduces the error caused by density in subsequent convex integration scheme.
\end{rem}
\subsubsection{Temporal correctors}
The last type of corrector consists of the temporal correctors, in order to balance the high temporal oscillations in \eqref{eq3.42} and \eqref{eq3.43} caused by the temporal concentration functions $g_{(k)}$. Drawing inspiration from \cite{LZZ-2022-jmpa,LQZZ-2022-arxiv,NY-2025-jns}, the temporal correctors $w_{(q+1)}^{(o)}$ and $d_{(q+1)}^{(o)}$ are defined by
\begin{align}
   & \begin{aligned}
        w_{(q+1)}^{(o)}:=&-\sigma^{-1}\sum_{k\in\Lambda_m\cup\Lambda_b}h_{(k)}\fint_{\mathbb{T}^3}W_{(k)}\otimes W_{(k)}{\rm d}x\nabla \left(\rho_l^{-1}a^2_{(k)}\right)\\
        &+\sigma^{-1}\sum_{k\in\Lambda_b}h_{(k)}\mu\fint_{\mathbb{T}^3}D_{(k)}\otimes D_{(k)}-\frac{1}{2}\left|D_{(k)}\right|^2\mathbb{I}{\rm d}x\nabla \left(a^2_{(k)}\right),
    \end{aligned}\label{eq3.50}\\
    &d_{(q+1)}^{(o)}:=-\sigma^{-1}\sum_{k\in\Lambda_b}h_{(k)}{\rm div}\left(\rho^{-1}_la^2_{(k)}\fint_{\mathbb{T}^3}D_{(k)}\otimes W_{(k)}-W_{(k)}\otimes D_{(k)}{\rm d}x\right),\label{eq3.51}
\end{align}
where $h_{(k)}$ is given by \eqref{eq3.12}.
\begin{rem}
    Since $\mathbb{P}_{=0}\left(D_{(k)}\otimes W_{(k)}-W_{(k)}\otimes D_{(k)}\right)=k_2\otimes k_1-k_1\otimes k_2$ is skew-symmetric, $d_{(q+1)}^{(o)}$ is divergence-free as well.
\end{rem}

Then, by the Leibniz rule,
\begin{equation}\label{eq3.52}
    \begin{aligned}     &\quad\partial_tw_{q+1}^{(o)}+\sum_{k\in\Lambda_m}\partial_t\left(\sigma^{-1}h_{(k)}\right)\fint_{\mathbb{T}^3}W_{(k)}\otimes W_{(k)}{\rm d}x\nabla\left(\rho_l^{-1}a^2_{(k)}\right)\\        &\quad+\sum_{k\in\Lambda_b}\partial_t\left(\sigma^{-1}h_{(k)}\right)\fint_{\mathbb{T}^3}W_{(k)}\otimes W_{(k)}{\rm d}x\nabla\left(\rho_l^{-1}a^2_{(k)}\right)\\
    &\quad-\mu\sum_{k\in\Lambda_b}\partial_t\left(\sigma^{-1}h_{(k)}\right)\fint_{\mathbb{T}^3}D_{(k)}\otimes D_{(k)}-\frac{1}{2}|D_{(k)}|^2\mathbb{I}{\rm d}x\nabla\left(a^2_{(k)}\right)\\
    &=-\sigma^{-1}\sum_{k\in\Lambda_m\cup\Lambda_b}h_{(k)}\left(k_1\cdot\nabla\right)\partial_t\left(\rho_l^{-1}a^2_{(k)}\right)k_1\\
    &\quad+\sigma^{-1}\sum_{k\in\Lambda_b}\mu h_{(k)}\left(\left(k_2\cdot\nabla\right)\partial_t\left(a^2_{(k)}\right)k_2-\frac{1}{2}\partial_t\nabla\left(a^2_{(k)}\right)\right),
    \end{aligned}
\end{equation}
and
\begin{equation}\label{eq3.53}
    \begin{aligned}
&\quad\partial_td_{q+1}^{(o)}+\sum_{k\in\Lambda_b}\partial_t\left(\sigma^{-1}h_{(k)}\right)\fint_{\mathbb{T}^3}\left(D_{(k)}\otimes W_{(k)}-W_{(k)}\otimes D_{(k)}\right){\rm d}x\nabla\left(\rho_l^{-1}a^2_{(k)}\right)\\
&=-\sigma^{-1}\sum_{k\in\Lambda_b}h_{(k)}\left(k_2\otimes k_1-k_1\otimes k_2\right)\partial_t\nabla\left(\rho_l^{-1}a^2_{(k)}\right).
    \end{aligned}
\end{equation}

Now, we define the magnetic and momentum perturbations at level $q+1$ by
\begin{align}
    &w_{q+1}:=w_{q+1}^{(p)}+w_{q+1}^{(c)}+w_{q+1}^{(o)},\label{eq3.54}\\
    &d_{q+1}:=d_{q+1}^{(p)}+d_{q+1}^{(c)}+d_{q+1}^{(o)}.\label{eq3.55}
\end{align}
Note that, by the constructions above, $w_{q+1}$ is mean-free but not divergence-free, and, $d_{q+1}$ is both mean-free and divergence-free.

The new momentum and magnetic fields $m_{q+1}$ and $b_{q+1}$ at $q+1$ level is then defined by
\begin{align}\label{eq3.56}
    m_{q+1}:=m_l+w_{q+1},\quad b_{q+1}:=b_l+d_{q+1},
\end{align}
where $m_l$ and $b_l$ are the previous mollified momentum and magnetic fields defined in \eqref{eq2.16}.
\subsection{Density perturbation}
Because the momentum perturbation $w_{q+1}$ may not be divergence-free, we need to construct density perturbation under the constraints of the transport equation $\eqref{eq2.1}_1$ at level $q+1$.

More precisely, we defin the density perturbation $z_{q+1}$ as
\begin{equation}\label{eq3.57}
    \begin{aligned}
        z_{q+1}(t,x)&:=-\int_0^t{\rm div}w_{q+1}(s,x){\rm d}s\\
        &=\sigma^{-1}\sum_{k\in\Lambda_m\cup\Lambda_b}\int_0^th_{(k)}{\rm div}\left(k_1\otimes k_1\nabla \left(\rho_l^{-1}a^2_{(k)}\right)\right)(s,x){\rm d}s\\
        &\quad-\sigma^{-1}\sum_{k\in\Lambda_b}\mu\int_0^th_{(k)}{\rm div}\left(\left(k_2\otimes k_2-\frac{1}{2}|k_2|^2\mathbb{I}\right)\nabla \left(a^2_{(k)}\right)\right)(s,x){\rm d}s.
    \end{aligned}
\end{equation}
Then, the new density at level $q+1$ is defined by
\begin{align}\label{eq3.58}
    \rho_{q+1}:=\rho_l+z_{q+1}.
\end{align}
Furthermore, we have 
\begin{align}\label{eq3.59}
    \partial_t\rho_{q+1}+{\rm div}m_{q+1}=\partial_tz_{q+1}+{\rm div}w_{q+1}=0,
\end{align}
which coincides with the transport equation $\eqref{eq2.1}_1$ at level $q+1$.
\subsection{Estimates of perturbations}
In this section, we summarize the crucial estimates of the perturbations in Proposition \ref{prop-estimate-perturbation} below.
\begin{pro}\label{prop-estimate-perturbation}
    For any $p\in[1,+\infty]$ and integer $1\leq N\leq4$, the following estimates will hold,
    \begin{align}
        &\|\nabla^Nw_{q+1}^{(p)}\|_{L^p_tC_x}+\|\nabla^Nd_{q+1}^{(p)}\|_{L^p_tC_x}\lesssim l^{-5}\lambda^N\tau^{\frac{1}{2}-\frac{1}{p}},\label{eq3.60}\\
        &\|\nabla^Nw_{q+1}^{(c)}\|_{L^p_tC_x}+\|\nabla^Nd_{q+1}^{(c)}\|_{L^p_tC_x}\lesssim l^{-15}\lambda^{N-1}\tau^{\frac{1}{2}-\frac{1}{p}},\label{eq3.61}\\
        &\|\nabla^Nw_{q+1}^{(o)}\|_{L^p_tC_x}+\|\nabla^Nd_{q+1}^{(o)}\|_{L^p_tC_x}\lesssim l^{-10N-18}\sigma^{-1},\label{eq3.62}\\
        &\|w_{q+1}\|_{C^N_{t,x}}+\|d_{q+1}\|_{C^N_{t,x}}\lesssim\lambda^{2N+2}.\label{eq3.63}
    \end{align}
    Moreover, for integer $0\leq N\leq4$ and $0\leq M\leq1$, it holds that,
    \begin{align}
        \|\partial_t^Mz_{q+1}\|_{C_tC^N_x}\lesssim l^{-10N-28}\sigma^{-1},\label{eq3.64}
    \end{align}
    where the implicit constants are independent of $q$.
\end{pro}
\begin{pf}
    To begin with, using \eqref{eq3.13}, \eqref{eq3.14}, \eqref{eq3.40}, \eqref{eq3.41} and  Lemmas \ref{lem-mikado-flows}, \ref{lem-magnetic-amplitude-estimate} and \ref{lem-momentum-amplitude-estimate}, it leads for any $p\in[1,+\infty]$,
    \begin{align*}
&\quad\|\nabla^Nw_{q+1}^{(p)}\|_{L^p_tC_x}+\|\nabla^Nd_{q+1}^{(p)}\|_{L^p_tC_x}\\
&\lesssim\sum_{k\in\Lambda_m\cup\Lambda_b}\sum_{N_1+N_2=N}\|a_{(k)}\|_{C^{N_1}_x}\|g_{(k)}\|_{L^p_t}\|\nabla^{N_2}W_{(k)}\|_{C_x}\\
&\quad+\sum_{k\in\Lambda_b}\sum_{N_1+N_2=N}\|a_{(k)}\|_{C^{N_1}_x}\|g_{(k)}\|_{L^p_t}\|\nabla^{N_2}D_{(k)}\|_{C_x}\\
&\lesssim\sum_{N_1+N_2=N}l^{-5N_1-1}\tau^{\frac{1}{2}-\frac{1}{p}}\lambda^{N_2}+l^{-10N_1-5}\tau^{\frac{1}{2}-\frac{1}{p}}\lambda^{N_2}\\
&\lesssim l^{-5}\lambda^N\tau^{\frac{1}{2}-\frac{1}{p}},
    \end{align*}
where the last inequality is due to $l^{-10}\ll \lambda$, that is, \eqref{eq3.60} is valid.

Next, by \eqref{eq3.14}, \eqref{eq3.44}, \eqref{eq3.45} and \eqref{eq3.46} and Lemmas \ref{lem-mikado-flows}, \ref{lem-magnetic-amplitude-estimate} and \ref{lem-momentum-amplitude-estimate},
  \begin{align*}
&\quad\|\nabla^Nw_{q+1}^{(c)}\|_{L^p_tC_x}+\|\nabla^Nd_{q+1}^{(c)}\|_{L^p_tC_x}\\
&\lesssim\sum_{k\in\Lambda_m\cup\Lambda_b}\|g_{(k)}\|_{L^p_t}\sum_{N_1+N_2=N}\left(\|a_{(k)}\|_{C^{N_1+2}_x}\|W^c_{(k)}\|_{C^{N_2}_x}+\|a_{(k)}\|_{C^{N_1+1}_x}\|W^c_{(k)}\|_{C^{N_2+1}_x}\right)\\
&\quad+\sum_{k\in\Lambda_b}\|g_{(k)}\|_{L^p_t}\sum_{N_1+N_2=N}\left(\|a_{(k)}\|_{C^{N_1+2}_x}\|D^c_{(k)}\|_{C^{N_2}_x}+\|a_{(k)}\|_{C^{N_1+1}_x}\|D^c_{(k)}\|_{C^{N_2+1}_x}\right)\\
&\lesssim\sum_{N_1+N_2=N}\tau^{\frac{1}{2}-\frac{1}{p}}\left(l^{-5N_1-11}\lambda^{N_2-2}+l^{-5N_1-6}\lambda^{N_2-1}+l^{-10N_1-25}\lambda^{N_2-2}+l^{-10N_1-15}\lambda^{N_2-1}\right)\\
&\lesssim l^{-15}\lambda^{N-1}\tau^{\frac{1}{2}-\frac{1}{p}},
    \end{align*}
    which leads \eqref{eq3.61}.

    Moreover, we can obtain the estimates of temporal correctors by combining \eqref{eq2.21}, \eqref{eq2.22}, \eqref{eq3.15}, \eqref{eq3.50} and \eqref{eq3.51} and Lemmas \ref{lem-mikado-flows}, \ref{lem-magnetic-amplitude-estimate} and \ref{lem-momentum-amplitude-estimate},
    \begin{align*}
        &\quad\|\nabla^Nw_{q+1}^{(o)}\|_{L^p_tC_x}+\|\nabla^Nd_{q+1}^{(o)}\|_{L^p_tC_x}\\
        &\lesssim\sigma^{-1}\sum_{k\in\Lambda_m\cup\Lambda_b}\|h_{(k)}\|_{C_t}\left\|\nabla^{N+1}\left(\rho^{-1}_la^2_{(k)}\right)\right\|_{C_{t,x}}+\sigma^{-1}\sum_{k\in\Lambda_b}\|h_{(k)}\|_{C_t}\left\|\nabla^{N+1}\left(a^2_{(k)}\right)\right\|_{C_{t,x}}\\
        &\lesssim\sigma^{-1}\sum_{k\in\Lambda_m\cup\Lambda_b}\|h_{(k)}\|_{C_t}\sum_{N_1+N_2=N+1}\|\rho_l^{-1}\|_{C^{N_1}_{t,x}}\|a^2_{(k)}\|_{C^{N_2}_{t,x}}+\sigma^{-1}\sum_{k\in\Lambda_b}\|h_{(k)}\|_{C_t}\|a^2_{(k)}\|_{C^{N+1}_{t,x}}\\
        &\lesssim\sigma^{-1}\left(l^{-10N-17}\lambda_q^{\frac{(N+1)\varepsilon}{4}}+l^{-5N-8}\right)\\
        &\lesssim l^{-10N-18}\sigma^{-1},
    \end{align*}
    which the last step goes under \eqref{eq2.3s} and \eqref{eq2.2s}.

    In order to establish the $C^N_{t,x}$-estimates of momentum and magnetic perturbations, we use Lemmas \ref{lem-mikado-flows}--\ref{lem-momentum-amplitude-estimate},
    \begin{equation}\label{eq3.65}
        \begin{aligned}
            &\quad\|w_{q+1}^{(p)}\|_{C^N_{t,x}}+\|d_{q+1}^{(p)}\|_{C^N_{t,x}}\\
            &\lesssim\sum_{k\in\Lambda_m\cup\Lambda_b}\|a_{(k)}\|_{C^N_{t,x}}\sum_{0\leq N_1+N_2\leq N}\|g_{(k)}\|_{C^{N_1}_t}\|W_{(k)}\|_{C^{N_2}_{x}}\\
            &\quad+\sum_{k\in\Lambda_b}\|a_{(k)}\|_{C^N_{t,x}}\sum_{0\leq N_1+N_2\leq N}\|g_{(k)}\|_{C^{N_1}_t}\|D_{(k)}\|_{C^{N_2}_{x}}\\
            &\lesssim\sum_{0\leq N_1+N_2\leq N}\sigma^{N_1}\tau^{\frac{1}{2}+N_1}\lambda^{N_2}\left(l^{-10N-5}+l^{-5N-1}+l^{-2}\right)\\
            &\lesssim \lambda^{2N+2},
        \end{aligned}
    \end{equation}
    where we also use \eqref{eq2.3s}, \eqref{eq2.2s}, \eqref{eq2.17} and \eqref{eq3.1}.

    Similarly, we obtain
    \begin{equation}\label{eq3.66}
        \begin{aligned}
            &\quad\|w_{q+1}^{(c)}\|_{C^N_{t,x}}+\|d_{q+1}^{(c)}\|_{C^N_{t,x}}\\
            &\lesssim\sum_{k\in\Lambda_m\cup\Lambda_b}\|a_{(k)}\|_{C^{N+2}_{t,x}}\sum_{0\leq N_1+N_2\leq N}\|g_{(k)}\|_{C^{N_1}_t}\|W^c_{(k)}\|_{C^{N_2+1}_{x}}\\
            &\quad+\sum_{k\in\Lambda_b}\|a_{(k)}\|_{C^{N+2}_{t,x}}\sum_{0\leq N_1+N_2\leq N}\|g_{(k)}\|_{C^{N_1}_t}\|D^c_{(k)}\|_{C^{N_2+1}_{x}}\\
            &\lesssim\sum_{0\leq N_1+N_2\leq N}\sigma^{N_1}\tau^{\frac{1}{2}+N_1}\lambda^{N_2-1}\left(l^{-10N-25}+l^{-5N-11}+l^{-2}\right)\\
            &\lesssim \lambda^{2N+1},
        \end{aligned}
    \end{equation}
    and
     \begin{equation}\label{eq3.67}
        \begin{aligned}
            &\quad\|w_{q+1}^{(o)}\|_{C^N_{t,x}}+\|d_{q+1}^{(o)}\|_{C^N_{t,x}}\\
            &\lesssim\sigma^{-1}\sum_{k\in\Lambda_m\cup\Lambda_b}\left\|h_{(k)}\nabla^{N+1}\left(\rho^{-1}_la^2_{(k)}\right)\right\|_{C^N_{t,x}}\\
            &\quad+\sigma^{-1}\sum_{k\in\Lambda_b}\left\|h_{(k)}\nabla^{N+1}\left(a^2_{(k)}\right)\right\|_{C^N_{t,x}}\\
            &\lesssim\sigma^{N-1}\tau^N\left(l^{-10N-19}+l^{-5N-9}\right)\\
            &\lesssim\lambda^{2N+1},
        \end{aligned}
    \end{equation}
    where the last step is due to \eqref{eq2.3s}, \eqref{eq2.2s}, \eqref{eq2.17}, \eqref{eq3.1} and the fact that
    \begin{align*}
        \|h_{(k)}\|_{C^N_t}\lesssim\sigma^N\tau^N.
    \end{align*}
Thus, \eqref{eq3.63} can be obtained by combining \eqref{eq3.65}--\eqref{eq3.67}.

It remains to check \eqref{eq3.64}. This can be done by using \eqref{eq3.13} and \eqref{eq3.57} and Lemmas \ref{lem-mikado-flows}--\ref{lem-momentum-amplitude-estimate}:
\begin{align*}
    &\quad\|\partial_t^Mz_{q+1}\|_{C_tC^N_x}\\
    &\lesssim\sigma^{-1}\sum_{k\in\Lambda_m\cup\Lambda_b}\left\|\partial_t^M\int_0^th_{(k)}{\rm div}\left(k_1\otimes k_1\nabla \left(\rho_l^{-1}a^2_{(k)}\right)\right)(s,x){\rm d}s\right\|_{C_tC^N_x}\\
    &\quad+\sigma^{-1}\sum_{k\in\Lambda_b}\mu\left\|\partial^M_t\int_0^th_{(k)}{\rm div}\left(\left(k_2\otimes k_2-\frac{1}{2}|k_2|^2\mathbb{I}\right)\nabla \left(a^2_{(k)}\right)\right)(s,x){\rm d}s\right\|_{C_tC^N_x}\\
    &\lesssim\sigma^{-1}\left(\sum_{k\in\Lambda_m\cup\Lambda_b}\|h_{(k)}\|_{C_t}\|\rho_l^{-1}a^2_{(k)}\|_{C_tC^{N+2}_x}+\sum_{k\in\Lambda_b}\|h_{(k)}\|_{C_t}\|a^2_{(k)}\|_{C_tC^{N+2}_x}\right)\\
    &\lesssim\sigma^{-1}\sum_{k\in\Lambda_m\cup\Lambda_b}\left(\|\rho_l^{-1}\|_{C_tC_x^{N+2}}\|a^2_{(k)}\|_{C_{t,x}}+\sum_{1\leq N^{\prime}\leq N+2}\|\rho_l^{-1}\|_{C_tC_x^{N+2-N^{\prime}}}\|a^2_{(k)}\|_{C_tC_x^{N^{\prime}}}\right)\\
    &\quad+\sigma^{-1}\sum_{k\in\Lambda_b}\|a^2_{(k)}\|_{C^{N+2}_{t,x}}\\
    &\lesssim\sigma^{-1}\left(\lambda_q^{\frac{(N+2)\varepsilon}{4}}l^{-4}+l^{-5N-13}+\sum_{1\leq N^{\prime}\leq N+2}\lambda_q^{\frac{(N+2-N^{\prime})\varepsilon}{4}}l^{-10N^{\prime}-7}\right)\\
    &\lesssim l^{-10N-28}\sigma^{-1}.
\end{align*}
Therefore, the proof of Proposition \ref{prop-estimate-perturbation} has been finished.
    \hfill$\square$
\end{pf}
\subsection{Verification of inductive estimates}
We are now in the stage to verify the inductive estimates \eqref{eq2.5}--\eqref{eq2.7} and \eqref{eq2.10}--\eqref{eq2.13} for the perturbations.

Our first concern is the density function $\rho_{q+1}$. By \eqref{eq2.17}, \eqref{eq2.21}, \eqref{eq3.1}, \eqref{eq3.58} and \eqref{eq3.64},
\begin{align}\label{eq3.68}
    \rho_{q+1}\leq\rho_l+\|z_{q+1}\|_{C_{t,x}}\leq C_2-\frac{1}{2}\lambda_q^{-\beta}+l^{-28}\sigma^{-1}\leq C_2-\frac{1}{2}\lambda_q^{-\beta}+\lambda_{q+1}^{-14\varepsilon}\leq C_2-\lambda_{q+1}^{-\beta},
\end{align}
where we used $\lambda_{q+1}^{-14\varepsilon}+\lambda_{q+1}^{-\beta}\ll\frac{1}{2}\lambda_q^{-\beta}$ in the last inequality, because of \eqref{eq2.3s}. Similarly, we also obtain
\begin{align}\label{eq3.69}
    \rho_{q+1}\geq\rho_l-\|z_{q+1}\|_{C_{t,x}}\geq C_1+\frac{1}{2}\lambda_q^{-\beta}-l^{-28}\sigma^{-1}\geq C_1+\frac{1}{2}\lambda_q^{-\beta}-\lambda_{q+1}^{-14\varepsilon}\geq C_1+\lambda_{q+1}^{-\beta}.
\end{align}
Then \eqref{eq2.5} is valid at level $q+1$ by combining \eqref{eq3.68} and \eqref{eq3.69}.

Next, by \eqref{eq3.57} and \eqref{eq3.58},
\begin{align}\label{eq3.70}
    \int_{\mathbb{T}^3}\rho_{q+1}(t,x){\rm d}x=\int_{\mathbb{T}^3}\rho_l(t,x){\rm d}x+\int_{\mathbb{T}^3}z_{q+1}(t,x){\rm d}x=\int_{\mathbb{T}^3}\rho_{q}(t,x){\rm d}x,
\end{align}
for any $t\in[0,T]$, which gives that \eqref{eq2.10} at level $q+1$.

Moreover, for any $0\leq N\leq4$, $0\leq M\leq1$, we obtain 
\begin{equation}\label{eq3.71}
    \begin{aligned}
        \|\partial_t^M\rho_{q+1}\|_{C_tC^N_x}
        &\lesssim\|\partial_t^M\rho_{l}\|_{C_tC^N_x}+\|\partial_t^Mz_{q+1}\|_{C_tC^N_x}\\
        &\lesssim\lambda_q^{\frac{N\varepsilon}{4}}+l^{-10N-28}\sigma^{-1}\\
        &\lesssim \lambda_{q+1}^{\frac{N\varepsilon}{4}},
    \end{aligned}
\end{equation}
by using \eqref{eq2.22} and \eqref{eq3.64} which leads \eqref{eq2.6} at level $q+1$.

We also derive from \eqref{eq2.23} and \eqref{eq3.64} that
\begin{equation}\label{eq3.72}
    \begin{aligned}
        \|\rho_{q+1}-\rho_q\|_{C_tC^1_x}
        &\lesssim\|\rho_{l}-\rho_q\|_{C_tC^1_x}+\|z_{q+1}\|_{C_tC^1_x}\\
        &\lesssim l\lambda_q^{\frac{\varepsilon}{2}}+l^{-38}\sigma^{-1}\\
        &\lesssim\delta^{\frac{1}{2}}_{q+2},
    \end{aligned}
\end{equation}
where the last step goes under \eqref{eq2.3s}, \eqref{eq2.4}, \eqref{eq2.17} and \eqref{eq3.1}, that is, it yields \eqref{eq2.11}.

Then, we consider the momentum and magnetic perturbations. To begin with, by virtue of \eqref{eq2.24} and \eqref{eq3.63}, we have
\begin{equation}\label{eq3.73}
    \begin{aligned}
        \|m_{q+1}\|_{C^N_{t,x}}+\|b_{q+1}\|_{C^N_{t,x}}
        &\lesssim \|m_{l}\|_{C^N_{t,x}}+ \|w_{q+1}\|_{C^N_{t,x}}+\|b_{l}\|_{C^N_{t,x}}+ \|d_{q+1}\|_{C^N_{t,x}}\\
        &\lesssim l^{-N+1}\lambda^{4}_q+\lambda_{q+1}^{2N+2}\\
        &\lesssim\lambda_{q+1}^{2N+2},
    \end{aligned}
\end{equation}
which leads \eqref{eq2.7} at level $q+1$.

Note that, the estimates in \eqref{eq3.60} alone do not yield the decay required by \eqref{eq2.12}. The technique here is the key $L^p$-decorrelation, which permits to derive the $L^2_tC_x$-decay of the principal parts $w^{(p)}_{q+1}$ and $d^{(p)}_{q+1}$.
\begin{lem}\label{lem-lp-decorrelation}($L^p$-decorrelation, \cite[Lemma 3.7]{BV-2019-annmath}, \cite[Lemma 2.4]{CL-2021-annpde}). Let $\sigma\in\mathbb{N}$ and $f,g:\mathbb{T}^N\rightarrow\mathbb{R}$ be smooth functions. Then for any $p\in[1,+\infty]$,
\begin{align}\label{eq3.74}
    \left|\|fg(\sigma\cdot)\|_{L^p(\mathbb{T}^N)}-\|f\|_{L^p(\mathbb{T}^N)}\|g\|_{L^p(\mathbb{T}^N)}\right|\lesssim\sigma^{-\frac{1}{p}}\|f\|_{C^1(\mathbb{T}^N)}\|g\|_{L^p(\mathbb{T}^N)}.
\end{align}    
\end{lem}
We apply Lemma \ref{lem-lp-decorrelation} with $N=1$, $f=\|a_{(k)}(t,\cdot)\|_{L^{\infty}}$, $g=g_{(k)}$ and $\sigma=\lambda_{q+1}^{15\varepsilon}$, furthermore, by  \eqref{eq2.3s}, \eqref{eq2.2s} and  Lemmas \ref{lem-mikado-flows}--\ref{lem-momentum-amplitude-estimate}, we obtain
\begin{align}\label{eq3.75}
    \begin{aligned}
        \|w^{(p)}_{q+1}\|_{L^2_tC_x}+\|d^{(p)}_{q+1}\|_{L^2_tC_x}      &\lesssim\sum_{k\in\Lambda_m\cup\Lambda_b}\left\|\|a_{(k)}\|_{C_x}g_{(k)}\right\|_{L^2_t}\|W_{(k)}\|_{C_x}\\
        &\quad+\sum_{k\in\Lambda_b}\left\|\|a_{(k)}\|_{C_x}g_{(k)}\right\|_{L^2_t}\|D_{(k)}\|_{C_x}\\      &\lesssim\sum_{k\in\Lambda_m\cup\Lambda_b}\left(\|a_{(k)}\|_{L^2_tC_x}\|g_{(k)}\|_{L^2_t}+\sigma^{-\frac{1}{2}}\|a_{(k)}\|_{C^1_{t,x}}\|g_{(k)}\|_{L^2_t}\right)\|\phi_{(k)}\|_{C_x}\\
        &\lesssim\delta_{q+1}^{\frac{1}{2}}+\sigma^{-\frac{1}{2}}\left(l^{-15}+l^{-6}\right)\\
        &\lesssim\delta_{q+1}^{\frac{1}{2}}.
    \end{aligned}
\end{align}
Then, taking into account \eqref{eq2.3s}, \eqref{eq3.61}, \eqref{eq3.62} and \eqref{eq3.75}, we deduce
\begin{equation}\label{eq3.76}
    \begin{aligned}
       \|w_{q+1}\|_{L^2_tC_x}+\|d_{q+1}\|_{L^2_tC_x}
       &\lesssim\|w^{(p)}_{q+1}\|_{L^2_tC_x}+\|w^{(c)}_{q+1}\|_{L^2_tC_x}+\|w^{(o)}_{q+1}\|_{L^2_tC_x}\\
       &\quad+\|d^{(p)}_{q+1}\|_{L^2_tC_x}+\|d^{(c)}_{q+1}\|_{L^2_tC_x}+\|d^{(o)}_{q+1}\|_{L^2_tC_x}\\
       &\lesssim \delta_{q+1}^{\frac{1}{2}}+l^{-15}\lambda_{q+1}^{-1}+l^{-18}\sigma^{-1}\\
       &\lesssim\delta_{q+1}^{\frac{1}{2}}.
    \end{aligned}
\end{equation}
Moreover, using \eqref{eq2.3s} and Proposition \ref{prop-estimate-perturbation} with $p=1$ and $N=0$ yields that
\begin{equation}\label{eq3.77}
    \begin{aligned}
       \|w_{q+1}\|_{L^1_tC_x}+\|d_{q+1}\|_{L^1_tC_x}
       &\lesssim\|w^{(p)}_{q+1}\|_{L^1_tC_x}+\|w^{(c)}_{q+1}\|_{L^1_tC_x}+\|w^{(o)}_{q+1}\|_{L^1_tC_x}\\
       &\quad+\|d^{(p)}_{q+1}\|_{L^1_tC_x}+\|d^{(c)}_{q+1}\|_{L^1_tC_x}+\|d^{(o)}_{q+1}\|_{L^1_tC_x}\\
       &\lesssim l^{-5}\tau^{-\frac{1}{2}}+l^{-15}\lambda_{q+1}^{-1}\tau^{-\frac{1}{2}}+l^{-18}\sigma^{-1}\\
       &\lesssim\delta_{q+2}^{\frac{1}{2}}.
    \end{aligned}
\end{equation}
Furthermore, by combining with \eqref{eq2.25}, \eqref{eq3.76} and \eqref{eq3.77}, we have
\begin{equation}\label{eq3.78}
    \begin{aligned}
       &\quad\|m_{q+1}-m_q\|_{L^2_tC_x}+\|b_{q+1}-b_q\|_{L^2_tC_x}\\
       &\lesssim \|m_{q+1}-m_l\|_{L^2_tC_x}+\|m_{l}-m_q\|_{L^2_tC_x}+\|b_{q+1}-b_l\|_{L^2_tC_x}+\|b_{l}-b_q\|_{L^2_tC_x}\\
       &\lesssim \|m_{l}-m_q\|_{C_{t,x}}+ \|b_{l}-b_q\|_{C_{t,x}}+\|w_{q+1}\|_{L^2_tC_x}+\|d_{q+1}\|_{L^2_tC_x}\\
       &\lesssim l\lambda^4_q+\delta_{q+1}^{\frac{1}{2}}\\
       &\lesssim\delta_{q+1}^{\frac{1}{2}}.
    \end{aligned}
\end{equation}
and
\begin{equation}\label{eq3.79}
    \begin{aligned}
       &\quad\|m_{q+1}-m_q\|_{L^1_tC_x}+\|b_{q+1}-b_q\|_{L^1_tC_x}\\
       &\lesssim \|m_{q+1}-m_l\|_{L^1_tC_x}+\|m_{l}-m_q\|_{L^1_tC_x}+\|b_{q+1}-b_l\|_{L^1_tC_x}+\|b_{l}-b_q\|_{L^1_tC_x}\\
       &\lesssim \|m_{l}-m_q\|_{C_{t,x}}+ \|b_{l}-b_q\|_{C_{t,x}}+\|w_{q+1}\|_{L^1_tC_x}+\|d_{q+1}\|_{L^1_tC_x}\\
       &\lesssim l\lambda^4_q+\delta_{q+2}^{\frac{1}{2}}\\
       &\lesssim\delta_{q+2}^{\frac{1}{2}}.
    \end{aligned}
\end{equation}

Regarding the $L^p_tC_x^s$-decay of the momentum and magnetic perturbations with $(p,s)$ satisfying \eqref{eq1.7s}, using interpolation and Proposition \ref{prop-estimate-perturbation}, we get
\begin{equation}\label{eq3.80}
    \begin{aligned}
        &\quad\|w^{(p)}_{q+1}\|_{L^p_tC_x^s}+\|d^{(p)}_{q+1}\|_{L^p_tC_x^s}\\
        &\lesssim\|w^{(p)}_{q+1}\|_{L^p_tC_x}^{1-s}\|w^{(p)}_{q+1}\|_{L^p_tC_x^1}^s+\|d^{(p)}_{q+1}\|_{L^p_tC_x}^{1-s}\|d^{(p)}_{q+1}\|_{L^p_tC_x^1}^s\\
        &\lesssim\left(l^{-5}\tau^{\frac{1}{2}-\frac{1}{p}}\right)^{1-s}\left(l^{-5}\lambda_{q+1}\tau^{\frac{1}{2}-\frac{1}{p}}\right)^s\\
        &\lesssim l^{-5}\lambda_{q+1}^s\tau^{\frac{1}{2}-\frac{1}{p}}\\
        &\lesssim\lambda_{q+1}^{\alpha-\frac{2\alpha}{p}+s+\varepsilon(\frac{10}{p}-4)},
    \end{aligned}
\end{equation}
and, similarly,
\begin{align}    &\|w^{(c)}_{q+1}\|_{L^p_tC_x^s}+\|d^{(c)}_{q+1}\|_{L^p_tC_x^s}\lesssim l^{-55}\lambda_{q+1}^{s-1}\tau^{\frac{1}{2}-\frac{1}{p}}\lesssim\lambda_{q+1}^{\alpha-1-\frac{2\alpha}{p}+s+\varepsilon(\frac{10}{p}-4)},\label{eq3.81}\\
    &\|w^{(o)}_{q+1}\|_{L^p_tC_x^s}+\|d^{(o)}_{q+1}\|_{L^p_tC_x^s}\lesssim l^{-18-10s}\sigma^{-1}\lesssim l^{-28}\lambda_{q+1}^{-15\varepsilon}.\label{eq3.82}
\end{align}
Thus, taking into account \eqref{eq2.25} and \eqref{eq3.80}--\eqref{eq3.82} we obtain
\begin{equation}\label{eq3.83}
    \begin{aligned}
        &\quad\|m_{q+1}-m_q\|_{L^p_tC_x^s}+\|b_{q+1}-b_q\|_{L^p_tC_x^s}\\
    &\lesssim\|m_{l}-m_q\|_{L^p_tC^s_x}+ \|b_{l}-b_q\|_{L^p_tC^s_x}+\|w_{q+1}\|_{L^p_tC^s_x}+\|d_{q+1}\|_{L^p_tC^s_x}\\
    &\lesssim\|m_{l}-m_q\|_{L^p_tC_x}^{1-s}\|m_{l}-m_q\|_{L^p_tC^1_x}^s+\|w_{q+1}\|_{L^p_tC^s_x}\\
    &\quad+\|b_{l}-b_q\|_{L^p_tC_x}^{1-s}\|b_{l}-b_q\|_{L^p_tC^1_x}^s+\|d_{q+1}\|_{L^p_tC^s_x}\\
    &\lesssim \left(l\lambda_q^4\right)^{1-s}\left(l\lambda_q^6\right)^{s}+\lambda_{q+1}^{\alpha-\frac{2\alpha}{p}+s+\varepsilon(\frac{10}{p}-4)}+l^{-28}\lambda_{q+1}^{-15\varepsilon},
    \end{aligned}
\end{equation}
which together with \eqref{eq2.3s} and \eqref{eq2.2s} yields that
\begin{align}\label{eq3.84}
    \|m_{q+1}-m_q\|_{L^p_tC_x^s}+\|b_{q+1}-b_q\|_{L^p_tC_x^s}\lesssim\delta_{q+2}^{\frac{1}{2}}.
\end{align}

Finally, concerning the $C_tH^{-1}_x$-decay of the momentum and magnetic perturbations, by using \eqref{eq2.25}, \eqref{eq3.14}, \eqref{eq3.46}--\eqref{eq3.48}, \eqref{eq3.50}, \eqref{eq3.51} and Lemma \ref{lem-magnetic-amplitude-estimate} and \ref{lem-momentum-amplitude-estimate},
\begin{equation}\label{eq3.85}
    \begin{aligned}
        &\quad\|m_{q+1}-m_q\|_{C_tH^{-1}_x}+\|b_{q+1}-b_q\|_{C_tH^{-1}_x}\\
        &\lesssim\|m_l-m_q\|_{C_tH^{-1}_x}+\|b_l-b_q\|_{C_tH^{-1}_x}+\|w_{q+1}\|_{C_tH^{-1}_x}+\|d_{q+1}\|_{C_tH^{-1}_x}\\
        &\lesssim\|m_l-m_q\|_{C_{t,x}}+\|w^{(p)}_{q+1}+w^{(c)}_{q+1}\|_{C_tH^{-1}_x}+\|w^{(o)}_{q+1}\|_{C_tL^2_x}\\
        &\quad+\|b_l-b_q\|_{C_{t,x}}+\|d^{(p)}_{q+1}+d^{(c)}_{q+1}\|_{C_tH^{-1}_x}+\|d^{(o)}_{q+1}\|_{C_tL^2_x}\\
        &\lesssim l\lambda_q^4+\sum_{k\in\Lambda_m\cup\Lambda_b}\|g_{(k)}{\rm curl}\left(a_{(k)}W^c_{(k)}\right)\|_{C_tL^2_x}+\sum_{k\in\Lambda_b}\|g_{(k)}{\rm curl}\left(a_{(k)}D^c_{(k)}\right)\|_{C_tL^2_x}+l^{-18}\sigma^{-1}\\
        &\lesssim l\lambda_q^4+l^{-18}\sigma^{-1}+\sum_{k\in\Lambda_b}\|g_{(k)}\|_{C_t}\left(\|\nabla a_{(k)}\|_{C_{t,x}}\|D^c_{(k)}\|_{L^2_x}+\|a_{(k)}\|_{C_{t,x}}\|\nabla D^c_{(k)}\|_{L^2_x}\right)\\
        &\quad+\sum_{k\in\Lambda_m\cup\Lambda_b}\|g_{(k)}\|_{C_t}\left(\|\nabla a_{(k)}\|_{C_{t,x}}\|W^c_{(k)}\|_{L^2_x}+\|a_{(k)}\|_{C_{t,x}}\|\nabla W^c_{(k)}\|_{L^2_x}\right)\\
        &\lesssim \lambda_q^{-26}+\lambda_{q+1}^{-14\varepsilon}+l^{-15}\lambda_{q+1}^{\alpha-1-5\varepsilon}\\
        &\lesssim\delta_{q+2}^{\frac{1}{2}},
    \end{aligned}
\end{equation}
where the last inequality is valid due to \eqref{eq2.3s} and \eqref{eq2.2s}.

Thus far, the inductive estimates \eqref{eq2.5}--\eqref{eq2.7} and \eqref{eq2.10}--\eqref{eq2.13} are verified at level $q+1$.
\section{Reynolds and magnetic stresses}\label{sec4}
We are now in the stage to treat the delicate Reynolds and magnetic stresses and check the corresponding inductive estimates \eqref{eq2.8} and \eqref{eq2.9} at level $q+1$.

To begin with, we recall from \cite{LZZ-2022-jmpa} the key inverse-divergence operators $\mathcal{R}^m$ and $\mathcal{R}^b$, which are defined by
\begin{align}  &\left(\mathcal{R}^mv\right)_{kl}:=\partial_k\Delta^{-1}v_l+\partial_l\Delta^{-1}v_k-\frac{1}{2}\left(\delta_{kl}+\partial_k\partial_l\Delta^{-1}\right){\rm div}\Delta^{-1}v,\label{eq4.1}\\
&\left(\mathcal{R}^bf\right)_{ij}=\epsilon_{ijk}(-\Delta)^{-1}\left({\rm curl}f\right)_k,\label{eq4.2}
\end{align}
where $\epsilon_{ijk}$ is the Levi-Civita tensor and the smooth vector fields $v$ and $f$ are both mean-free, moreover, $f$ is divergence-free.

The operator $\mathcal{R}^m$ maps smooth vector field to symmetric and trace-free matrices, while the operator $\mathcal{R}^b$ returns skew-symmetric matrices. Moreover, one has the following identities
\begin{align}\label{eq4.3}
    {\rm div}\mathcal{R}^m(v)=v,\quad{\rm div}\mathcal{R}^b(f)=f.
\end{align}
We note that $\mathcal{R}^m\nabla$, $\mathcal{R}^m{\rm div}$, $\mathcal{R}^b\nabla$ and $\mathcal{R}^b{\rm div}$ are all Calder\'on-Zygmund operators and thus they are bounded in periodic H\"older spaces. See \cite{CZ-1954-studiamath,DCS-2013-invent,BBV-2020-annpde} for details.
\subsection{Decomposition of Reynolds and magnetic stresses}
For more precise estimates, we need to decompose Reynolds and magnetic stresses appropriately.
\subsubsection{Decomposition of magnetic stress}
Using $\eqref{eq2.1}_3$ at level $q+1$, $\eqref{eq2.18}_3$, \eqref{eq2.20}, \eqref{eq3.25}, \eqref{eq3.42} and \eqref{eq3.53}--\eqref{eq3.55}, we derive that 
\begin{equation}\label{eq4.4}
    \begin{aligned}
        {\rm div}R^b_{q+1}=
&\partial_t\left(d_{q+1}^{(p)}+d_{q+1}^{(c)}\right)+\eta\mu^{-1}(-\Delta)^{\alpha_2}d_{q+1}+{\rm div}\left((\rho_{q+1}^{-1}-\rho_l^{-1})(b_l\otimes m_l-m_l\otimes b_l)\right)\\
&\underbrace{+{\rm div}\left(\rho_{q+1}^{-1}(b_l\otimes w_{q+1}+d_{q+1}\otimes m_l-m_l\otimes d_{q+1}-w_{q+1}\otimes b_l)\right)\qquad\qquad\qquad}_{{\rm div}R^b_{{\rm lin}}}\\
&+{\rm div}\left(\rho_l^{-1}(d_{q+1}^{(p)}\otimes w_{q+1}^{(p)}-w_{q+1}^{(p)}\otimes d_{q+1}^{(p)})+R^b_l\right)+\partial_td_{q+1}^{(o)}\\
&\underbrace{+{\rm div}\left((\rho_{q+1}^{-1}-\rho_l^{-1})(d_{q+1}\otimes w_{q+1}-w_{q+1}\otimes d_{q+1})\right)\qquad\quad}_{{\rm div}R^b_{{\rm osc}}}\\
&+{\rm div}\left(\rho_l^{-1}\left(d_{q+1}^{(p)}\otimes (w_{q+1}^{(c)}+w_{q+1}^{(o)})-w_{q+1}^{(p)}\otimes (d_{q+1}^{(c)}+d_{q+1}^{(o)})\right)\right)\\
&\underbrace{+{\rm div}\left(\rho_l^{-1}\left((d_{q+1}^{(c)}+d_{q+1}^{(o)})\otimes w_{q+1}-(w_{q+1}^{(c)}+w_{q+1}^{(o)})\otimes d_{q+1}\right)\right)\quad}_{{\rm div}R^b_{{\rm cor}}}\\
&+{\rm div}R^b_{{\rm com}},
    \end{aligned}
\end{equation}
where $R^b_{{\rm com}}$ is the commutator stress to the magnetic equation given by \eqref{eq2.20}.

Using the inverse-divergence operator $\mathcal{R}^b$, we define the new magnetic stress at level $q+1$ by
\begin{align}\label{eq4.5}
    R^b_{q+1}:=R^b_{{\rm lin}}+R^b_{{\rm osc}}+R^b_{{\rm cor}}+R^b_{{\rm com}},
\end{align}
where the linear error
\begin{equation}\label{eq4.6}
    \begin{aligned}
        R^b_{{\rm lin}}&:=
        \mathcal{R}^b\partial_t\left(d_{q+1}^{(p)}+d_{q+1}^{(c)}\right)+\eta\mu^{-1}\mathcal{R}^b(-\Delta)^{\alpha_2}d_{q+1}+\mathcal{R}^b{\rm div}\left((\rho_{q+1}^{-1}-\rho_l^{-1})(b_l\otimes m_l-m_l\otimes b_l)\right)\\
&\quad+\mathcal{R}^b{\rm div}\left(\rho_{q+1}^{-1}(b_l\otimes w_{q+1}+d_{q+1}\otimes m_l-m_l\otimes d_{q+1}-w_{q+1}\otimes b_l)\right)\\
&=:R^b_{{\rm lin},1}+R^b_{{\rm lin},2}+R^b_{{\rm lin},3}+R^b_{{\rm lin},4},
    \end{aligned}
\end{equation}
the oscillation error
\begin{equation}\label{eq4.7}
    \begin{aligned}
        R^b_{{\rm osc}}&:=      \sum_{k\in\Lambda_b}\mathcal{R}^b\mathbb{P}_{\neq0}\left(\mathbb{P}_{\neq0}(D_{(k)}\otimes W_{(k)}-W_{(k)}\otimes D_{(k)})\nabla\left(\rho_l^{-1}a^2_{(k)}\right)\right)\\
        &\quad-\sigma^{-1}\sum_{k\in\Lambda_b}h_{(k)}\partial_t\mathcal{R}^b(k_2\otimes k_1-k_1\otimes k_2)\nabla(a^2_{(k)})\\
        &\quad+\mathcal{R}^b{\rm div}\left((\rho_{q+1}^{-1}-\rho_l^{-1})(d_{q+1}\otimes w_{q+1}-w_{q+1}\otimes d_{q+1})\right)\\
        &=:R^b_{{\rm osc},1}+R^b_{{\rm osc},2}+R^b_{{\rm osc},3},
    \end{aligned}
\end{equation}
the corrector error
\begin{equation}\label{eq4.8}
    \begin{aligned}
        R^b_{{\rm cor}}
        &:=
\mathcal{R}^b{\rm div}\left(\rho_l^{-1}\left(d_{q+1}^{(p)}\otimes (w_{q+1}^{(c)}+w_{q+1}^{(o)})-w_{q+1}^{(p)}\otimes (d_{q+1}^{(c)}+d_{q+1}^{(o)})\right)\right)\\
&\quad+\mathcal{R}^b{\rm div}\left(\rho_l^{-1}\left((d_{q+1}^{(c)}+d_{q+1}^{(o)})\otimes w_{q+1}-(w_{q+1}^{(c)}+w_{q+1}^{(o)})\otimes d_{q+1}\right)\right),
    \end{aligned}
\end{equation}
and $R^b_{{\rm com}}$ is the commutator error given by \eqref{eq2.20}.

We also note that, the nonlinear terms in the
magnetic equation is skew-symmetric, which, in particular, yields that
\begin{align}\label{eq4.9}
    {\rm div}\left({\rm div}R^b_{q+1}\right)=0.
\end{align}
Furthermore, it leads 
\begin{align}\label{eq4.10}
    R^b_{q+1}=\mathcal{R}^b{\rm div}R^b_{q+1}.
\end{align}
\subsubsection{Decomposition of Reynolds stress}
Now we concern the Reynolds stress. By virtue of $\eqref{eq2.1}_2$, $\eqref{eq2.18}_2$, \eqref{eq2.19}, \eqref{eq3.43}, \eqref{eq3.52}, \eqref{eq3.54} and \eqref{eq3.55}, we compute
\begin{equation}\label{eq4.11}
    \begin{aligned}
        {\rm div}R^m_{q+1}=
&\partial_t\left(w_{q+1}^{(p)}+w_{q+1}^{(c)}\right)+\nu^s(-\Delta)^{\alpha_1}\left((\rho^{-1}_{q+1}-\rho^{-1}_l)m_l+(\rho^{-1}_{q+1}-\rho^{-1}_l)w_{q+1}+\rho^{-1}_lw_{q+1}\right)\\
&-(\nu^b+\frac{1}{3}\nu^s)\nabla{\rm div}\left((\rho^{-1}_{q+1}-\rho^{-1}_l)m_l+(\rho^{-1}_{q+1}-\rho^{-1}_l)w_{q+1}+\rho^{-1}_lw_{q+1}\right)\\
&+{\rm div}(\rho_{q+1}^{-1}(m_l\otimes w_{q+1}+w_{q+1}\otimes m_l)+(\rho^{-1}_{q+1}-\rho^{-1}_l)m_l\otimes m_l\\
&\underbrace{\qquad\qquad\qquad\quad-\mu\left(b_l\otimes d_{q+1}+d_{q+1}\otimes b_l-(b_l\cdot d_{q+1})\mathbb{I}\right))\qquad\qquad\qquad\qquad}_{{\rm div}R^m_{{\rm lin}}}\\
&+{\rm div}\left(\rho_l^{-1}(w_{q+1}^{(p)}\otimes w_{q+1}^{(p)}-\mu\left(d_{q+1}^{(p)}\otimes d_{q+1}^{(p)}-\frac{1}{2}\left|d_{q+1}^{(p)}\right|^2\mathbb{I}\right)+R^m_l\right)+\partial_tw_{q+1}^{(o)}\\
&\underbrace{+{\rm div}\left((\rho_{q+1}^{-1}-\rho_l^{-1})w_{q+1}\otimes w_{q+1}\right)\qquad\qquad\qquad\qquad\qquad\qquad\qquad\qquad}_{{\rm div}R^m_{{\rm osc}}}\\
&+{\rm div}\left(\rho_l^{-1}\left(w_{q+1}^{(p)}\otimes (w_{q+1}^{(c)}+w_{q+1}^{(o)})+(w_{q+1}^{(c)}+w_{q+1}^{(o)})\otimes w_{q+1}\right)\right)\\
&-{\rm div}\left(\mu\left((d_{q+1}^{(c)}+d_{q+1}^{(o)})\otimes d_{q+1}+d_{q+1}^{(p)}\otimes (d_{q+1}^{(c)}+d_{q+1}^{(o)})\right)\right)\\
&\underbrace{+{\rm div}\left(\frac{\mu}{2}\left(\left(d_{q+1}^{(c)}+d_{q+1}^{(o)}\right)\cdot d_{q+1}+d_{q+1}^{(p)}\cdot\left(d_{q+1}^{(c)}+d_{q+1}^{(o)}\right)\right)\mathbb{I}\right)\qquad}_{{\rm div}R^m_{{\rm cor}}}\\
&+\underbrace{\nabla P(\rho_{q+1})-\nabla P(\rho_l)}_{R^m_{{\rm pre}}}+{\rm div}R^m_{{\rm com}},
    \end{aligned}
\end{equation}
where $R^m_{{\rm com}}$ is the commutator stress to the momentum equation given by \eqref{eq2.19}.

Then, using the inverse-divergence operator $\mathcal{R}^m$ we define the new Reynolds stress at level $q+1$ by
\begin{align}\label{eq4.12}
    R^m_{q+1}:=R^m_{{\rm lin}}+R^m_{{\rm osc}}+R^m_{{\rm cor}}+R^m_{{\rm pre}}+R^m_{{\rm com}},
\end{align}
where the linear error
\begin{equation}\label{eq4.13}
    \begin{aligned}
        R^m_{{\rm lin}}&:=
        \mathcal{R}^m\partial_t\left(w_{q+1}^{(p)}+w_{q+1}^{(c)}\right)\\
        &\quad+\nu^s\mathcal{R}^m(-\Delta)^{\alpha_1}\left((\rho^{-1}_{q+1}-\rho^{-1}_l)m_l+(\rho^{-1}_{q+1}-\rho^{-1}_l)w_{q+1}+\rho^{-1}_lw_{q+1}\right)\\
&\quad-(\nu^b+\frac{1}{3}\nu^s)\mathcal{R}^m\nabla{\rm div}\left((\rho^{-1}_{q+1}-\rho^{-1}_l)m_l+(\rho^{-1}_{q+1}-\rho^{-1}_l)w_{q+1}+\rho^{-1}_lw_{q+1}\right)\\
&\quad+\mathcal{R}^m{\rm div}\left(\rho_{q+1}^{-1}(m_l\otimes w_{q+1}+w_{q+1}\otimes m_l)+(\rho^{-1}_{q+1}-\rho^{-1}_l)m_l\otimes m_l\right)\\
&\quad-\mathcal{R}^m{\rm div}\left(\mu\left(b_l\otimes d_{q+1}+d_{q+1}\otimes b_l-(b_l\cdot d_{q+1})\mathbb{I}\right)\right)\\
&=:R^m_{{\rm lin},1}+R^m_{{\rm lin},2}+R^m_{{\rm lin},3}+R^m_{{\rm lin},4}+R^m_{{\rm lin},5},
    \end{aligned}
\end{equation}
the oscillation error
\begin{equation}\label{eq4.14}
    \begin{aligned}
        R^m_{{\rm osc}}&:=\sum_{k\in\Lambda_b}g^2_{(k)}\mathcal{R}^m\mathbb{P}_{\neq0}\left(\mathbb{P}_{\neq0}\left(W_{(k)}\otimes W_{(k)}\right)\nabla(\rho_l^{-1}a^2_{(k)})\right.\\
        &\qquad\left.-\mu\mathbb{P}_{\neq0}\left( D_{(k)}\otimes D_{(k)}-\frac{1}{2}|D_{(k)}|^2\mathbb{I}\right)\nabla(a^2_{(k)})\right)
        \\
        &\quad+\sum_{k\in\Lambda_m}\mathcal{R}^m\mathbb{P}_{\neq0}\left(g^2_{(k)}\mathbb{P}_{\neq0}(W_{(k)}\otimes W_{(k)})\nabla(\rho_l^{-1}a^2_{(k)}) \right)\\
        &\quad-\sigma^{-1}\sum_{k\in\Lambda_m\cup\Lambda_b}h_{(k)}\mathcal{R}^m\left(\left(k_1\cdot\nabla\right)\partial_t(\rho_l^{-1}a^2_{(k)})k_1\right)\\
        &\quad+\sigma^{-1}\sum_{k\in\Lambda_b}\mu h_{(k)}\left(\mathcal{R}^m\left(\left(k_2\cdot\nabla\right)\partial_t(a^2_{(k)})k_2\right)-\frac{1}{2}\partial_t\mathcal{R}^m\nabla(a^2_{(k)})\right)\\
        &\quad+\mathcal{R}^m{\rm div}\left((\rho_{q+1}^{-1}-\rho_l^{-1})w_{q+1}\otimes w_{q+1}\right)\\
        &=:  R^m_{{\rm osc},1}+R^m_{{\rm osc},2}+R^m_{{\rm osc},3}+R^m_{{\rm osc},4}+R^m_{{\rm osc},5},
    \end{aligned}
\end{equation}
the corrector error
\begin{equation}\label{eq4.15}
    \begin{aligned}
        R^m_{{\rm cor}}&:=
        \mathcal{R}^m{\rm div}\left(\rho_l^{-1}\left(w_{q+1}^{(p)}\otimes (w_{q+1}^{(c)}+w_{q+1}^{(o)})+(w_{q+1}^{(c)}+w_{q+1}^{(o)})\otimes w_{q+1}\right)\right)\\
&\quad-\mathcal{R}^m{\rm div}\left(\mu\left((d_{q+1}^{(c)}+d_{q+1}^{(o)})\otimes d_{q+1}+d_{q+1}^{(p)}\otimes (d_{q+1}^{(c)}+d_{q+1}^{(o)})\right)\right)\\
&\quad+\mathcal{R}^m{\rm div}\left(\frac{\mu}{2}\left(\left(d_{q+1}^{(c)}+d_{q+1}^{(o)}\right)\cdot d_{q+1}+d_{q+1}^{(p)}\cdot\left(d_{q+1}^{(c)}+d_{q+1}^{(o)}\right)\right)\mathbb{I}\right),
    \end{aligned}
\end{equation}
the pressure error 
\begin{equation}\label{eq4.16}
    R^m_{{\rm pre}}
:=\mathcal{R}^m\nabla\left(P(\rho_{q+1})-P(\rho_l)\right),
\end{equation}
and the commutator error
\begin{equation}\label{eq4.17}
   \begin{aligned}
        R^m_{{\rm com}}&:=
    \nu^s\mathcal{R}^m(-\Delta)^{\alpha_1}\left(\rho_l^{-1}m_l-((\rho_q^{-1}m_q)*_x\phi_l)*_t\varphi_l\right)+\mathcal{R}^m\nabla\left(P(\rho_l)-P_l\right)\\
        &\quad-(\nu^b+\frac{1}{3}\nu^s)\mathcal{R}^m\nabla{\rm div}\left(\rho_l^{-1}m_l-((\rho_q^{-1}m_q)*_x\phi_l)*_t\varphi_l\right)\\
        &\quad+\mathcal{R}^m{\rm div}\left(\rho_l^{-1}m_l\otimes m_l-\left(\rho_q^{-1}m_q\otimes m_q\right)*_x\phi_l*_t\varphi_l\right)\\
        &\quad-\mu\mathcal{R}^m{\rm div}\left(b_l\otimes b_l-\frac{1}{2}|b_l|^2\mathbb{I}-(b_q\otimes b_q-\frac{1}{2}|b_q|^2\mathbb{I})*_x\phi_l*_t\varphi_l\right)\\
        &=:R^m_{{\rm com},1}+R^m_{{\rm com},2}+R^m_{{\rm com},3}+R^m_{{\rm com},4}+R^m_{{\rm com},5}.
   \end{aligned}
\end{equation}

By \eqref{eq3.43}--\eqref{eq3.53} and the fact that ${\rm div}(W_{(k)}\otimes W_{(k)})=0$, we have
\begin{align}\label{eq4.18}
    R^m_{q+1}=\mathcal{R}^m{\rm div}R^m_{q+1}.
\end{align}
\subsection{Verification of $C^1_{t,x}$-estimates}
The purpose of this subsection is to verify the $C^1_{t,x}$-estimates \eqref{eq2.8} of $R^m_{q+1}$ and $R^b_{q+1}$. 

First, by \eqref{eq4.10}, \eqref{eq4.18}, the Sobolev embedding $W^{1,4}_x\hookrightarrow C_x$, equations $\eqref{eq2.1}_2$ and $\eqref{eq2.1}_3$ at level $q+1$ and the fact $\mathcal{R}^m,\mathcal{R}^b\sim|\nabla|^{-1}$ in Sobolev spaces $W^{s,p}$, we obtain
\begin{equation}\label{eq4.19}
    \begin{aligned}
        &\quad\|R^m_{q+1}\|_{C_tC^1_x}+\|R^b_{q+1}\|_{C_tC^1_x}\\
        &\lesssim\left\|\mathcal{R}^m\left({\rm div}R^m_{q+1}\right)\right\|_{C_tW^{2,4}_x}+\left\|\mathcal{R}^b\left({\rm div}R^b_{q+1}\right)\right\|_{C_tW^{2,4}_x}\\
        &\lesssim\|\partial_tm_{q+1}\|_{C_tW^{1,4}_x}+\|\partial_tb_{q+1}\|_{C_tW^{1,4}_x}+\|\nu^s(-\Delta)^{\alpha_1}(\rho_{q+1}^{-1}m_{q+1})\|_{C_tW^{1,4}_x}+\|\nabla P(\rho_{q+1})\|_{C_tW^{1,4}_x}\\
        &\quad+\|\eta\mu^{-1}(-\Delta)^{\alpha_2}b_{q+1}\|_{C_tW^{1,4}_x}+\left\|(\nu^b+\frac{1}{3}\nu^s)\nabla{\rm div}(\rho_{q+1}^{-1}m_{q+1})\right\|_{C_tW^{1,4}_x}\\
        &\quad+\left\|{\rm div}\left(\rho_{q+1}^{-1}(b_{q+1}\otimes m_{q+1}-m_{q+1}\otimes b_{q+1})\right)\right\|_{C_tW^{1,4}_x}\\
        &\quad+\left\|{\rm div}\left(\rho_{q+1}^{-1}m_{q+1}\otimes m_{q+1}-\mu(b_{q+1}\otimes b_{q+1}-\frac{1}{2}|b_{q+1}|^2\mathbb{I})\right)\right\|_{C_tW^{1,4}_x}.
    \end{aligned}
\end{equation}
Note that, by interpolation and \eqref{eq2.5} at level $q+1$, we have
\begin{equation}\label{eq4.20}
    \begin{aligned}
      \|\nu^s(-\Delta)^{\alpha_1}(\rho_{q+1}^{-1}m_{q+1})\|_{C_tW^{1,4}_x}
        &\lesssim\|\rho_{q+1}^{-1}m_{q+1}\|^{1-\frac{2\alpha_1+1}{4}}_{C_tC_x}\|\rho_{q+1}^{-1}m_{q+1}\|^{\frac{2\alpha_1+1}{4}}_{C_tW_x^{4,\infty}}\\
        &\lesssim\|m_{q+1}\|_{C_{t,x}}^{\frac{3-2\alpha_1}{4}}\left(\|\rho_{q+1}^{-1}\|_{C_tC_x^4}\|m_{q+1}\|_{C_{t,x}^4}\right)^{\frac{2\alpha_1+1}{4}}
    \end{aligned}
\end{equation}
and
\begin{equation}\label{eq4.21} 
        \|\eta\mu^{-1}(-\Delta)^{\alpha_2}b_{q+1}\|_{C_tW^{1,4}_x}\lesssim\|b_{q+1}\|_{C_{t,x}}^{\frac{3-2\alpha_2}{4}}\|b_{q+1}\|_{C_{t,x}^4}^{\frac{2\alpha_2+1}{4}}.
\end{equation}
Moreover,
\begin{equation}\label{eq4.22}
        \left\|(\nu^b+\frac{1}{3}\nu^s)\nabla{\rm div}(\rho_{q+1}^{-1}m_{q+1})\right\|_{C_tW^{1,4}_x}\lesssim\|\rho_{q+1}^{-1}m_{q+1}\|_{C_tC_x^3}\lesssim\|\rho_{q+1}^{-1}\|_{C_tC_x^3}\|m_{q+1}\|_{C_tC^3_x},
\end{equation}
Furthermore, by Leibniz rule, we get
\begin{equation}\label{eq4.23}
    \begin{aligned}
        &\quad\left\|{\rm div}\left(\rho_{q+1}^{-1}(b_{q+1}\otimes m_{q+1}-m_{q+1}\otimes b_{q+1})\right)\right\|_{C_tW^{1,4}_x}\\
        &\lesssim\|\rho_{q+1}^{-1}(b_{q+1}\otimes m_{q+1}-m_{q+1}\otimes b_{q+1})\|_{C_tC^2_x}\\      &\lesssim\|\rho_{q+1}^{-1}\|_{C_tC^2_x}\sum_{N_1+N_2=2}\|m_{q+1}\|_{C_tC^{N_1}_x}\|b_{q+1}\|_{C_tC^{N_2}_x},
    \end{aligned}
\end{equation}
and
\begin{equation}\label{eq4.24}
    \begin{aligned}
        &\quad\left\|{\rm div}\left(\rho_{q+1}^{-1}m_{q+1}\otimes m_{q+1}-\mu(b_{q+1}\otimes b_{q+1}-\frac{1}{2}|b_{q+1}|^2\mathbb{I})\right)\right\|_{C_tW^{1,4}_x}\\
        &\lesssim\left\|\rho_{q+1}^{-1}m_{q+1}\otimes m_{q+1}-\mu(b_{q+1}\otimes b_{q+1}-\frac{1}{2}|b_{q+1}|^2\mathbb{I})\right\|_{C_tC^2_x}\\     &\lesssim\|\rho_{q+1}^{-1}\|_{C_tC^2_x}\sum_{N_1+N_2=2}\left(\|m_{q+1}\|_{C_tC^{N_1}_x}\|m_{q+1}\|_{C_tC^{N_2}_x}+\|b_{q+1}\|_{C_tC^{N_1}_x}\|b_{q+1}\|_{C_tC^{N_2}_x}\right).
    \end{aligned}
\end{equation}
Then, plugging \eqref{eq4.20}--\eqref{eq4.24} into \eqref{eq4.19} yields that
\begin{equation}\label{eq4.25}
    \begin{aligned}
        &\quad\|R^m_{q+1}\|_{C_tC^1_x}+\|R^b_{q+1}\|_{C_tC^1_x}\\
        &\lesssim\|\partial_tm_{q+1}\|_{C_tW^{1,4}_x}+\|\partial_tb_{q+1}\|_{C_tW^{1,4}_x}+\|\nabla P(\rho_{q+1})\|_{C_tW^{1,4}_x}\\
        &\quad+\|m_{q+1}\|_{C_{t,x}}^{\frac{3-2\alpha_1}{4}}\left(\|\rho_{q+1}^{-1}\|_{C_tC_x^4}\|m_{q+1}\|_{C_{t,x}^4}\right)^{\frac{2\alpha_1+1}{4}}+\|b_{q+1}\|_{C_{t,x}}^{\frac{3-2\alpha_2}{4}}\|b_{q+1}\|_{C_{t,x}^4}^{\frac{2\alpha_2+1}{4}}\\
        &\quad+\|\rho_{q+1}^{-1}\|_{C_tC^2_x}\sum_{N_1+N_2=2}\|m_{q+1}\|_{C_tC^{N_1}_x}\|b_{q+1}\|_{C_tC^{N_2}_x}\\
        &\quad+\|\rho_{q+1}^{-1}\|_{C_tC^2_x}\sum_{N_1+N_2=2}\left(\|m_{q+1}\|_{C_tC^{N_1}_x}\|m_{q+1}\|_{C_tC^{N_2}_x}+\|b_{q+1}\|_{C_tC^{N_1}_x}\|b_{q+1}\|_{C_tC^{N_2}_x}\right).
    \end{aligned}
\end{equation}

Using \eqref{eq2.5} and \eqref{eq2.6} at level $q+1$, we obtain
\begin{equation}\label{eq4.26}
    \|\partial_t^M\rho_{q+1}^{-1}\|_{C_tC^N_x}\lesssim\lambda_{q+1}^{\frac{N\varepsilon}{4}}.
\end{equation}
Moreover, since $P^{\prime}$ and $P^{\prime\prime}$ are continuous and $\rho_{q+1}$ is uniformly bounded away from zero and infinity, that is, by Leibniz rule again, we get
\begin{equation}\label{eq4.27}
    \|\nabla P(\rho_{q+1})\|_{C_tC^1_x}\lesssim\|P^{\prime}(\rho_{q+1})\|_{C_{t,x}}\|\rho_{q+1}\|_{C_tC^2_x}+\|P^{\prime\prime}(\rho_{q+1})\|_{C_{t,x}}\|\rho_{q+1}\|^2_{C_tC^1_x}\lesssim\lambda_{q+1}^{\varepsilon}.
\end{equation}
Thus, combining with \eqref{eq2.2s}, \eqref{eq4.25}--\eqref{eq4.27} and \eqref{eq2.7} at level $q+1$, it yields
\begin{equation}\label{eq4.28}
\|R^m_{q+1}\|_{C_tC^1_x}+\|R^b_{q+1}\|_{C_tC^1_x}\lesssim\lambda_{q+1}^6+\lambda_{q+1}^{8+\frac{\varepsilon}{2}}+\lambda_{q+1}^{4+5\alpha_1}+\lambda_{q+1}^{4+5\alpha_2}\lesssim\lambda_{q+1}^9.
\end{equation}

Now we turn to the $C^1_tC_x$-estimates. Similarly, we deduce
\begin{equation}\label{eq4.29}
    \begin{aligned}
        &\quad\|\partial_tR^m_{q+1}\|_{C_{t,x}}+\|\partial_tR^b_{q+1}\|_{C_{t,x}}\\
        &=\left\|\partial_t\mathcal{R}^m\left({\rm div}R^m_{q+1}\right)\right\|_{C_{t,x}}+\left\|\partial_t\mathcal{R}^b\left({\rm div}R^b_{q+1}\right)\right\|_{C_{t,x}}\\      
        &\lesssim\left\|\partial_t\mathcal{R}^m\left({\rm div}R^m_{q+1}\right)\right\|_{C_{t}W^{1,4}_x}+\left\|\partial_t\mathcal{R}^b\left({\rm div}R^b_{q+1}\right)\right\|_{C_{t}W^{1,4}_x}\\
        &\lesssim\|\partial_t{\rm div}R^m_{q+1}\|_{C_{t}L^4_x}+\|\partial_t{\rm div}R^b_{q+1}\|_{C_{t}L^4_x}\\ 
        &\lesssim\|\partial_t{\rm div}R^m_{q+1}\|_{C_{t}L^{\infty}_x}+\|\partial_t{\rm div}R^b_{q+1}\|_{C_{t}L^{\infty}_x}\\ &\lesssim\|\partial^2_tm_{q+1}\|_{C_tL^{\infty}_x}+\|\partial^2_tb_{q+1}\|_{C_tL^{\infty}_x}\\
        &\quad+\|\nu^s(-\Delta)^{\alpha_1}\partial_t(\rho_{q+1}^{-1}m_{q+1})\|_{C_tL^{\infty}_x}+\|\eta\mu^{-1}(-\Delta)^{\alpha_2}\partial_tb_{q+1}\|_{C_tL^{\infty}_x}\\
        &\quad+\left\|(\nu^b+\frac{1}{3}\nu^s)\nabla{\rm div}\partial_t(\rho_{q+1}^{-1}m_{q+1})\right\|_{C_tL^{\infty}_x}+\left\|{\rm div}\partial_t\left(\rho_{q+1}^{-1}(b_{q+1}\otimes m_{q+1}-m_{q+1}\otimes b_{q+1})\right)\right\|_{C_tL^{\infty}_x}\\
        &\quad+\left\|{\rm div}\partial_t\left(\rho_{q+1}^{-1}m_{q+1}\otimes m_{q+1}-\mu(b_{q+1}\otimes b_{q+1}-\frac{1}{2}|b_{q+1}|^2\mathbb{I})\right)\right\|_{C_tL^{\infty}_x}+\|\nabla\partial_t P(\rho_{q+1})\|_{C_tL^{\infty}_x}\\
        &\lesssim\|m_{q+1}\|_{C^2_{t,x}}+\|b_{q+1}\|_{C^2_{t,x}}+\|(-\Delta)^{\alpha_1}(\partial_t\rho_{q+1}^{-1}m_{q+1})\|_{C_{t,x}}+\|(-\Delta)^{\alpha_1}(\rho_{q+1}^{-1}\partial_tm_{q+1})\|_{C_{t,x}}\\
        &\quad+\|(-\Delta)^{\alpha_2}\partial_tb_{q+1}\|_{C_{t,x}}+\sum_{M_1+M_2=1}\|\partial_t^{M_1}\rho_{q+1}^{-1}\|_{C_tC^2_x}\|\partial_t^{M_2}m_{q+1}\|_{C_tC^2_x}\\
        &\quad+\|\rho_{q+1}^{-1}\|_{C_tC^2_x}\sum_{N_1+N_2=2}\|m_{q+1}\|_{C_tC^{N_1}_x}\|b_{q+1}\|_{C_tC^{N_2}_x}\\
        &\quad+\|\rho_{q+1}^{-1}\|_{C_tC^2_x}\sum_{N_1+N_2=2}\left(\|m_{q+1}\|_{C_tC^{N_1}_x}\|m_{q+1}\|_{C_tC^{N_2}_x}+\|b_{q+1}\|_{C_tC^{N_1}_x}\|b_{q+1}\|_{C_tC^{N_2}_x}\right)\\
        &\quad+\|P^{\prime}(\rho_{q+1})\|_{C_{t,x}}\|\partial_t\nabla\rho_{q+1}\|_{C_{t,x}}+\|P^{\prime\prime}(\rho_{q+1})\|_{C_{t,x}}\|\nabla\rho_{q+1}\|_{C_{t,x}}\|\partial_t\rho_{q+1}\|_{C_{t,x}}.
    \end{aligned}
\end{equation}
It remains to estimate the hypo-dissipative terms, that is, using interpolation, \eqref{eq2.7} at level $q+1$ and \eqref{eq4.26}, we get
\begin{equation}\label{eq4.30}
    \begin{aligned}
        &\quad\|(-\Delta)^{\alpha_1}(\partial_t\rho_{q+1}^{-1}m_{q+1})\|_{C_{t,x}}+\|(-\Delta)^{\alpha_1}(\rho_{q+1}^{-1}\partial_tm_{q+1})\|_{C_{t,x}}+\|(-\Delta)^{\alpha_2}\partial_tb_{q+1}\|_{C_{t,x}}\\
        &\lesssim\left(\|\partial_t\rho_{q+1}^{-1}\|_{C_{t,x}}\|m_{q+1}\|_{C_{t,x}}\right)^{1-\alpha_1}\left(\|\partial_t\rho_{q+1}^{-1}\|_{C_tC^2_x}\|m_{q+1}\|_{C_tC^2_x}\right)^{\alpha_1}\\
        &\quad+\left(\|\rho_{q+1}^{-1}\|_{C_{t,x}}\|\partial_tm_{q+1}\|_{C_{t,x}}\right)^{1-\alpha_1}\left(\|\rho_{q+1}^{-1}\|_{C_tC^2_x}\|\partial_tm_{q+1}\|_{C_tC^2_x}\right)^{\alpha_1}\\
        &\quad+\|b_{q+1}\|_{C^1_tC_x}^{1-\alpha_2}\|b_{q+1}\|_{C^1_tC^2_x}^{\alpha_2}\\
        &\lesssim\left(\lambda_{q+1}^{2+\varepsilon}\right)^{1-\alpha_1}\left(\lambda_{q+1}^{6+\varepsilon}\right)^{\alpha_1}+\left(\lambda_{q+1}^{4+\varepsilon}\right)^{1-\alpha_1}\left(\lambda_{q+1}^{8+\varepsilon}\right)^{\alpha_1}+\lambda_{q+1}^{4-4\alpha_2}\lambda_{q+1}^{8\alpha_2}\\
        &\lesssim\lambda_{q+1}^9.
    \end{aligned}
\end{equation}
Thus, taking into account \eqref{eq2.6} and \eqref{eq2.7} at level $q+1$, \eqref{eq4.26}, \eqref{eq4.29} and \eqref{eq4.30}, we deduce that
\begin{equation}\label{eq4.31}
    \begin{aligned}
    \|\partial_tR^m_{q+1}\|_{C_{t,x}}+\|\partial_tR^b_{q+1}\|_{C_{t,x}}
    &\lesssim\lambda_{q+1}^6+\lambda_{q+1}^9+\sum_{M_1+M_2=1}\lambda_{q+1}^{\frac{\varepsilon}{2}}\lambda_{q+1}^{6+2M_2}\\
    &\quad+\lambda_{q+1}^{\frac{\varepsilon}{2}}\sum_{N_1+N_2=2}\lambda_{q+1}^{2N_1+2}\lambda_{q+1}^{2N_2+2}+\lambda_{q+1}^{\varepsilon}\\
    &\lesssim\lambda_{q+1}^9,
    \end{aligned}
\end{equation}
where we used \eqref{eq2.2s} in the last step. 

Therefore, we have shown that the $C^1_{t,x}$-estimates \eqref{eq2.8} is valid at level $q+1$.
\subsection{Verification of $L^1_tC_x$-decay}
We are now in the stage to verify the $L^1_tC_x$-decay \eqref{eq2.9} at level $q+1$. Since Calder\'on-Zygmund operators are bounded in H\"older spaces $C^{\varepsilon}_x$, where $\varepsilon\in(0,1)$ can be chosen as in \eqref{eq2.2s}, we will estimate the new Reynolds and magnetic stresses in the space $L^1_tC^{\varepsilon}_x$ rather than $L^1_tC_x$.
\subsubsection{Linear errors $R^m_{{\rm lin}}$ and $R^b_{{\rm lin}}$}
To begin with, concerning the temporal derivative terms $R^m_{{\rm lin},1}$ and $R^b_{{\rm lin},1}$ in \eqref{eq4.6} and \eqref{eq4.13}, by Lemmas \ref{lem-mikado-flows}--\ref{lem-momentum-amplitude-estimate} and interpolation, we have
\begin{equation}\label{eq4.32}
    \begin{aligned}
    \|R^m_{{\rm lin},1}\|_{L^1_tC^{\varepsilon}_x}
& \lesssim \sum_{k \in \Lambda_m\cup\Lambda_b}\left\|\mathcal{R}^m {\rm curl} {\rm curl}\partial_t\left(g_{(k)} a_{(k)} W_{(k)}^c\right)\right\|_{L_t^1 C_x^{\varepsilon}} \\
& \lesssim \sum_{k \in \Lambda_m\cup\Lambda_b}\left(\left\|g_{(k)}\right\|_{L_t^1}\left\|a_{(k)}\right\|_{C_{t, x}^2}+\left\|\partial_t g_{(k)}\right\|_{L_t^1}\left\|a_{(k)}\right\|_{C_{t, x}^1}\right)\left\|\nabla W_{(k)}^c\right\|_{C_x^{\varepsilon}} \\
& \lesssim \sum_{k \in \Lambda_m\cup\Lambda_b}\left(\left\|g_{(k)}\right\|_{L_t^1}\left\|a_{(k)}\right\|_{C_{t, x}^2}+\left\|\partial_t g_{(k)}\right\|_{L_t^1}\left\|a_{(k)}\right\|_{C_{t, x}^1}\right)\left\|\nabla W_{(k)}^c\right\|_{C_x}^{1-\varepsilon}\left\|\nabla W_{(k)}^c\right\|_{C_x^1}^{\varepsilon}\\
&\lesssim\left(l^{-27}\tau^{-\frac{1}{2}}+l^{-17}\sigma\tau^{\frac{1}{2}}\right)\lambda_{q+1}^{\varepsilon-1}\\
&\lesssim l^{-27}\left(\lambda_{q+1}^{\varepsilon-\alpha-1}+\lambda_{q+1}^{11\varepsilon+\alpha-1}\right)\\
&\lesssim l^{-27}\lambda_{q+1}^{-9\varepsilon},
\end{aligned}
\end{equation}
and
\begin{equation}\label{eq4.33}
    \begin{aligned}
    \|R^b_{{\rm lin},1}\|_{L^1_tC^{\varepsilon}_x}
& \lesssim \sum_{k \in \Lambda_b}\left\|\mathcal{R}^b {\rm curl} {\rm curl}\partial_t\left(g_{(k)} a_{(k)} D_{(k)}^c\right)\right\|_{L_t^1 C_x^{\varepsilon}} \\
& \lesssim \sum_{k \in \Lambda_b}\left(\left\|g_{(k)}\right\|_{L_t^1}\left\|a_{(k)}\right\|_{C_{t, x}^2}+\left\|\partial_t g_{(k)}\right\|_{L_t^1}\left\|a_{(k)}\right\|_{C_{t, x}^1}\right)\left\|\nabla D_{(k)}^c\right\|_{C_x^{\varepsilon}} \\
& \lesssim \sum_{k \in \Lambda_b}\left(\left\|g_{(k)}\right\|_{L_t^1}\left\|a_{(k)}\right\|_{C_{t, x}^2}+\left\|\partial_t g_{(k)}\right\|_{L_t^1}\left\|a_{(k)}\right\|_{C_{t, x}^1}\right)\left\|\nabla D_{(k)}^c\right\|_{C_x}^{1-\varepsilon}\left\|\nabla D_{(k)}^c\right\|_{C_x^1}^{\varepsilon}\\
&\lesssim\left(l^{-13}\tau^{-\frac{1}{2}}+l^{-8}\sigma\tau^{\frac{1}{2}}\right)\lambda_{q+1}^{\varepsilon-1}\\
&\lesssim l^{-13}\left(\lambda_{q+1}^{\varepsilon-\alpha-1}+\lambda_{q+1}^{11\varepsilon+\alpha-1}\right)\\
&\lesssim l^{-13}\lambda_{q+1}^{-9\varepsilon},
\end{aligned}
\end{equation}
where we also used \eqref{eq2.2s} in the last inequality of both \eqref{eq4.32} and \eqref{eq4.33}.

Regarding the hypo-dissipative terms $R^m_{{\rm lin},2}$ and $R^b_{{\rm lin},2}$ in \eqref{eq4.6} and \eqref{eq4.13}, we estimate the terms involving $\alpha_i$ separately. If $\alpha_i\leq\frac12$, we use the smoothing estimate as above; if $\alpha_i>\frac12$, we use the interpolation estimate involving $|\nabla|^{2\alpha_i-1}$.

On one hand, for the case of $\alpha_1,\alpha_2\in(0,\frac{1}{2}]$, by a direct calculation,
\begin{equation}\label{eq4.34}
    \begin{aligned}
\left\|R^m_{{\rm lin}, 2}\right\|_{L_t^1 C_x^{\varepsilon}} &\lesssim \left\|\rho_l^{-1} w_{q+1}+\left(\rho_{q+1}^{-1}-\rho_l^{-1}\right) m_{\ell}+\left(\rho_{q+1}^{-1}-\rho_l^{-1}\right) w_{q+1}\right\|_{L_t^1 C_x^{\varepsilon}} \\
&\lesssim \left\|\rho_l^{-1}\right\|_{C_t C_x^{\varepsilon}}\left\|w_{q+1}\right\|_{L_t^1 C_x}+\left\|\rho_l^{-1}\right\|_{C_{t, x}}\left\|w_{q+1}\right\|_{L_t^1 C_x^{\varepsilon}} \\
& \quad+\left\|\left(\rho_{q+1}^{-1}-\rho_l^{-1}\right)\right\|_{C_t C_x^{\varepsilon}}\left(\left\|m_l\right\|_{L_t^1 C_x^{\varepsilon}}+\left\|w_{q+1}\right\|_{L_t^1 C_x^{\varepsilon}}\right) ,
\end{aligned}
\end{equation}
and
\begin{equation}\label{eq4.35}
\left\|R^b_{{\rm lin}, 2}\right\|_{L_t^1 C_x^{\varepsilon}}\lesssim\|d_{q+1}\|_{L^1_tC_x^{\varepsilon}}.
\end{equation}
Note that, by\eqref{eq3.54} and \eqref{eq3.55},
\begin{equation}\label{eq4.36}
    \begin{aligned}       &\quad\|w_{q+1}\|_{L^1_tC_x^{\varepsilon}}+\|d_{q+1}\|_{L^1_tC_x^{\varepsilon}}\\
    &\lesssim\|w^{(p)}_{q+1}\|_{L^1_tC_x^{\varepsilon}}+\|w^{(c)}_{q+1}\|_{L^1_tC_x^{\varepsilon}}+\|w^{(o)}_{q+1}\|_{L^1_tC_x^{\varepsilon}}\\
   &\quad+\|d^{(p)}_{q+1}\|_{L^1_tC_x^{\varepsilon}}+\|d^{(c)}_{q+1}\|_{L^1_tC_x^{\varepsilon}}+\|d^{(o)}_{q+1}\|_{L^1_tC_x^{\varepsilon}}.\\ 
    \end{aligned}
\end{equation}
Furthermore, by \eqref{eq2.2s}, Proposition \ref{prop-estimate-perturbation} and interpolation,
\begin{equation}\label{eq4.37}
    \begin{aligned}      \|w^{(p)}_{q+1}\|_{L^1_tC_x^{\varepsilon}}+\|d^{(p)}_{q+1}\|_{L^1_tC_x^{\varepsilon}}
        &\lesssim\|w^{(p)}_{q+1}\|^{1-\varepsilon}_{L^1_tC_x}\|w^{(p)}_{q+1}\|^{\varepsilon}_{L^1_tC_x^1}+\|d^{(p)}_{q+1}\|^{1-\varepsilon}_{L^1_tC_x}\|d^{(p)}_{q+1}\|_{L^1_tC_x^1}^{\varepsilon}\\
        &\lesssim l^{-5}\lambda_{q+1}^{\varepsilon}\tau^{-\frac{1}{2}}\\
        &\lesssim l^{-5}\lambda_{q+1}^{-14\varepsilon}.
    \end{aligned}
\end{equation}
Similarly, we get
\begin{align}  &\|w^{(c)}_{q+1}\|_{L^1_tC_x^{\varepsilon}}+\|d^{(c)}_{q+1}\|_{L^1_tC_x^{\varepsilon}}\lesssim l^{-15}\lambda_{q+1}^{\varepsilon-1}\tau^{-\frac{1}{2}}\lesssim l^{-15}\lambda_{q+1}^{-1-14\varepsilon},\label{eq4.38}\\
&\|w^{(o)}_{q+1}\|_{L^1_tC_x^{\varepsilon}}+\|d^{(o)}_{q+1}\|_{L^1_tC_x^{\varepsilon}}\lesssim l^{-18-10\varepsilon}\sigma^{-1}\lesssim l^{-28}\lambda_{q+1}^{-15\varepsilon}.\label{eq4.39}
\end{align}
By combining with \eqref{eq4.36}--\eqref{eq4.39}, we obtain
\begin{equation}\label{eq4.40}  \|w_{q+1}\|_{L^1_tC_x^{\varepsilon}}+\|d_{q+1}\|_{L^1_tC_x^{\varepsilon}}\lesssim l^{-5}\lambda_{q+1}^{-14\varepsilon}.
\end{equation}

Concerning the density, by \eqref{eq2.22}, \eqref{eq3.58} \eqref{eq3.64} and \eqref{eq4.26}, 
\begin{equation}\label{eq4.41}
    \begin{aligned}
        \|\rho_{q+1}^{-1}-\rho_l^{-1}\|_{C_tC^N_x}
        &\lesssim\|\rho_{q+1}^{-1}\|_{C_tC^N_x}\|\rho_l^{-1}\|_{C_tC^N_x}\|z_{q+1}\|_{C_tC^N_x}\\
        &\lesssim\lambda_{q+1}^{\frac{N\varepsilon}{4}}\lambda_q^{\frac{N\varepsilon}{4}}l^{-10N-28}\sigma^{-1}\\
        &\lesssim l^{-10N-29}\lambda_{q+1}^{-14\varepsilon}.
    \end{aligned}
\end{equation}
Furthermore, by interpolation, we obtain
\begin{align}   &\|\rho_l^{-1}\|_{C_tC^{\varepsilon}_x}\lesssim\|\rho_l^{-1}\|_{C_{t,x}}^{1-\varepsilon}\|\rho_l^{-1}\|_{C_tC_x^1}^{\varepsilon}\lesssim\lambda_q^{\frac{\varepsilon^2}{4}},\label{eq4.42}\\
&\|\rho_{q+1}^{-1}\|_{C_tC^{\varepsilon}_x}\lesssim\|\rho_{q+1}^{-1}\|_{C_{t,x}}^{1-\varepsilon}\|\rho_{q+1}^{-1}\|_{C_tC_x^1}^{\varepsilon}\lesssim\lambda_{q+1}^{\frac{\varepsilon^2}{4}},\label{eq4.43}\\
&\|z_{q+1}\|_{C_tC^{\varepsilon}_x}\lesssim\|z_{q+1}\|_{C_{t,x}}^{1-\varepsilon}\|z_{q+1}\|_{C_tC_x^1}^{\varepsilon}\lesssim l^{-28-10\varepsilon}\sigma^{-1}\lesssim l^{-29}\lambda_{q+1}^{-15\varepsilon},\label{eq4.44}
\end{align}
and the mollification of momentum an magnetic field
\begin{equation}\label{eq4.45}
    \|m_l\|_{C_tC^{\varepsilon}_x}+\|b_l\|_{C_tC^{\varepsilon}_x}\lesssim\|m_l\|_{C_{t,x}}^{1-\varepsilon}\|m_l\|_{C_tC_x^1}^{\varepsilon}+\|b_l\|_{C_{t,x}}^{1-\varepsilon}\|b_l\|_{C_tC_x^1}^{\varepsilon}\lesssim \lambda_q^{4}.
\end{equation}
And then, using \eqref{eq4.41}--\eqref{eq4.44}, by interpolation again,
\begin{equation}\label{eq4.46}
    \begin{aligned}
        \|\rho_{q+1}^{-1}-\rho_l^{-1}\|_{C_tC^{\varepsilon}_x}      &\lesssim\|\rho_{q+1}^{-1}\|_{C_tC^{\varepsilon}_x}\|\rho_l^{-1}\|_{C_tC^{\varepsilon}_x}\|z_{q+1}\|_{C_tC^{\varepsilon}_x}\\
        &\lesssim\lambda_{q+1}^{\frac{\varepsilon^2}{4}}\lambda_q^{\frac{\varepsilon^2}{4}}l^{-29}\sigma^{-1}\\
        &\lesssim l^{-29}\lambda_{q+1}^{-14\varepsilon}.
    \end{aligned}
\end{equation}

Hence, substituting \eqref{eq4.40} and \eqref{eq4.42}--\eqref{eq4.46} into \eqref{eq4.34} and \eqref{eq4.35}, using Proposition \ref{prop-estimate-perturbation}, we have
\begin{equation}\label{eq4.47}
    \begin{aligned}
        \left\|R^m_{{\rm lin}, 2}\right\|_{L_t^1 C_x^{\varepsilon}}+\left\|R^b_{{\rm lin}, 2}\right\|_{L_t^1 C_x^{\varepsilon}}
        &\lesssim\lambda_q^{\frac{\varepsilon^2}{4}}l^{-5}\tau^{-\frac{1}{2}}+l^{-5}\lambda_{q+1}^{-14\varepsilon}+l^{-29}\lambda_{q+1}^{-14\varepsilon}\left(\lambda_q^{4}+l^{-5}\lambda_{q+1}^{-14\varepsilon}\right)\\
        &\lesssim l^{-34}\lambda_{q+1}^{-14\varepsilon},
    \end{aligned}
\end{equation}
where the last step is due to $\alpha\geq20\varepsilon$.

On the other hand, regarding the case where $\alpha_1,\alpha_2\in(\frac{1}{2},1)$, we first deduce
\begin{equation}\label{eq4.48}
    \begin{aligned}
        \left\|R^m_{{\rm lin},2}\right\|_{L^1_tC^{\varepsilon}_x}
        &\lesssim\left\||\nabla|^{2\alpha_1-1}(\rho_l^{-1}w_{q+1})\right\|_{L^1_tC^{\varepsilon}_x}+\left\||\nabla|^{2\alpha_1-1}((\rho_{q+1}^{-1}-\rho_l^{-1})m_l)\right\|_{L^1_tC^{\varepsilon}_x}\\
        &\quad+\left\||\nabla|^{2\alpha_1-1}((\rho_{q+1}^{-1}-\rho_l^{-1})w_{q+1})\right\|_{L^1_tC^{\varepsilon}_x},
    \end{aligned}
\end{equation}
and
\begin{equation}\label{eq4.49} \left\|R^b_{{\rm lin},2}\right\|_{L^1_tC^{\varepsilon}_x}\lesssim\left\||\nabla|^{2\alpha_2-1}d_{q+1}\right\|_{L^1_tC^{\varepsilon}_x}.
\end{equation}
Next, by interpolation, \eqref{eq2.2s}, \eqref{eq2.22} and Proposition \ref{prop-estimate-perturbation}, we derive
\begin{equation}\label{eq4.50}
    \begin{aligned}
        &\quad\left\||\nabla|^{2\alpha_1-1}(\rho_l^{-1}w^{(p)}_{q+1})\right\|_{L^1_tC^{\varepsilon}_x}\\
        &\lesssim \left\|\rho_l^{-1}w_{q+1}^{(p)}\right\|^{1-\frac{2\alpha_1-1+\varepsilon}{2}}_{L^1_tC_x}\left\|\rho_l^{-1}w_{q+1}^{(p)}\right\|^{\frac{2\alpha_1-1+\varepsilon}{2}}_{L^1_tC^2_x}\\
        &\lesssim\left(\left\|\rho_l^{-1}\right\|_{C_{t,x}}\left\|w_{q+1}^{(p)}\right\|_{L^1_tC_x}\right)^{1-\frac{2\alpha_1-1+\varepsilon}{2}}\left(\left\|\rho_l^{-1}\right\|_{C_tC_x^2}\left\|w_{q+1}^{(p)}\right\|_{L^1_tC^2_x}\right)^{\frac{2\alpha_1-1+\varepsilon}{2}}\\
        &\lesssim\left(l^{-5}\tau^{-\frac{1}{2}}\right)^{1-\frac{2\alpha_1-1+\varepsilon}{2}}\left(\lambda_q^{\frac{\varepsilon}{2}}l^{-5}\lambda_{q+1}^2\tau^{-\frac{1}{2}}\right)^{\frac{2\alpha_1-1+\varepsilon}{2}}\\
        &\lesssim l^{-6}\tau^{-\frac{1}{2}}\lambda_{q+1}^{2\alpha_1-1+\varepsilon}\\
        &\lesssim l^{-6}\lambda_{q+1}^{-14\varepsilon},
    \end{aligned}
\end{equation}
and
\begin{equation}\label{eq4.51}
    \begin{aligned}
 \left\||\nabla|^{2\alpha_2-1}d^{(p)}_{q+1}\right\|_{L^1_tC^{\varepsilon}_x}
        &\lesssim \left\|d_{q+1}^{(p)}\right\|^{1-\frac{2\alpha_2-1+\varepsilon}{2}}_{L^1_tC_x}\left\|d_{q+1}^{(p)}\right\|^{\frac{2\alpha_2-1+\varepsilon}{2}}_{L^1_tC^2_x}\\
        &\lesssim\left(l^{-5}\tau^{-\frac{1}{2}}\right)^{1-\frac{2\alpha_2-1+\varepsilon}{2}}\left(l^{-5}\lambda_{q+1}^2\tau^{-\frac{1}{2}}\right)^{\frac{2\alpha_2-1+\varepsilon}{2}}\\
        &\lesssim l^{-5}\tau^{-\frac{1}{2}}\lambda_{q+1}^{2\alpha_2-1+\varepsilon}\\
        &\lesssim l^{-5}\lambda_{q+1}^{-14\varepsilon}.
    \end{aligned}
\end{equation}
Similarly, we also obtain
\begin{align}
    &\left\||\nabla|^{2\alpha_1-1}(\rho_l^{-1}w^{(c)}_{q+1})\right\|_{L^1_tC^{\varepsilon}_x}\lesssim\left(l^{-15}\lambda_{q+1}^{-1}\tau^{-\frac{1}{2}}\right)^{1-\frac{2\alpha_1-1+\varepsilon}{2}}\left(\lambda_q^{\frac{\varepsilon}{2}}l^{-15}\lambda_{q+1}\tau^{-\frac{1}{2}}\right)^{\frac{2\alpha_1-1+\varepsilon}{2}}\lesssim l^{-16}\lambda_{q+1}^{-1},\label{eq4.52}\\
    &\left\||\nabla|^{2\alpha_1-1}(\rho_l^{-1}w^{(o)}_{q+1})\right\|_{L^1_tC^{\varepsilon}_x}\lesssim\left(l^{-18}\sigma^{-1}\right)^{1-\frac{2\alpha_1-1+\varepsilon}{2}}\left(\lambda_q^{\frac{\varepsilon}{2}}l^{-38}\sigma^{-1}\right)^{\frac{2\alpha_1-1+\varepsilon}{2}}\lesssim l^{-38}\sigma^{-1},\label{eq4.53}\\
    &\left\||\nabla|^{2\alpha_2-1}d^{(c)}_{q+1}\right\|_{L^1_tC^{\varepsilon}_x}\lesssim\left(l^{-15}\lambda_{q+1}^{-1}\tau^{-\frac{1}{2}}\right)^{1-\frac{2\alpha_2-1+\varepsilon}{2}}\left(l^{-15}\lambda_{q+1}\tau^{-\frac{1}{2}}\right)^{\frac{2\alpha_2-1+\varepsilon}{2}}\lesssim l^{-15}\lambda_{q+1}^{-1},\label{eq4.54}\\
    &\left\||\nabla|^{2\alpha_2-1}d^{(o)}_{q+1}\right\|_{L^1_tC^{\varepsilon}_x}\lesssim\left(l^{-18}\sigma^{-1}\right)^{1-\frac{2\alpha_2-1+\varepsilon}{2}}\left(l^{-38}\sigma^{-1}\right)^{\frac{2\alpha_2-1+\varepsilon}{2}}\lesssim l^{-38}\sigma^{-1}.\label{eq4.55}
\end{align}
Hence, combining \eqref{eq4.50}--\eqref{eq4.55} and the fact that $l^{-38}\ll\lambda_{q+1}^{2\varepsilon}$ altogether we get
\begin{align}
    &\left\||\nabla|^{2\alpha_1-1}(\rho_l^{-1}w_{q+1})\right\|_{L^1_tC^{\varepsilon}_x}\lesssim l^{-6}\lambda_{q+1}^{-14\varepsilon}+l^{-16}\lambda^{-1}_{q+1}+l^{-38}\sigma^{-1}\lesssim l^{-16}\lambda_{q+1}^{-13\varepsilon},\label{eq4.56}\\
    &\left\||\nabla|^{2\alpha_2-1}d_{q+1}\right\|_{L^1_tC^{\varepsilon}_x}\lesssim l^{-5}\lambda_{q+1}^{-14\varepsilon}+l^{-15}\lambda^{-1}_{q+1}+l^{-38}\sigma^{-1}\lesssim l^{-15}\lambda_{q+1}^{-13\varepsilon}.\label{eq4.57}
\end{align}

For the remaining terms on the right-hand side of \eqref{eq4.48}, by interpolation and \eqref{eq2.3s}, \eqref{eq2.22}, \eqref{eq2.24}, \eqref{eq4.41}, Proposition \ref{prop-estimate-perturbation}, and the similar technique of \eqref{eq4.50}, we also have
\begin{align}
&\left\||\nabla|^{2\alpha_1-1}((\rho_{q+1}^{-1}-\rho_l^{-1})m_l)\right\|_{L^1_tC^{\varepsilon}_x}\lesssim l^{-52}\lambda_{q+1}^{-14\varepsilon},\label{eq4.58}\\
     &\left\||\nabla|^{2\alpha_1-1}((\rho_{q+1}^{-1}-\rho_l^{-1})w_{q+1})\right\|_{L^1_tC^{\varepsilon}_x}\lesssim l^{-87}\lambda^{-28\varepsilon}_{q+1}.\label{eq4.59}
\end{align}

Thus, plugging \eqref{eq4.56}--\eqref{eq4.59} into \eqref{eq4.48} and \eqref{eq4.49} we obtain that for $\alpha\in(\frac{1}{2},1)$,
\begin{equation}\label{eq4.60}
 \begin{aligned}        &\quad\left\|R^m_{{\rm lin},2}\right\|_{L^1_tC^{\varepsilon}_x}+\left\|R^b_{{\rm lin},2}\right\|_{L^1_tC^{\varepsilon}_x}\\
 &\lesssim l^{-16}\lambda_{q+1}^{-13\varepsilon}+l^{-15}\lambda_{q+1}^{-13\varepsilon}+l^{-52}\lambda_{q+1}^{-14\varepsilon}+l^{-87}\lambda_{q+1}^{-28\varepsilon}\\
 &\lesssim l^{-16}\lambda_{q+1}^{-12\varepsilon}.
 \end{aligned}
\end{equation}
Furthermore, together with \eqref{eq4.47}, we obtain the following estimates for $\alpha\in(0,1)$:
\begin{equation}\label{eq4.61}
\left\|R^m_{{\rm lin},2}\right\|_{L^1_tC^{\varepsilon}_x}+\left\|R^b_{{\rm lin},2}\right\|_{L^1_tC^{\varepsilon}_x}\lesssim l^{-16}\lambda_{q+1}^{-12\varepsilon}.
\end{equation}

Now we consider the bulk viscosity term $R^m_{{\rm lin},3}$ in \eqref{eq4.13}, by a direct calculation and \eqref{eq3.49},
\begin{equation}\label{eq4.62}
    \begin{aligned}
        \left\|R^m_{{\rm lin},3}\right\|_{L^1_tC^{\varepsilon}_x}     &\lesssim\left\|\nabla\rho_l^{-1}\right\|_{C_tC^{\varepsilon}_x}\left\|w_{q+1}\right\|_{L^1_tC^{\varepsilon}_x}+\left\|\nabla(\rho^{-1}_{q+1}-\rho_l^{-1})\right\|_{C_tC^{\varepsilon}_x}\left(\left\|m_l\right\|_{L^1_tC^{\varepsilon}_x}+\left\|w_{q+1}\right\|_{L^1_tC^{\varepsilon}_x}\right)\\
        &\quad+\left\|\rho^{-1}_{q+1}-\rho_l^{-1}\right\|_{C_tC^{\varepsilon}_x}\left(\left\|{\rm div}m_l\right\|_{L^1_tC^{\varepsilon}_x}+\left\|{\rm div}w^{(o)}_{q+1}\right\|_{L^1_tC^{\varepsilon}_x}\right)\\
        &\quad+\left\|\rho_l^{-1}\right\|_{C_tC^{\varepsilon}_x}\left\|{\rm div}w_{q+1}^{(o)}\right\|_{L^1_tC^{\varepsilon}_x}.
    \end{aligned}
\end{equation}
By a similar technique of \eqref{eq4.42} and \eqref{eq4.46}, we get
\begin{align}  &\left\|\nabla\rho_l^{-1}\right\|_{C_tC^{\varepsilon}_x}\lesssim\lambda_{q}^{\frac{\varepsilon^2}{4}+\frac{\varepsilon}{4}}\ll\lambda_q^{\varepsilon},\label{eq4.63}\\
&\left\|\nabla(\rho^{-1}_{q+1}-\rho_l^{-1})\right\|_{C_tC^{\varepsilon}_x}\lesssim\lambda_q^{\frac{\varepsilon^2}{4}}\lambda_{q+1}^{\frac{\varepsilon^2}{4}}l^{-28-10\varepsilon}\sigma^{-1}\lesssim l^{-30}\lambda_{q+1}^{-14\varepsilon}.\label{eq4.64}
\end{align}
Moreover, using \eqref{eq2.24}, \eqref{eq3.62} and interpolation inequality, we have
\begin{equation}\label{eq4.65}
    \left\|{\rm div}m_l\right\|_{L^1_tC^{\varepsilon}_x}\lesssim\left\|m_l\right\|_{C_tC^1_x}^{1-\varepsilon}\left\|m_l\right\|_{C_tC^2_x}^{\varepsilon}\lesssim\left(\lambda_q^4\right)^{1-\varepsilon}\left(l^{-1}\lambda_q^4\right)^{\varepsilon}\lesssim l^{-\varepsilon}\lambda_q^4,
\end{equation}
and
\begin{equation}\label{eq4.66}
    \begin{aligned}
        \left\|{\rm div}w^{(o)}_{q+1}\right\|_{L^1_tC^{\varepsilon}_x}
        &\lesssim\left\|{\rm div}w^{(o)}_{q+1}\right\|_{L^1_tC_x}^{1-\varepsilon}\left\|{\rm div}w^{(o)}_{q+1}\right\|_{L^1_tC^1_x}^{\varepsilon}\\
        &\lesssim\left(l^{-18}\sigma^{-1}\right)^{1-\varepsilon}\left(l^{-28}\sigma^{-1}\right)^{\varepsilon}\\
        &\lesssim l^{-28}\lambda_{q+1}^{-15\varepsilon}.
    \end{aligned}
\end{equation}
Hence, substituting \eqref{eq4.40}, \eqref{eq4.45}, \eqref{eq4.46} and \eqref{eq4.63}--\eqref{eq4.66} into \eqref{eq4.62} and using \eqref{eq2.3s},
\begin{equation}\label{eq4.67}
    \begin{aligned}
        \left\|R^m_{{\rm lin},3}\right\|_{L^1_tC^{\varepsilon}_x}     &\lesssim\lambda_q^{\varepsilon}l^{-5}\lambda_{q+1}^{-14\varepsilon}+l^{-30}\lambda_{q+1}^{-14\varepsilon}\left(\lambda_q^4+l^{-5}\lambda_{q+1}^{-14\varepsilon}\right)\\
        &\quad+l^{-28}\lambda_{q+1}^{-15\varepsilon}+l^{-29}\lambda_{q+1}^{-14\varepsilon}\left(l^{-\varepsilon}\lambda^{4}_q+l^{-28}\lambda_{q+1}^{-15\varepsilon}\right)\\
        &\lesssim l^{-31}\lambda_{q+1}^{-14\varepsilon}.
    \end{aligned}
\end{equation}

For the remaining terms $R^m_{{\rm lin},4}$, $R^m_{{\rm lin},5}$, $R^b_{{\rm lin},3}$ and $R^b_{{\rm lin},4}$ in \eqref{eq4.6} and \eqref{eq4.13}, using \eqref{eq4.40}, \eqref{eq4.43}, \eqref{eq4.45} and \eqref{eq4.46}, by a direct calculation,
\begin{equation}\label{eq4.68}
    \begin{aligned}      &\quad\left\|R^m_{{\rm lin},4}\right\|_{L^1_tC^{\varepsilon}_x}+\left\|R^m_{{\rm lin},5}\right\|_{L^1_tC^{\varepsilon}_x}+\left\|R^b_{{\rm lin},3}\right\|_{L^1_tC^{\varepsilon}_x}+\left\|R^b_{{\rm lin},4}\right\|_{L^1_tC^{\varepsilon}_x}\\  &\lesssim\left\|\rho_{q+1}^{-1}\right\|_{C_tC^{\varepsilon}_x}\left(\left\|m_l\right\|_{C_tC^{\varepsilon}_x}+\left\|b_l\right\|_{C_tC^{\varepsilon}_x}\right)\left(\left\|d_{q+1}\right\|_{L^1_tC^{\varepsilon}_x}+\left\|w_{q+1}\right\|_{L^1_tC^{\varepsilon}_x}\right)\\
    &\quad+\left\|\rho_{q+1}^{-1}-\rho_l^{-1}\right\|_{C_tC^{\varepsilon}_x}\left(\left\|m_l\right\|_{C_tC^{\varepsilon}_x}+\left\|b_l\right\|_{C_tC^{\varepsilon}_x}\right)^2+\left\|b_l\right\|_{C_tC^{\varepsilon}_x}\left\|d_{q+1}\right\|_{L^1_tC^{\varepsilon}_x}\\
    &\lesssim \lambda_{q+1}^{\frac{\varepsilon^2}{4}}\lambda_q^4l^{-5}\lambda_{q+1}^{-14\varepsilon}+l^{-29}\lambda^{-14\varepsilon}_{q+1}\lambda_q^8+\lambda_q^4l^{-5}\lambda_{q+1}^{-14\varepsilon}\\
    &\lesssim l^{-30}\lambda_{q+1}^{-14\varepsilon}.
    \end{aligned}
\end{equation}

We now conclude from \eqref{eq4.32}, \eqref{eq4.33}, \eqref{eq4.61}, \eqref{eq4.67} and \eqref{eq4.68} altogether that
\begin{equation}\label{eq4.69}
    \begin{aligned}
\left\|R^m_{{\rm lin}}\right\|_{L^1_tC^{\varepsilon}_x}+\left\|R^b_{{\rm lin}}\right\|_{L^1_tC^{\varepsilon}_x}
&\lesssim l^{-27}\lambda_{q+1}^{-9\varepsilon}+l^{-13}\lambda_{q+1}^{-9\varepsilon} +l^{-16}\lambda_{q+1}^{-12\varepsilon}+l^{-31}\lambda_{q+1}^{-14\varepsilon}+l^{-30}\lambda_{q+1}^{-14\varepsilon}\\
&\lesssim l^{-27}\lambda_{q+1}^{-9\varepsilon}.
    \end{aligned}
\end{equation}
\subsubsection{Oscillation errors $R^m_{{\rm osc}}$ and $R^b_{{\rm osc}}$}
We now deal with the oscillation errors. First, we concern the high-low spatial oscillation errors $R^m_{{\rm osc},1}$, $R^m_{{\rm osc},2}$ and $R^b_{{\rm osc},1}$ in \eqref{eq4.7} and \eqref{eq4.14}. 

The key fact is that, the momentum and magnetic flows are of high oscillations while the amplitude function is slowly varying. Thus, we use the following stationary phase lemma whose proof is similar to \cite{DCS-2013-invent}, here we omit the details.
\begin{lem}\label{lem-stationary-phase}
    Let $\varepsilon\in(0,1)$ and $m\geq1$. Assume that $\theta\in C^{m,\varepsilon}(\mathbb{T}^3)$, then we have
    \begin{align*}
        \left\||\nabla|^{-1}\left(\theta(x)e^{i\lambda\xi\cdot x}\right)\right\|_{C^{\varepsilon}_x}\lesssim\frac{\|\theta\|_{C_x}}{\lambda^{1-\varepsilon}}+\frac{\|\theta\|_{C^{m,\varepsilon}_x}}{\lambda^{m-\varepsilon}},
    \end{align*}
    where $\lambda\xi\in\mathbb{Z}^3$ and the implicit constant is independent of $q$.
\end{lem}

Note that 
\begin{align*}
    &\mathbb{P}_{\neq0}\left(W_{(k)}\otimes W_{(k)}\right)=\mathbb{P}_{\geq(\lambda/2)}\left(W_{(k)}\otimes W_{(k)}\right)=\mathbb{P}_{\geq(\lambda/2)}(\phi_{(k)}^2)k_1\otimes k_1,\\
    &\mathbb{P}_{\neq0}\left(D_{(k)}\otimes D_{(k)}\right)=\mathbb{P}_{\geq(\lambda/2)}\left(D_{(k)}\otimes D_{(k)}\right)=\mathbb{P}_{\geq(\lambda/2)}(\phi_{(k)}^2)k_2\otimes k_2,\\
    &\mathbb{P}_{\neq0}\left(|D_{(k)}|^2\mathbb{I}\right)=\mathbb{P}_{\geq(\lambda/2)}\left(|D_{(k)}|^2\mathbb{I}\right)=\mathbb{P}_{\geq(\lambda/2)}(\phi_{(k)}^2)\mathbb{I},
\end{align*}
and
\begin{align*}
    &\quad\mathbb{P}_{\neq0}\left(D_{(k)}\otimes W_{(k)}-W_{(k)}\otimes D_{(k)}\right)\\
    &=\mathbb{P}_{\geq(\lambda/2)}\left(D_{(k)}\otimes W_{(k)}-W_{(k)}\otimes D_{(k)}\right)\\
    &=\mathbb{P}_{\geq(\lambda/2)}(\phi_{(k)}^2)(k_2\otimes k_1-k_1\otimes k_2),
\end{align*}
furthermore, $\{\phi_{(k)}^2\}_{k\in\Lambda_m\cup\Lambda_b}$ are $(\mathbb{T}/\lambda)^3$-periodic functions, we can decompose
\begin{align}\label{eq4.70}
    \mathbb{P}_{\geq(\lambda/2)}(\phi_{(k)}^2)=\sum_{\xi\in\mathbb{Z}^3\backslash\{0\}}f_{k}(\xi)e^{i\lambda\xi\cdot x},
\end{align}
where $f_k(\xi)$ are the Fourier coefficients of $\mathbb{P}_{\geq(\lambda/2)}(\phi_{(k)}^2)$, and decay faster than arbitrary polynomials:
\begin{align*}
    |f_k(\xi)|\leq C|\xi|^{-m},
\end{align*}
for any $m\in\mathbb{N}$. Moreover, the constant $C$ is independent of $q$.

Thus, using \eqref{eq4.70}, Lemmas \ref{lem-magnetic-amplitude-estimate}, \ref{lem-momentum-amplitude-estimate} and \ref{lem-stationary-phase} we have
\begin{equation}\label{eq4.71}
    \begin{aligned}
&\quad\left\|R^m_{{\rm osc},1}\right\|_{L_t^1 C_x^{\varepsilon}}+\left\|R^m_{{\rm osc},2}\right\|_{L_t^1 C_x^{\varepsilon}} +\left\|R^b_{{\rm osc},1}\right\|_{L_t^1 C_x^{\varepsilon}}\\
& \lesssim \sum_{k \in \Lambda_m\cup\Lambda_b}\left\|g_{(k)}\right\|_{L_t^2}^2\left\|\mathcal{R}^m \mathbb{P}_{\neq 0}\left(\mathbb{P}_{\geq(\lambda / 2)}\left(W_{(k)} \otimes W_{(k)}\right) \nabla\left(\rho_{l}^{-1} a_{(k)}^2\right)\right)\right\|_{C_t C_x^{\varepsilon}} \\
&\quad+\sum_{k \in \Lambda_b}\left\|g_{(k)}\right\|_{L_t^2}^2\left\|\mathcal{R}^m \mathbb{P}_{\neq 0}\left(\mathbb{P}_{\geq(\lambda / 2)}\left(D_{(k)} \otimes D_{(k)}-\frac{1}{2}|D_{(k)}|^2\mathbb{I}\right) \nabla\left(a_{(k)}^2\right)\right)\right\|_{C_t C_x^{\varepsilon}}\\
&\quad+\sum_{k \in \Lambda_b}\left\|g_{(k)}\right\|_{L_t^2}^2\left\|\mathcal{R}^b \mathbb{P}_{\neq 0}\left(\mathbb{P}_{\geq(\lambda / 2)}\left(D_{(k)} \otimes W_{(k)}-W_{(k)} \otimes D_{(k)}\right) \nabla\left(\rho_{l}^{-1} a_{(k)}^2\right)\right)\right\|_{C_t C_x^{\varepsilon}} \\
& \lesssim \sum_{k \in \Lambda_m\cup\Lambda_b} \sum_{\xi \in \mathbb{Z}^3 \backslash\{0\}}|\xi|^{-10} \lambda^{\varepsilon-1}\left(\left\|\nabla\left(\rho_{l}^{-1} a_{(k)}^2\right)\right\|_{C_{t, x}}+\left\|\nabla\left(\rho_{l}^{-1} a_{(k)}^2\right)\right\|_{C_t C_x^{1, {\varepsilon}}}\right) \\
&\quad+\sum_{k \in \Lambda_b} \sum_{\xi \in \mathbb{Z}^3 \backslash\{0\}}|\xi|^{-10} \lambda^{\varepsilon-1}\left(\left\|\nabla\left(a_{(k)}^2\right)\right\|_{C_{t, x}}+\left\|\nabla\left(a_{(k)}^2\right)\right\|_{C_t C_x^{1, {\varepsilon}}}\right) \\
& \lesssim \sum_{k \in \Lambda_m\cup\Lambda_b} \lambda_{q+1}^{\varepsilon-1}\left(\left\|\rho_{l}^{-1}\right\|_{C_t C_x^3}+1\right)\left\|a_{(k)}^2\right\|_{C_{t, x}^3} \\
& \lesssim \lambda_{q+1}^{\varepsilon-1}\left(\lambda_q^{\frac{3\varepsilon}{4}}+1\right)l^{-37}\\
&\lesssim l^{-38}\lambda_{q+1}^{\varepsilon-1}.
\end{aligned}
\end{equation}

Regarding the temporal oscillation errors $R^m_{{\rm osc},3}$, $R^m_{{\rm osc},4}$ and $R^b_{{\rm osc},2}$ in \eqref{eq4.7} and \eqref{eq4.14}, we use \eqref{eq2.21}, \eqref{eq2.22}, \eqref{eq3.15}, \eqref{eq3.27} and \eqref{eq3.39} to estimate
\begin{equation}\label{eq4.72}
    \begin{aligned}
&\quad\left\|R^m_{{\rm osc},3}\right\|_{L_t^1 C_x^{\varepsilon}}+\left\|R^m_{{\rm osc},4}\right\|_{L_t^1 C_x^{\varepsilon}}+\left\|R^b_{{\rm osc},2}\right\|_{L_t^1 C_x^{\varepsilon}}\\
& \lesssim \sigma^{-1} \sum_{k \in \Lambda_m\cup\Lambda_b}\left\|h_{(k)}\right\|_{C_t}\left\|\partial_t \nabla\left(\rho_{l}^{-1} a_{(k)}^2\right)\right\|_{L_t^1 C_x^{\varepsilon}} +\sigma^{-1} \sum_{k \in\Lambda_b}\left\|h_{(k)}\right\|_{C_t}\left\|\partial_t \nabla\left(a_{(k)}^2\right)\right\|_{L_t^1 C_x^{\varepsilon}}\\
& \lesssim \sigma^{-1} \sum_{k \in \Lambda_m\cup\Lambda_b}\left\|h_{(k)}\right\|_{C_t}\left\|\partial_t \nabla\left(\rho_{l}^{-1} a_{(k)}^2\right)\right\|_{C_{t, x}}^{1-\varepsilon}\left\|\partial_t \nabla\left(\rho_{l}^{-1} a_{(k)}^2\right)\right\|_{C_{t, x}^1}^{\varepsilon} \\
&\quad+\sigma^{-1} \sum_{k \in\Lambda_b}\left\|h_{(k)}\right\|_{C_t}\left\|\partial_t \nabla\left(a_{(k)}^2\right)\right\|_{C_{t, x}}^{1-\varepsilon}\left\|\partial_t \nabla\left(a_{(k)}^2\right)\right\|_{C_{t, x}^1}^{\varepsilon} \\
& \lesssim \sigma^{-1} \sum_{k \in \Lambda_m\cup\Lambda_b}\left\|h_{(k)}\right\|_{C_t}\left(\left\|\rho_{l}^{-1}\right\|_{C_{t, x}^2}\left\|a_{(k)}^2\right\|_{C_{t, x}^2}\right)^{1-\varepsilon}\left(\left\|\rho_{l}^{-1}\right\|_{C_{t, x}^3}\left\|a_{(k)}^2\right\|_{C_{t, x}^3}\right)^{\varepsilon}\\
&\quad+\sigma^{-1} \sum_{k \in \Lambda_b}\left\|h_{(k)}\right\|_{C_t}\left\|a_{(k)}^2\right\|_{C_{t, x}^2}^{2-\varepsilon}\left\|a_{(k)}^2\right\|_{C_{t, x}^3}^{3\varepsilon}\\
&\lesssim \sigma^{-1}l^{-28(1-\varepsilon)-38\varepsilon}\\
&\lesssim l^{-29}\lambda_{q+1}^{-15\varepsilon}.
\end{aligned}
\end{equation}

Finally, for the density errors $R^m_{{\rm osc},5}$ and $R^b_{{\rm osc},3}$ in \eqref{eq4.7} and \eqref{eq4.14}, by \eqref{eq4.46}, Proposition \ref{prop-estimate-perturbation} and interpolation,
\begin{equation}\label{eq4.73}
    \begin{aligned}
        &\quad\left\|R^m_{{\rm osc},5}\right\|_{L_t^1 C_x^{\varepsilon}}+\left\|R^b_{{\rm osc},3}\right\|_{L_t^1 C_x^{\varepsilon}}\\
        &\lesssim\left\|\rho_{q+1}^{-1}-\rho_l^{-1}\right\|_{C_tC^{\varepsilon}_x}\left(\left\|w_{q+1}\right\|^2_{L^2_tC^{\varepsilon}_x}+\left\|w_{q+1}\right\|_{L^2_tC^{\varepsilon}_x}\left\|d_{q+1}\right\|_{L^2_tC^{\varepsilon}_x}\right)\\
        &\lesssim l^{-29}\lambda_{q+1}^{-14\varepsilon}l^{-10}\lambda_{q+1}^{2\varepsilon}\\
        &\lesssim l^{-39}\lambda_{q+1}^{-12\varepsilon}.
    \end{aligned}
\end{equation}

Thus, combining \eqref{eq4.71}--\eqref{eq4.73} altogether we conclude
\begin{equation}\label{eq4.74}
    \left\|R^m_{{\rm osc}}\right\|_{L_t^1 C_x^{\varepsilon}}+\left\|R^b_{{\rm osc}}\right\|_{L_t^1 C_x^{\varepsilon}}\lesssim l^{-38}\lambda_{q+1}^{\varepsilon-1}+l^{-29}\lambda_{q+1}^{-15\varepsilon}+l^{-39}\lambda_{q+1}^{-12\varepsilon}\lesssim l^{-29}\lambda_{q+1}^{-11\varepsilon}.
\end{equation}
\subsubsection{Corrector errors $R^m_{{\rm cor}}$ and $R^b_{{\rm cor}}$} Now we consider the corrector errors $R^m_{{\rm cor}}$ and $R^b_{{\rm cor}}$ in \eqref{eq4.8} and \eqref{eq4.15}, using \eqref{eq4.42}, interpolation inequality and Proposition \ref{prop-estimate-perturbation} we derive
\begin{equation}\label{eq4.75}
 \begin{aligned}
&\quad\left\|R^m_{{\rm cor}}\right\|_{L_t^1 C_x^{\varepsilon}}+\left\|R^b_{{\rm cor}}\right\|_{L_t^1 C_x^{\varepsilon}}\\  
& \lesssim\left\|\rho_{l}^{-1}\right\|_{C_t C_x^{\varepsilon}}\left\|w_{q+1}^{(c)}+w_{q+1}^{(o)}\right\|_{L_t^2 C_x^{\varepsilon}}\left(\left\|w_{q+1}^{(p)}\right\|_{L_t^2 C_x^{\varepsilon}}+\left\|w_{q+1}\right\|_{L_t^2 C_x^{\varepsilon}}\right)\\
&\quad+\left\|d_{q+1}^{(c)}+d_{q+1}^{(o)}\right\|_{L_t^2 C_x^{\varepsilon}}\left(\left\|d_{q+1}^{(p)}\right\|_{L_t^2 C_x^{\varepsilon}}+\left\|d_{q+1}\right\|_{L_t^2 C_x^{\varepsilon}}\right)\\
&\quad+\left\|\rho_{l}^{-1}\right\|_{C_t C_x^{\varepsilon}}\left(\left\|d_{q+1}^{(c)}+d_{q+1}^{(o)}\right\|_{L_t^2 C_x^{\varepsilon}}\left\|w_{q+1}\right\|_{L_t^2 C_x^{\varepsilon}}+\left\|w_{q+1}^{(c)}+w_{q+1}^{(o)}\right\|_{L_t^2 C_x^{\varepsilon}}\left\|d_{q+1}\right\|_{L_t^2 C_x^{\varepsilon}}\right)\\
&\lesssim\left(\lambda_q^{\frac{\varepsilon^2}{4}}+1\right)\left(l^{-15}\lambda_{q+1}^{\varepsilon-1}+l^{-18-10\varepsilon}\sigma^{-1}\right)l^{-5}\lambda_{q+1}^{\varepsilon}\\
&\lesssim l^{-24}\lambda_{q+1}^{-14\varepsilon}.
\end{aligned}
\end{equation}
\subsubsection{Pressure error $R^m_{{\rm pre}}$}
We will use the boundedness of Calder\'on-Zygmund operators, the mean value theorem and interpolation inequality to deal with the pressure error $R^m_{{\rm pre}}$ in \eqref{eq4.16}.

More  precisely,
\begin{equation}\label{eq4.76}
    \begin{aligned}
\left\|R^m_{{\rm pre}}\right\|_{L_t^1 C_x^{\varepsilon}} & \lesssim\left\|P\left(\rho_{q+1}\right)-P\left(\rho_{l}\right)\right\|_{C_t C_x^{\varepsilon}} \\
& \lesssim\left\|\int_{\rho_{l}}^{\rho_{q+1}} P^{\prime}(y) \mathrm{d} y\right\|_{C_{t, x}}^{1-\varepsilon}\left\|P\left(\rho_{q+1}\right)-P\left(\rho_{l}\right)\right\|_{C_t C_x^1}^{\varepsilon} \\
& \lesssim\left\|P^{\prime}(\zeta)\left(\rho_{q+1}-\rho_{l}\right)\right\|_{C_{t, x}}^{1-\varepsilon}\left(\left\|P\left(\rho_{q+1}\right)\right\|_{C_t C_x^1}+\left\|P\left(\rho_{l}\right)\right\|_{C_t C_x^1}\right)^{\varepsilon} \\
& \lesssim\left\|P^{\prime}(\zeta)\right\|_{C_{t, x}}^{1-\varepsilon}\left\|\rho_{q+1}-\rho_{l}\right\|_{C_{t, x}}^{1-\varepsilon}\left(\left\|P\left(\rho_{q+1}\right)\right\|_{C_{t, x}}+\left\|P\left(\rho_{l}\right)\right\|_{C_{t, x}}\right. \\
& \left.\qquad+\left\|P^{\prime}\left(\rho_{q+1}\right)\right\|_{C_{t, x}}\left\|\rho_{q+1}\right\|_{C_t C_x^1}+\left\|P^{\prime}\left(\rho_{l}\right)\right\|_{C_{t, x}}\left\|\rho_{l}\right\|_{C_t C_x^1}\right)^{\varepsilon},
\end{aligned}
\end{equation}
where $\min\{\rho_{q+1},\rho_l\}\leq\zeta\leq\max\{\rho_{q+1},\rho_l\}$.

Since $\rho_{q+1}$ and $\rho_l$ are uniformly away from zero and infinity, and $P$, $P^{\prime}$, $P^{\prime\prime}$ are all continuous, we have
\begin{equation}\label{eq4.77}
\left\|P^{\prime}(\zeta)\right\|_{C_{t, x}}+\left\|P^{\prime}\left(\rho_{l}\right)\right\|_{C_{t, x}}+\left\|P^{\prime}\left(\rho_{q+1}\right)\right\|_{C_{t, x}} \lesssim 1 .
\end{equation}
Thus, we conclude the following estimates by using \eqref{eq2.6} at level $q+1$, \eqref{eq2.22}, \eqref{eq3.64} and \eqref{eq4.41}:
\begin{equation}\label{eq4.78}
    \left\|R^m_{{\rm pre}}\right\|_{L_t^1 C_x^{\varepsilon}}\lesssim \left(l^{-29}\lambda_{q+1}^{-14\varepsilon}\right)^{1-\varepsilon}\left(1+\lambda_{q+1}^{\frac{\varepsilon}{4}}+\lambda_{q}^{\frac{\varepsilon}{4}}\right)^{\varepsilon}\lesssim l^{-29}\lambda_{q+1}^{-13\varepsilon}.
\end{equation}
\subsubsection{Commutator errors $R^m_{{\rm com}}$ and $R^b_{{\rm com}}$}
It remains to deal with the commutator errors $R^m_{{\rm com}}$ and $R^b_{{\rm com}}$ given by\eqref{eq2.19} and \eqref{eq2.20}.

To begin with, concerning the shear viscous commutator $R^m_{{\rm com},1}$. For the case $\alpha_1\in(\frac{1}{2},1)$, using interpolation inequality will leads 
\begin{equation}\label{eq4.79}
\left\|R^m_{{\rm com},1}\right\|_{L_t^1 C_x^{\varepsilon}} \lesssim\left\|\rho_{l}^{-1} m_{l}-\left(\rho_q^{-1} m_q\right) *_x \phi_{l} *_t \varphi_{l}\right\|_{C_{t, x}}^{1-\frac{2 \alpha_1-1+\varepsilon}{2}}\left\|\rho_{l}^{-1} m_{l}-\left(\rho_q^{-1} m_q\right) *_x \phi_{l} *_t \varphi_{l}\right\|_{C_t C_x^2}^{\frac{2 \alpha_1-1+\varepsilon}{2}} .
\end{equation}
Furthermore, by \eqref{eq2.6}, \eqref{eq2.7}, \eqref{eq2.25}, \eqref{eq4.26}, and using \eqref{eq4.41} at level $q$,
\begin{equation}\label{eq4.80}
\begin{aligned}
& \quad\left\|\rho_{l}^{-1} m_{l}-\left(\rho_q^{-1} m_q\right) *_x \phi_{l} *_t \varphi_{l}\right\|_{C_{t, x}} \\
&\lesssim \left\|\rho_{l}^{-1} m_{l}-\rho_{l}^{-1} m_q\right\|_{C_{t, x}}+\left\|\rho_{l}^{-1} m_q-\rho_q^{-1} m_q\right\|_{C_{t, x}}+\left\|\rho_q^{-1} m_q-\left(\rho_q^{-1} m_q\right) *_x \phi_{l} *_t \varphi_{l}\right\|_{C_{t, x}} \\
 &\lesssim \left\|\rho_{l}^{-1}\right\|_{C_{t, x}}\left\|m_{l}-m_q\right\|_{C_{t, x}}+\left\|\rho_{l}^{-1}-\rho_q^{-1}\right\|_{C_{t, x}}\left\|m_q\right\|_{C_{t, x}}+l\left\|\rho_q^{-1} m_q\right\|_{C_{t, x}^1} \\
 &\lesssim l\left\|m_q\right\|_{C_{t, x}^1}+\left\|\rho_{l}^{-1}\right\|_{C_{t, x}}\left\|\rho_q^{-1}\right\|_{C_{t, x}}\left\|\rho_{l}-\rho_q\right\|_{C_{t, x}}\left\|m_q\right\|_{C_{t, x}}+l\left\|\rho_q^{-1}\right\|_{C_{t, x}^1}\left\|m_q\right\|_{C_{t, x}^1} \\
 &\lesssim l \lambda_q^4+l \lambda_q^{2+\frac{\varepsilon}{4}}+l \lambda_q^{4+\frac{\varepsilon}{4}} \lesssim l^{\frac{1}{2}},
\end{aligned}
\end{equation}
and
\begin{equation}\label{eq4.81}
\begin{aligned}
& \quad\left\|\rho_{l}^{-1} m_{l}-\left(\rho_q^{-1} m_q\right) *_x \phi_{l} *_t \varphi_{l}\right\|_{C_t C_x^2} \\
 &\lesssim \left\|\rho_{l}^{-1}\right\|_{C_t C_x^2}\left\|m_{l}-m_q\right\|_{C_t C_x^2}+\left\|\rho_{l}^{-1}-\rho_q^{-1}\right\|_{C_t C_x^2}\left\|m_q\right\|_{C_t C_x^2}+l\left\|\rho_q^{-1} m_q\right\|_{C_t C_x^3}+l\left\|\rho_q^{-1} m_q\right\|_{C_t^1 C_x^2} \\
 &\lesssim l \lambda_q^{\frac{\varepsilon}{2}}\left\|m_q\right\|_{C_{t, x}^3}+l^{-49} \lambda^{-14 \varepsilon} \lambda_q^6+l\left\|\rho_q^{-1}\right\|_{C_t C_x^3}\left\|m_q\right\|_{C_t C_x^3}+l\left\|\rho_q^{-1}\right\|_{C_t^1 C_x^2}\left\|m_q\right\|_{C_t^1 C_x^2} \\
 &\lesssim l \lambda_q^{8+\frac{\varepsilon}{2}}+l^{-49} \lambda_{q+1}^{-14 \varepsilon}+l \lambda_q^{8+\varepsilon} \lesssim l^{\frac{1}{2}}.
\end{aligned}
\end{equation}
For the case $\alpha_1\in(0,\frac{1}{2}]$, a similar argument yields that
\begin{equation}\label{eq4.82}
\left\|R^m_{{\rm com},1}\right\|_{L_t^1 C_x^{\varepsilon}} \lesssim\left\|\rho_{l}^{-1} m_{l}-\left(\rho_q^{-1} m_q\right) *_x \phi_{l} *_t \varphi_{l}\right\|_{C_{t, x}}^{1-\frac{\varepsilon}{2}}\left\|\rho_{l}^{-1} m_{l}-\left(\rho_q^{-1} m_q\right) *_x \phi_{l} *_t \varphi_{l}\right\|_{C_t C_x^2}^{\frac{\varepsilon}{2}}\lesssim l^{\frac{1}{2}} .
\end{equation}
Thus, we conclude from \eqref{eq4.79}--\eqref{eq4.82} that
\begin{equation}\label{eq4.83}
    \left\|R^m_{{\rm com},1}\right\|_{L_t^1 C_x^{\varepsilon}}\lesssim l^{\frac{1}{2}}. 
\end{equation}

Then, for the pressure commutator error $R^m_{{\rm com},2}$, similar to \eqref{eq4.78}, we have
\begin{equation}\label{eq4.84}
\begin{aligned}
\left\|R^m_{{\rm com},2}\right\|_{L_t^1 C_x^{\varepsilon}} & \lesssim\left\|P\left(\rho_{l}\right)-P_{l}\right\|_{C_{t, x}}^{1-\varepsilon}\left\|P\left(\rho_{l}\right)-P_{l}\right\|_{C_t C_x^1}^{\varepsilon} \\
& \lesssim\left(\left\|P\left(\rho_{l}\right)-P\left(\rho_q\right)\right\|_{C_{t, x}}+\left\|P\left(\rho_q\right)-P_{l}\right\|_{C_{t, x}}\right)^{1-\varepsilon}\left(\left\|P\left(\rho_q\right)\right\|_{C_t C_x^1}+\left\|P_{l}\right\|_{C_t C_x^1}\right)^{\varepsilon} \\
& \lesssim\left(\left\|\rho_{l}-\rho_q\right\|_{C_{t, x}}+l\left\|P\left(\rho_q\right)\right\|_{C_{t, x}^1}\right)^{1-\varepsilon}\left(\left\|P\left(\rho_q\right)\right\|_{C_{t, x}}+\left\|P^{\prime}\left(\rho_q\right)\right\|_{C_{t, x}}\left\|\rho_q\right\|_{C_t C_x^1}\right)^{\varepsilon}\\
&\lesssim\left(l\lambda_q^{\frac{\varepsilon}{4}}\right)^{1-\varepsilon}\left(1+\lambda_q^{\frac{\varepsilon}{4}}\right)^{\varepsilon}\lesssim l^{1-\varepsilon}\lambda_q^{\frac{\varepsilon}{4}}\lesssim l^{\frac{1}{2}}.
\end{aligned}
\end{equation}

Next, the bulk viscous commutator error $R^m_{{\rm com},3}$ can be estimated in a similar fashion as $R^m_{{\rm com},1}$:
\begin{equation}\label{eq4.85}
\left\|R^m_{{\rm com},3}\right\|_{L_t^1 C_x^{\varepsilon}} \lesssim\left\|\rho_{l}^{-1} m_{l}-\left(\rho_q^{-1} m_q\right) *_x \phi_{l} *_t \varphi_{l}\right\|_{C_{t, x}}^{1-\frac{1+\varepsilon}{2}}\left\|\rho_{l}^{-1} m_{l}-\left(\rho_q^{-1} m_q\right) *_x \phi_{l} *_t \varphi_{l}\right\|_{C_t C_x^2}^{\frac{1+\varepsilon}{2}}\lesssim l^{\frac{1}{2}} .
\end{equation}

At last, concerning the nonlinear commutator errors $R^m_{{\rm com},4}$, $R^m_{{\rm com},5}$ and $R^b_{{\rm com}}$, we see that
\begin{equation}\label{eq4.86}
\begin{aligned}
&\quad\left\|R^m_{{\rm com},4}\right\|_{L_t^1 C_x^{\varepsilon}}+\left\|R^m_{{\rm com},5}\right\|_{L_t^1 C_x^{\varepsilon}} +\left\|R^b_{{\rm com}}\right\|_{L_t^1 C_x^{\varepsilon}} \\ 
&\lesssim \left\|\rho_{l}^{-1} m_{l} \otimes m_{l}-\rho_q^{-1} m_q \otimes m_q\right\|_{L_t^1 C_x^{\varepsilon}}+\left\|\rho_q^{-1} m_q \otimes m_q-\left(\rho_q^{-1} m_q \otimes m_q\right) *_x \phi_{l} *_t \varphi_{l}\right\|_{L_t^1 C_x^{\varepsilon}}\\
&\quad+\left(\left\|b_{l} \otimes b_{l}-b_q \otimes b_q\right\|_{L_t^1 C_x^{\varepsilon}}+\left\||b_l|^2\mathbb{I}-|b_q|^2\mathbb{I}\right\|_{L_t^1 C_x^{\varepsilon}}\right)\\
&\quad+\left(\left\|b_q \otimes b_q-\left(b_q \otimes b_q\right) *_x \phi_{l} *_t \varphi_{l}\right\|_{L_t^1 C_x^{\varepsilon}}+\left\||b_q|^2\mathbb{I}-\left(|b_q|^2\mathbb{I}\right) *_x \phi_{l} *_t \varphi_{l}\right\|_{L_t^1 C_x^{\varepsilon}}\right)\\
&\quad+\left(\left\|\rho_{l}^{-1} b_{l} \otimes m_{l}-\rho_q^{-1} b_q \otimes m_q\right\|_{L_t^1 C_x^{\varepsilon}}+\left\|\rho_{l}^{-1} m_{l} \otimes b_{l}-\rho_q^{-1} m_q \otimes b_q\right\|_{L_t^1 C_x^{\varepsilon}}\right)\\
&\quad+\left(\left\|\rho_q^{-1} b_q \otimes m_q-\left(\rho_q^{-1} b_q \otimes m_q\right) *_x \phi_{l} *_t \varphi_{l}\right\|_{L_t^1 C_x^{\varepsilon}}\right.\\
&\qquad+\left.\left\|\rho_q^{-1} m_q \otimes b_q-\left(\rho_q^{-1} m_q \otimes b_q\right) *_x \phi_{l} *_t \varphi_{l}\right\|_{L_t^1 C_x^{\varepsilon}}\right)\\
&=:K_{11}+K_{12}+K_{21}+K_{22}+K_{31}+K_{32}.
\end{aligned}
\end{equation}
by a direct calculation we have
\begin{equation}\label{eq4.87}
\begin{aligned}
K_{11} & \lesssim\left\|\rho_{l}^{-1} m_{l} \otimes\left(m_{l}-m_q\right)\right\|_{L_t^1 C_x^{\varepsilon}}+\left\|\rho_{l}^{-1}\left(m_{l}-m_q\right) \otimes m_q\right\|_{L_t^1 C_x^{\varepsilon}}+\left\|\left(\rho_{l}^{-1}-\rho_q^{-1}\right) m_q \otimes m_q\right\|_{L_t^1 C_x^{\varepsilon}} \\
& \lesssim\left\|\rho_{l}^{-1}\right\|_{C_t C_x^{\varepsilon}}\left\|m_{l}-m_q\right\|_{L_t^1 C_x^{\varepsilon}}\left(\left\|m_{l}\right\|_{C_t C_x^{\varepsilon}}+\left\|m_q\right\|_{C_t C_x^{\varepsilon}}\right)+\left\|\rho_{l}^{-1}-\rho_q^{-1}\right\|_{C_t C_x^{\varepsilon}}\left\|m_q\right\|_{L_t^2 C_x^{\varepsilon}}^2 .
\end{aligned}
\end{equation}
Next, by interpolation inequality, \eqref{eq2.22}, \eqref{eq2.23} and \eqref{eq2.25},
\begin{equation}\label{eq4.88}
    \begin{aligned}
\left\|m_{l}-m_q\right\|_{L_t^1 C_x^{\varepsilon}} & \lesssim\left\|m_{l}-m_q\right\|_{C_{t, x}}^{1-\varepsilon}\left\|m_{l}-m_q\right\|_{C_{t, x}^1}^{\varepsilon} \\
& \lesssim\left(l \lambda_q^4\right)^{1-\varepsilon}\left(l \lambda_q^6\right)^{\varepsilon} \lesssim l \lambda_q^{4+2 \varepsilon},
\end{aligned}
\end{equation}
and
\begin{equation}\label{eq4.89}
    \begin{aligned}
\left\|\rho_{l}^{-1}-\rho_q^{-1}\right\|_{C_t C_x^{\varepsilon}} & \lesssim\left\|\rho_{l}^{-1}\right\|_{C_t C_x^{\varepsilon}}\left\|\rho_q^{-1}\right\|_{C_t C_x^{\varepsilon}}\left\|\rho_{l}-\rho_q\right\|_{C_t C_x^{\varepsilon}} \\
& \lesssim\left(\lambda_q^{\frac{\varepsilon^2}{4}}\right)^2\left\|\rho_{l}-\rho_q\right\|_{C_{t, x}}^{1-\varepsilon}\left\|\rho_{l}-\rho_q\right\|_{C_t C_x^1}^{\varepsilon} \lesssim l \lambda_q^{\varepsilon} .
\end{aligned}
\end{equation}
Thus, plugging \eqref{eq4.88} and \eqref{eq4.89} into \eqref{eq4.87} and using \eqref{eq4.42} and \eqref{eq4.45} leads that
\begin{equation}\label{eq4.90}
    K_{11}\lesssim\lambda_q^{\frac{\varepsilon^2}{4}}l \lambda_q^{4+2 \varepsilon}\lambda^{4}_q+l \lambda_q^{\varepsilon}\lambda^8_q\lesssim l\lambda_q^9\lesssim l^{\frac{1}{2}}.
\end{equation}
In a similar fashion as $K_{11}$, we can estimate that
\begin{align}
    &K_{21}\lesssim l \lambda_q^{4+2 \varepsilon}\lambda^{4}_q+l \lambda^8_q\lesssim l\lambda_q^9\lesssim l^{\frac{1}{2}},\label{eq4.91}\\
    &K_{31}\lesssim\lambda_q^{\frac{\varepsilon^2}{4}}l \lambda_q^{4+2 \varepsilon}\lambda^{4}_q+l \lambda_q^{\varepsilon}\lambda^8_q\lesssim l\lambda_q^9\lesssim l^{\frac{1}{2}}.\label{eq4.92}
\end{align}
On the other hand, using \eqref{eq2.6}, \eqref{eq2.7} and interpolation we obtain
\begin{equation}\label{eq4.93}
\begin{aligned}
K_{12} & \lesssim\left\|\rho_q^{-1} m_q \otimes m_q-\left(\rho_q^{-1} m_q \otimes m_q\right) *_x \phi_{l} *_t \varphi_{l}\right\|_{C_{t, x}}^{1-\varepsilon}\\
&\quad\times\left\|\rho_q^{-1} m_q \otimes m_q-\left(\rho_q^{-1} m_q \otimes m_q\right) *_x \phi_{l} *_t \varphi_{l}\right\|_{C_t C_x^1}^{\varepsilon} \\
& \lesssim\left(l\left\|\rho_q^{-1} m_q \otimes m_q\right\|_{C_{t, x}^1}\right)^{1-\varepsilon}\left(l\left\|\rho_q^{-1} m_q \otimes m_q\right\|_{C_{t, x}^2}\right)^{\varepsilon} \\
& \lesssim\left(l\left\|\rho_q^{-1}\right\|_{C_{t, x}^1}\left\|m_q\right\|_{C_{t, x}^1}^2\right)^{1-\varepsilon}\left(l\left\|\rho_q^{-1}\right\|_{C_{t, x}^2}^2\left\|m_q\right\|_{C_{t, x}}^2\right)^{\varepsilon} \\
& \lesssim\left(l \lambda_q^{\frac{\varepsilon}{4}} \lambda_q^8\right)^{1-\varepsilon}\left(l \lambda_q^{\frac{\varepsilon}{2}} \lambda_q^{12}\right)^{\varepsilon} \lesssim l \lambda_q^9 \lesssim l^{\frac{1}{2}}.
\end{aligned}
\end{equation}
By a similar fashion as $K_{12}$ again, we can estimate that
\begin{align}
    &K_{22}\lesssim\left(l \lambda_q^8\right)^{1-\varepsilon}\left(l  \lambda_q^{12}\right)^{\varepsilon} \lesssim l \lambda_q^9 \lesssim l^{\frac{1}{2}},\label{eq4.94}\\
    &K_{32}\lesssim\left(l \lambda_q^{\frac{\varepsilon}{4}} \lambda_q^8\right)^{1-\varepsilon}\left(l \lambda_q^{\frac{\varepsilon}{2}} \lambda_q^{12}\right)^{\varepsilon} \lesssim l \lambda_q^9 \lesssim l^{\frac{1}{2}}.\label{eq4.95}
\end{align}
Thus, plugging \eqref{eq4.90}--\eqref{eq4.95} into \eqref{eq4.86}, we derive that
\begin{equation}\label{eq4.96}
    \left\|R^m_{{\rm com},4}\right\|_{L_t^1 C_x^{\varepsilon}}+\left\|R^m_{{\rm com},5}\right\|_{L_t^1 C_x^{\varepsilon}} +\left\|R^b_{{\rm com}}\right\|_{L_t^1 C_x^{\varepsilon}}\lesssim l^{\frac{1}{2}},
\end{equation}
which together with \eqref{eq4.83}--\eqref{eq4.85} gives that
\begin{equation}\label{eq4.97}
    \left\|R^m_{{\rm com}}\right\|_{L_t^1 C_x^{\varepsilon}} +\left\|R^b_{{\rm com}}\right\|_{L_t^1 C_x^{\varepsilon}}\lesssim l^{\frac{1}{2}}.
\end{equation}

Finally, combining \eqref{eq4.69}, \eqref{eq4.74}, \eqref{eq4.75}, \eqref{eq4.78} and \eqref{eq4.97} altogether and using \eqref{eq2.3s} we conclude that
\begin{equation}\label{eq4.98}
    \begin{aligned}
       &\quad\|R^m_{q+1}\|_{L^1_tC_x}+\|R^b_{q+1}\|_{L^1_tC_x}\\
       &\lesssim \|R^m_{{\rm lin}}\|_{L^1_tC_x^{\varepsilon}}+\|R^m_{{\rm osc}}\|_{L^1_tC_x^{\varepsilon}}+\|R^m_{{\rm cor}}\|_{L^1_tC_x^{\varepsilon}}+\|R^m_{{\rm pre}}\|_{L^1_tC_x^{\varepsilon}}+\|R^m_{{\rm com}}\|_{L^1_tC_x^{\varepsilon}}\\
       &\quad+\|R^b_{{\rm lin}}\|_{L^1_tC_x^{\varepsilon}}+\|R^b_{{\rm osc}}\|_{L^1_tC_x^{\varepsilon}}+\|R^b_{{\rm cor}}\|_{L^1_tC_x^{\varepsilon}}+\|R^b_{{\rm com}}\|_{L^1_tC_x^{\varepsilon}}\\
       &\lesssim l^{-27}\lambda_{q+1}^{-9\varepsilon}+l^{-29}\lambda_{q+1}^{-11\varepsilon}+l^{-24}\lambda_{q+1}^{-14\varepsilon}+l^{-29}\lambda_{q+1}^{-13\varepsilon}+l^{\frac{1}{2}}\\
       &\lesssim\delta_{q+2},
    \end{aligned}
\end{equation}
which means that the inductive estimate \eqref{eq2.9} is verified at level $q+1$.
\section{Proof of main results}\label{sec5}
We are now in the stage to prove the main results, namely Theorems \ref{thm-nonuniquenes-1}, \ref{thm-nonuniqueness-2}, \ref{thm-vanishing limit} and \ref{thm-main-iteration}. To begin, we prove the main iteration result Theorem \ref{thm-main-iteration}.
\subsection{Proof of Theorem \ref{thm-main-iteration}}
Since the iterative estimates \eqref{eq2.5}--\eqref{eq2.13} have been verified in the previous sections, we only need to check the inductive estimate \eqref{eq2.15} for the temporal support.

First, by the definition of the perturbations and $T_q$, we have
\begin{equation}\label{eq5.1}
    {\rm supp}_t (w_{q+1},d_{q+1})\subseteq\bigcup_{k\in\Lambda_m\cup\Lambda_b}{\rm supp}_t(a_{(k)})\subseteq N_{2l}([T_q,T]),
\end{equation}
and
\begin{equation}\label{eq5.2}
    {\rm supp}_tz_{q+1}\subseteq N_{2l}([T_q,T]).
\end{equation}
Furthermore, by \eqref{eq3.56} and \eqref{eq3.58},
\begin{equation}\label{eq5.3}
    \begin{aligned}
        {\rm supp}_t(\nabla\rho_{q+1},m_{q+1},b_{q+1}-b_*)
        &\subseteq{\rm supp}_t(\nabla\rho_l,m_l,b_l-b_*)\cup{\rm supp}_t(\nabla z_{q+1},w_{q+1},d_{q+1})\\
        &\subseteq N_{2l}([T_q,T]).
    \end{aligned}
\end{equation}

Next, for the Reynolds and magnetic stresses, by \eqref{eq4.5} and \eqref{eq4.12}, we derive
\begin{equation}\label{eq5.4}
    \begin{aligned}
        {\rm supp}_t(R^m_{q+1},R^b_{q+1})
        &\subseteq\bigcup_{k\in\Lambda_m\cup\Lambda_b}{\rm supp}_t(a_{(k)},m_l,b_l-b_*,\nabla\rho_{q+1},\nabla\rho_l)\cup N_l({\rm supp}_t(R^m_q,R^b_q))\\
        &\subseteq\bigcup_{k\in\Lambda_m\cup\Lambda_b}{\rm supp}_t(a_{(k)},m_l,b_l-b_*,\nabla z_{q+1},\nabla\rho_l)\cup N_l({\rm supp}_t(R^m_q,R^b_q))\\
        &\subseteq N_{2l}([T_q,T]).
    \end{aligned}
\end{equation}
Therefore, combining \eqref{eq5.3} and \eqref{eq5.4} altogether we have
\begin{equation}
    {\rm supp}_t(\nabla\rho_{q+1},m_{q+1},b_{q+1}-b_*,R^m_{q+1},R^b_{q+1})\subseteq N_{2l}([T_q,T]),
\end{equation}
which along with the fact that $2l\ll\delta_{q+2}^{\frac{1}{2}}$ give that \eqref{eq2.15}. Now the proof of Theorem \ref{thm-main-iteration} is complete. \hfill$\square$
\subsection{Proof of Theorem \ref{thm-nonuniquenes-1}}
We prove the statements $(i)$-$(v)$ in Theorem \ref{thm-nonuniquenes-1} below.

$(i)$. Let $\rho_0=\widetilde{\rho}$, $m_0=\widetilde{m}$, $b_0=\widetilde{b}$ and
\begin{align}
    &\begin{aligned}
        R^m_0
        &:=\mathcal{R}^m\left(\partial_t\widetilde{m}+\nu^s(-\Delta)^{\alpha_1}(\widetilde{\rho}^{-1}\widetilde{m})-(\nu^b+\frac{1}{3}\nu^s)\nabla{\rm div}(\widetilde{\rho}^{-1}\widetilde{m})+\nabla P(\widetilde{\rho})\right)\\
        &\quad+\widetilde{\rho}^{-1}\widetilde{m}\otimes\widetilde{m}-\mu\left(\widetilde{b}\otimes\widetilde{b}-\frac{1}{2}|\widetilde{b}|^2\mathbb{I}\right),
    \end{aligned}\label{eq5.6}\\
    &  R^b_0:=\mathcal{R}^b\left(\partial_t\widetilde{b}+\eta\mu^{-1}(-\Delta)^{\alpha_2}\widetilde{b}\right)+\widetilde{\rho}^{-1}\left(\widetilde{b}\otimes\widetilde{m}-\widetilde{m}\otimes\widetilde{b}\right). \label{eq5.7}
\end{align}
Since $(\widetilde{\rho},\widetilde{m})$ is a smooth solution to the transport equation $\eqref{eq1.4}_1$, $\widetilde{b}$ is divergence-free and the identities \eqref{eq4.3}, we derive that $(\rho_0,m_0,b_0,R^m_0,R^b_0)$ is a relaxed solution to \eqref{eq2.1}. Moreover, for $a$ large enough, the inductive estimates \eqref{eq2.5}--\eqref{eq2.9} are all satisfied at level $q=0$. Thus, by Theorem \ref{thm-main-iteration}, there exists a sequence of relaxed solutions $\{(\rho_q,m_q,b_q,R^m_q,R^b_q)\}_{q\geq0}$ to \eqref{eq2.1}, which satisfy \eqref{eq2.5}--\eqref{eq2.15} for any $q\geq0$.

First, concerning the density, using \eqref{eq2.11} we have that for $a$ large enough,
\begin{equation}\label{eq5.8}
    \sum_{q \geq 0}\left\|\rho_{q+1}-\rho_q\right\|_{C_t C_x^1} \lesssim \sum_{q \geq 0} \delta_{q+2}^{\frac{1}{2}}=
a^{\frac{3\beta}{2}b}
\sum_{q\geq0}a^{-\beta b^{q+2}}
\leq \varepsilon_*.
\end{equation}
It means that $\{\rho_q\}_{q\geq0}$ is a Cauchy sequence in the space $C_tC^1_x$, that is, there exists $\rho\in C_tC^1_x$ such that
\begin{equation}\label{eq5.9}
    \lim_{q\rightarrow+\infty}\rho_q=\rho\text{ in }C_tC^1_x,
\end{equation}
which along with \eqref{eq2.5} yields that
\begin{equation}\label{eq5.10}
    \lim_{q\rightarrow+\infty}\rho^{-1}_q=\rho^{-1}\text{ in }C_{t,x}.
\end{equation}

Next, for the momentum and magnetic field, using interpolation inequality, \eqref{eq2.4}, \eqref{eq2.7} and \eqref{eq2.12} we derive that for any $0<\beta^{\prime}<\beta/(4+\beta)$,
\begin{equation}\label{eq5.11}
\begin{aligned}
\sum_{q \geq 0}\left\|m_{q+1}-m_q\right\|_{H_t^{\beta^{\prime}} C_x} & \leq \sum_{q \geq 0}\left\|m_{q+1}-m_q\right\|_{L_t^2 C_x}^{1-\beta^{\prime}}\left\|m_{q+1}-m_q\right\|_{C_{t, x}^1}^{\beta^{\prime}} \\
& \lesssim \sum_{q \geq 0} \delta_{q+1}^{\frac{1-\beta^{\prime}}{2}} \lambda_{q+1}^{4 \beta^{\prime}} \\
& \lesssim \delta_1^{\frac{1-\beta^{\prime}}{2}} \lambda_1^{4 \beta^{\prime}}+\sum_{q \geq 1} \lambda_{q+1}^{-\beta\left(1-\beta^{\prime}\right)+4 \beta^{\prime}}<\infty,
\end{aligned}
\end{equation}
and
\begin{equation}\label{eq5.12}
\begin{aligned}
\sum_{q \geq 0}\left\|b_{q+1}-b_q\right\|_{H_t^{\beta^{\prime}} C_x} & \leq \sum_{q \geq 0}\left\|b_{q+1}-b_q\right\|_{L_t^2 C_x}^{1-\beta^{\prime}}\left\|b_{q+1}-b_q\right\|_{C_{t, x}^1}^{\beta^{\prime}} \\
& \lesssim \sum_{q \geq 0} \delta_{q+1}^{\frac{1-\beta^{\prime}}{2}} \lambda_{q+1}^{4 \beta^{\prime}} \\
& \lesssim \delta_1^{\frac{1-\beta^{\prime}}{2}} \lambda_1^{4 \beta^{\prime}}+\sum_{q \geq 1} \lambda_{q+1}^{-\beta\left(1-\beta^{\prime}\right)+4 \beta^{\prime}}<\infty,
\end{aligned}
\end{equation}
where we use $-\beta(1-\beta^{\prime})+4\beta^{\prime}<0$ in the last inequalities of both \eqref{eq5.11} and \eqref{eq5.12}. Similarly, by \eqref{eq2.13} we also have
\begin{equation}\label{eq5.13}
    \sum_{q \geq 0}\left\|m_{q+1}-m_q\right\|_{C_tH^{-1}_x}<\infty\quad\text{and}\quad\sum_{q \geq 0}\left\|b_{q+1}-b_q\right\|_{C_tH^{-1}_x}<\infty.
\end{equation}
Thus, $\{m_q\}_{q\geq0}$ and $\{b_q\}_{q\geq0}$ are Cauchy sequences in the space $H^{\beta^{\prime}}_tC_x\cap C_tH^{-1}_x$, that is there exist $m,b\in H^{\beta^{\prime}}_tC_x\cap C_tH^{-1}_x$ such that
\begin{equation}\label{eq5.14}
    \lim_{q\rightarrow+\infty}m_q=m\quad\text{and }\lim_{q\rightarrow+\infty}b_q=b\text{ in }H^{\beta^{\prime}}_tC_x\cap C_tH^{-1}_x.
\end{equation}
Furthermore, by \eqref{eq5.10} and \eqref{eq5.14}, we derive that
\begin{equation}\label{eq5.15}
\left\{
    \begin{aligned}
        &\lim_{q\rightarrow+\infty}\rho^{-1}_qm_q=\rho^{-1}m\text{ in }L^2_tC_x,\\
        &\lim_{q\rightarrow+\infty}\rho^{-1}_qm_q\otimes m_q=\rho^{-1}m\otimes m\text{ in }L^1_tC_x,\\
        &\lim_{q\rightarrow+\infty}\left(b_q\otimes b_q-\frac{1}{2}|b_q|^2\mathbb{I}\right)=b\otimes b-\frac{1}{2}|b|^2\mathbb{I}\text{ in }L^1_tC_x,\\
        &\lim_{q\rightarrow+\infty}\rho^{-1}_q\left(b_q\otimes m_q-m_q\otimes b_q\right)=\rho^{-1}\left(b\otimes m-m\otimes b\right)\text{ in }L^1_tC_x,\\
        &m(t)\rightharpoonup f^m_0 \text{ weakly in }H^{-1}(\mathbb{T}^3)\text{ as }t\rightarrow0,\\
        &b(t)\rightharpoonup f^b_0 \text{ weakly in }H^{-1}(\mathbb{T}^3)\text{ as }t\rightarrow0,
    \end{aligned}
    \right.
\end{equation}
for some $f^m_0,f^b_0\in H^{-1}(\mathbb{T}^3)$.

On the other hand, by mean-valued theorem, \eqref{eq5.9} and the continuity of $P$, $P^{\prime}$ and $P^{\prime\prime}$,
\begin{equation}\label{eq5.16}
    \lim_{q\rightarrow+\infty}P(\rho_q)=P(\rho)\text{ in }C_{t,x}.
\end{equation}

Thus, taking into account \eqref{eq2.9},
\begin{equation}\label{eq5.17}
    \lim_{q\rightarrow+\infty}R^m_q=0\quad\text{and }\lim_{q\rightarrow+\infty}R^b_q=0\text{ in }L^1_tC_x.
\end{equation}
We thus conclude that $(\rho,m,b)$ is a weak solution to \eqref{eq1.4}.

$(ii)$. We need only to check
\begin{equation}\label{eq5.18}
    m,b\in L^p_tC^s_x.
\end{equation}
Using \eqref{eq2.13}, by a similar technique of \eqref{eq5.11} and \eqref{eq5.12} we get
\begin{equation}\label{eq5.19}
    \sum_{q\geq0}\|m_{q+1}-m_q\|_{L^p_tC_x^s}<\infty\quad\text{and}\quad\sum_{q\geq0}\|b_{q+1}-b_q\|_{L^p_tC_x^s}<\infty,
\end{equation}
which means that $\{m_q\}_{q\geq0}$ and $\{b_q\}_{q\geq0}$ are both Cauchy sequences in $L^p_tC_x^s$. Hence, by the uniqueness of weak limits, we obtain \eqref{eq5.18} which together with \eqref{eq5.14} leads the regularity statement $(ii)$.

$(iii)$. By \eqref{eq2.10},
\begin{equation}\label{eq5.20}
    \int_{\mathbb{T}^3}\rho_q(t,x){\rm d}x=\int_{\mathbb{T}^3}\rho_0(t,x){\rm d}x=\int_{\mathbb{T}^3}\widetilde{\rho}(t,x){\rm d}x,
\end{equation}
for any $q\in\mathbb{N}$ and $t\in[0,T]$. Then using \eqref{eq5.9} to pass to the limit $q\rightarrow+\infty$, we get
\begin{equation}\label{eq5.21}
    \int_{\mathbb{T}^3}\rho(t,x){\rm d}x=\int_{\mathbb{T}^3}\widetilde{\rho}(t,x){\rm d}x,\quad\forall t\in[0,T],
\end{equation}
which means that the mass preservation statement $(iii)$ is valid.

$(iv)$. The small deviation between $\rho$ and $\widetilde{\rho}$ in $C_tC^1_x$ has been implied by \eqref{eq5.8}. Next, using \eqref{eq2.13} and estimating as in \eqref{eq5.11} and \eqref{eq5.12} we obtain
\begin{equation}\label{eq5.22}
\begin{aligned}
& \|m-\widetilde{m}\|_{L_t^1 C_x}+\|m-\widetilde{m}\|_{L_t^p C_x^s}+\|m-\widetilde{m}\|_{C_t H_x^{-1}} \\
\leq & \sum_{q \geq 0}\left(\left\|m_{q+1}-m_q\right\|_{L_t^1 C_x}+\left\|m_{q+1}-m_q\right\|_{L_t^p C_x^s}+\left\|m_{q+1}-m_q\right\|_{C_t H_x^{-1}}\right) \\
\leq & 2 \sum_{q \geq 0} \delta_{q+2}^{\frac{1}{2}} \leq \varepsilon_*,
\end{aligned}
\end{equation}
and
\begin{equation}\label{eq5.23}
\begin{aligned}
& \|b-\widetilde{b}\|_{L_t^1 C_x}+\|b-\widetilde{b}\|_{L_t^p C_x^s}+\|b-\widetilde{b}\|_{C_t H_x^{-1}} \\
\leq & \sum_{q \geq 0}\left(\left\|b_{q+1}-b_q\right\|_{L_t^1 C_x}+\left\|b_{q+1}-b_q\right\|_{L_t^p C_x^s}+\left\|b_{q+1}-b_q\right\|_{C_t H_x^{-1}}\right) \\
\leq & 2 \sum_{q \geq 0} \delta_{q+2}^{\frac{1}{2}} \leq \varepsilon_*,
\end{aligned}
\end{equation}
for $a$ sufficiently large.

Now we concern the small deviation of relative magnetic helicity between $b$ and $\widetilde{b}$. by a direct calculation,
\begin{equation}\label{eq5.24}
    \|\mathcal{H}_{\rm rel}(b\mid \bar b)-\mathcal{H}_{\rm rel}(\widetilde{b}\mid \bar {\widetilde{b}})\|_{L^1_t}\leq\|b^{\diamond}-\widetilde{b}^{\diamond}\|_{L^1_tC_x}\|\widehat{\widetilde{a}}\|_{C_{t,x}}+\|\widehat{a}-\widehat{\widetilde{a}}\|_{C_tL^2_x}\|b^{\diamond}\|_{L^1_tC_x}.
\end{equation}
Furthermore, using the Biot-Savart's law:
\begin{align*}
       \widehat{a}={\rm curl}(-\Delta)^{-1}b^{\diamond}\quad\text{and}\quad\widehat{\widetilde{a}}={\rm curl}(-\Delta)^{-1}\widetilde{b}^{\diamond},
\end{align*}
that is, by \eqref{eq5.22} and \eqref{eq5.23},
\begin{equation}\label{eq5.25}
    \left\{
    \begin{aligned}
        &\|b^{\diamond}-\widetilde{b}^{\diamond}\|_{L^1_tC_x}\leq\|b-\widetilde{b}\|_{L^1_tC_x}+\|\Bar{b}-\Bar{\widetilde{b}}\|_{L^1_t}\leq \frac{1}{2}\widetilde{C_1}^{-1}\varepsilon_*,\\
        &\|\widehat{\widetilde{a}}\|_{C_{t,x}}=\|{\rm curl}(-\Delta)^{-1}\widetilde{b}^{\diamond}\|_{C_{t,x}}\leq\widetilde{C_1},\\
        &\|\widehat{a}-\widehat{\widetilde{a}}\|_{C_tL^2_x}=\|{\rm curl}(-\Delta)^{-1}(b^{\diamond}-\widetilde{b}^{\diamond})\|_{C_tL^2_x}\leq\|b^{\diamond}-\widetilde{b}^{\diamond}\|_{C_tH^{-1}_x}\leq\frac{1}{2}\widetilde{C_2}^{-1}\varepsilon_*,\\
        &\|b^{\diamond}\|_{L^1_tC_x}\leq\|\widetilde{b}^{\diamond}\|_{L^1_tC_x}+\|b^{\diamond}-\widetilde{b}^{\diamond}\|_{L^1_tC_x}\leq \widetilde{C_2},
    \end{aligned}
    \right.
\end{equation}
for $a$ sufficiently large.

Thus, plugging \eqref{eq5.25} into \eqref{eq5.24} yields the small deviation of magnetic helicity between $b$ and $\widetilde{b}$, which verified the statement $(iv)$.

$(v)$. Finally, for the temporal supports, by \eqref{eq5.9} and \eqref{eq5.14}, we obtain
\begin{equation}\label{eq5.26}
    {\rm supp}_t(\nabla\rho,m,b-b_*)\subseteq N_{\sum_{q\geq0}\delta_{q+2}^{\frac{1}{2}}}([T_*,T])\subseteq N_{\varepsilon_*}([\widetilde{T},T]),
\end{equation}
for $a$ is large enough and the fact that $\sum_{q\geq0}\delta_{q+2}^{\frac{1}{2}}\leq\varepsilon_*$.

Therefore, we have finished the proof of Theorem \ref{thm-nonuniquenes-1}.
\hfill$\square$
\subsection{Proof of Theorem \ref{thm-nonuniqueness-2}}
 Without loss of generality, we choose $\widetilde{\rho}=1$ and  non-trivial divergence-free, mean-free vector fields $\widetilde{m}$ and $\widetilde{b}$:
 \begin{equation}\label{eq5.27}
     \begin{aligned}
         &\widetilde{m}:=\Psi(t)({\rm sin}x_3,0,0)^{{\rm T}},\\
         &\widetilde{b}:=\Psi(t)({\rm sin}x_3,{\rm cos}x_3,0)^{{\rm T}},
     \end{aligned}
 \end{equation}
 where $\Psi(t)$ is any cut-off function such that ${\rm supp}_t(\Psi(t))\subseteq[\frac{T}{4},\frac{3T}{4}]$, $\Psi\equiv1$ on $[\frac{5T}{16},\frac{11T}{16}]$ and $0\leq\Psi(t)\leq1$ on $[0,T]$.
 Next, for every $j\geq1$, let
 \begin{equation}\label{eq5.28}
     \widetilde{m}_j:=\frac{j}{c_0}\widetilde{m}\quad\text{and}\quad\widetilde{b}_j:=\frac{j}{c_1}\widetilde{b}+(0,0,1)^{{\rm T}},
 \end{equation}
 where $j\in\mathbb{N}_+$, $c_0:=\|\widetilde{m}\|_{L^1(0,T;C_x)}$ and $c_1:=\|\widetilde{b}\|_{L^1(0,T;C_x)}$. Thus $(\widetilde{\rho},\widetilde{m}_j)$ are solutions to the transport equation \eqref{eq1.6s} and $\widetilde{b}_j$ are divergence-free.

 Let $0<\varepsilon_*<{\rm min}\{\frac{1}{8},\frac{T}{8}\}$. Then for every $j\geq1$, Theorem \ref{thm-nonuniquenes-1} gives a weak solutions to \eqref{eq1.4} such that
 \begin{equation}\label{eq5.29}
 \left\{
     \begin{aligned}
         &\|m_j-\widetilde{m}_j\|_{L^1(0,T;C_x)}\leq\varepsilon_*,\\
         &\|b_j-\widetilde{b}_j\|_{L^1(0,T;C_x)}\leq\varepsilon_*,\\
         &\|\mathcal{H}_{\rm rel}(b_j\mid \bar b_j)-\mathcal{H}_{\rm rel}(\widetilde{b_j}\mid \bar {\widetilde{b_j}})\|_{L^1(0,T)}\leq\varepsilon_*,
     \end{aligned}
     \right.
 \end{equation}
 and
 \begin{align}
     &\int_{\mathbb{T}^3}\rho_j(t,x){\rm d}x=\int_{\mathbb{T}^3}\widetilde{\rho}(t,x){\rm d}x,\quad\forall t\in[0,T],\label{eq5.30}\\
     &{\rm supp}_t(\nabla\rho_j,m_j,b_j-b_*)\subseteq[\frac{T}{8},T],\label{eq5.31}
 \end{align}
 where $b_*=(0,0,1)^{{\rm T}}$.
 
 Thus, by \eqref{eq5.31}, $\rho_j$ is independent of the spatial variable for $t\in[0,\frac{T}{8}]$, that is, we derive that
 \begin{equation}
     \rho_j(0)=\fint_{\mathbb{T}^3}\rho_j(0,x){\rm d}x=\fint_{\mathbb{T}^3}\widetilde{\rho}_j(0,x){\rm d}x=1.
 \end{equation}
 It means that $(\rho_j,m_j,b_j)$ are weak solutions to \eqref{eq1.4} with the same initial datum $(\rho_j(0),m_j(0),b_j(0))=(1,(0,0,0)^T,(0,0,1)^T)$.

Next, using \eqref{eq5.28} and \eqref{eq5.29}, we derive that for $j\neq k$,
\begin{align*}
    \|m_j-m_k\|_{L^1(0,T;C_x)}
    &\geq\|\widetilde{m}_j-\widetilde{m}_k\|_{L^1(0,T;C_x)}-\|m_j-\widetilde{m}_j\|_{L^1(0,T;C_x)}-\|m_k-\widetilde{m}_k\|_{L^1(0,T;C_x)}\\
    &\geq\frac{|j-k|}{c_0}\|\widetilde{m}\|_{L^1(0,T;C_x)}-2\varepsilon_*\\
    &=|j-k|-2\varepsilon_*>\frac{1}{2},
\end{align*}
and
\begin{align*}
    \|b_j-b_k\|_{L^1(0,T;C_x)}
    &\geq\|\widetilde{b}_j-\widetilde{b}_k\|_{L^1(0,T;C_x)}-\|b_j-\widetilde{b}_j\|_{L^1(0,T;C_x)}-\|b_k-\widetilde{b}_k\|_{L^1(0,T;C_x)}\\
    &\geq\frac{|j-k|}{c_1}\|\widetilde{b}\|_{L^1(0,T;C_x)}-2\varepsilon_*\\
    &=|j-k|-2\varepsilon_*>\frac{1}{2},
\end{align*}
which means that $m_j\neq m_k$ and $b_j\neq b_k$ on $[0,T]$ for every $j\neq k$.

Concerning the relative magnetic helicity,
on the one hand, we have 
\begin{equation}\label{eq5.33}
 \mathcal{H}_{\rm rel}(b_j\mid \bar b_j)(0)=0,\quad\forall j\geq1,
\end{equation}
because of $b_j(0)=(0,0,1)^{{\rm T}}$. On the other hand, the vector potential $\widehat{{\widetilde{a}}}_j$ corresponding to  ${\widetilde{b}}^{\diamond}_j$ can be computed explicitly by
\begin{align*}
    \widehat{\widetilde{a}}_j=\frac{j}{c_1}\Psi(t)({\rm sin}x_3,{\rm cos}x_3,0)^{{\rm T}}.
\end{align*}
This leads that the relative magnetic helicity of $(\rho_j,m_j,b_j)$,
\begin{align*}
    \mathcal{H}_{\rm rel}({\widetilde{b}}_j\mid \bar {\widetilde{b}}_j)(t)=\frac{j^2}{c_1^2}\Psi^2(t)(2\pi)^3,
\end{align*}
and then by the small deviation of relative magnetic helicity \eqref{eq5.29}, for $\varepsilon_*<\frac{Tj^2}{8c_1^2}(2\pi)^3$, we have
\begin{align*}
    \|\mathcal{H}_{\rm rel}(b_j\mid \bar b_j)\|_{L^1_t(\frac{5T}{16},\frac{11T}{16})}\geq\frac{3Tj^2}{8c_1^2}(2\pi)^3-\varepsilon_*>0,
\end{align*}
which combining with \eqref{eq5.33} gives that the relative magnetic helicity $\mathcal{H}_{\rm rel}(b_j\mid \bar b_j)$ is not conserved for every $j\geq1$.

Therefore, the proof of Theorem \ref{thm-nonuniqueness-2} is complete.
\hfill$\square$
\subsection{Proof of Theorem \ref{thm-vanishing limit}}Let $\{\phi_{\varepsilon}\}_{\varepsilon>0}$ and $\{\varphi_{\varepsilon}\}_{\varepsilon>0}$ be two families of standard compactly supported mollifiers on $\mathbb{T}^3$ and $[-T,T]$, respectively. For every $n\geq1$, set
\begin{equation}\label{eq5.34}
    \begin{aligned}     &\rho_n:=\left(\rho*_x\phi_{\lambda_n^{-1}}\right)*_t\varphi_{\lambda_n^{-1}},\quad P_n:=\left(P(\rho)*_x\phi_{\lambda_n^{-1}}\right)*_t\varphi_{\lambda_n^{-1}},\\     &m_n:=\left(m*_x\phi_{\lambda_n^{-1}}\right)*_t\varphi_{\lambda_n^{-1}},\quad b_n:=\left(b*_x\phi_{\lambda_n^{-1}}\right)*_t\varphi_{\lambda_n^{-1}},
    \end{aligned}
\end{equation}
restricted to $[0,T]$, where $\lambda_n:=a^{b^n}$ and $\delta_n:=\lambda^{-2\beta}_n$.

More precisely, we choose $\varepsilon\in\mathbb{Q}_+$ small enough such that
\begin{equation}\label{eq5.35s}
    \varepsilon\leq\frac{1}{20}{\rm min}\left\{1-\alpha,\alpha,\frac{2\alpha}{p}-\alpha-s,\frac{\widetilde{\beta}}{2(1+\widetilde{\beta})},1-\frac{\widetilde{\beta}}{2}\right\},\quad \varepsilon b\in\mathbb{N},
\end{equation}
with $b\in 2\mathbb{N}$ and $\beta>0$ satisfy
\begin{equation}\label{eq5.36s}
    b>\frac{1000}{\varepsilon},\quad0<\beta b^2<{\rm min}\left\{\frac{1}{100},\frac{\widetilde{\beta}}{4}\right\}
\end{equation}
and $a\in\mathbb{N}$ large enough.

Since $(\rho,m,b)$ is a weak solution to the compressible ideal MHD equations \eqref{eq1.5}, we infer that $(\rho_n,m_n,b_n)$ satisfies the following relaxed systems
\begin{equation}\label{eq5.37}
    \left\{
    \begin{aligned}
        &\partial_t\rho_n+{\rm div}m_n=0,\\
        &\partial_tm_n+\kappa_n\nu^s(-\Delta)^{\alpha_1}(\rho^{-1}_nm_n)-\kappa_n(\nu^b+\frac{1}{3}\nu^s)\nabla{\rm div}(\rho^{-1}_nm_n)+\nabla P(\rho_n)\\
        &\qquad+{\rm div}\left(\rho_n^{-1}m_n\otimes m_n-\mu\left(b_n\otimes b_n-\frac{1}{2}|b_n|^2\mathbb{I}\right)\right)={\rm div}R^m_n,\\
        &\partial_tb_n+\kappa_n\eta\mu^{-1}(-\Delta)^{\alpha_2}b_n+{\rm div}\left(\rho_n^{-1}\left(b_n\otimes m_n-m_n\otimes b_n\right)\right)={\rm div}R^b_n,\\
        &{\rm div}b_n=0,
    \end{aligned}
    \right.
\end{equation}
where $\kappa_n=\lambda_n^{-2}$ and the Reynolds and magnetic stresses
\begin{align}
    &   \begin{aligned}
        R^m_n:=
        &\kappa_n\nu^s\mathcal{R}^m(-\Delta)^{\alpha_1}\left(\rho_n^{-1}m_n\right)-\kappa_n(\nu^b+\frac{1}{3}\nu^s)\mathcal{R}^m\nabla{\rm div}\left(\rho_n^{-1}m_n\right)+\mathcal{R}^m\nabla\left(P(\rho_n)-P_n\right)\\
        &+\left(\rho_n^{-1}m_n\otimes m_n-\mu(b_n\otimes b_n-\frac{1}{2}|b_n|^2\mathbb{I})\right)-\\
        &\left(\left(\rho^{-1}m\otimes m-\mu(b\otimes b-\frac{1}{2}|b|^2\mathbb{I})\right)*_x\phi_{\lambda_n^{-1}}\right)*_t\varphi_{\lambda_n^{-1}},
    \end{aligned}\label{eq5.38}\\
    & \begin{aligned}
        R^b_n:
        &=\kappa_n\eta\mu^{-1}\mathcal{R}^b(-\Delta)^{\alpha_2}b_n+\rho_n^{-1}(b_n\otimes m_n-m_n\otimes b_n)\\
    &-\left(\left(\rho^{-1}(b\otimes m-m\otimes b)\right)*_x\phi_{\lambda_n^{-1}}\right)*_t\varphi_{\lambda_n^{-1}}.
    \end{aligned}\label{eq5.39}
\end{align}
\begin{clm}
    For $a$ sufficiently large, $(\rho_n,m_n,b_n,R^m_n,R^b_n)$ satisfy the inductive estimates \eqref{eq2.5}--\eqref{eq2.9} at level $q=n+1$.
\end{clm}

To this end, let us start with the most delicate $L^1_tC_x$-decay estimates \eqref{eq2.9} of $R^m_n$ and $R^b_n$, 
\begin{equation}\label{5.40}
    \begin{aligned}
        \|R^m_n\|_{L^1_tC_x}
        &\lesssim\lambda_n^{-2}\||\nabla|^{2\alpha_1-1}(\rho_n^{-1}m_n)\|_{L^1_tC_x^{\varepsilon}}+\lambda_n^{-2}\|{\rm div}(\rho_n^{-1}m_n)\|_{L^1_tC_x^{\varepsilon}}+\|P(\rho_n)-P_n\|_{L^1_tC_x^{\varepsilon}}\\
        &\quad+\|\rho_n^{-1}m_n\otimes m_n-(\rho^{-1}m\otimes  m)*_x\phi_{\lambda_n^{-1}}*_t\varphi_{\lambda_n^{-1}}\|_{L^1_tC_x^{\varepsilon}}\\
        &\quad+\|b_n\otimes b_n-\frac{1}{2}|b_n|^2\mathbb{I}-(b\otimes b-\frac{1}{2}|b|^2\mathbb{I})*_x\phi_{\lambda_n^{-1}}*_t\varphi_{\lambda_n^{-1}}\|_{L^1_tC_x^{\varepsilon}}\\
        &=:K_1+K_2+K_3+K_4+K_5,
    \end{aligned}
\end{equation}
and
\begin{equation}\label{5.41}
    \begin{aligned}
        \|R^b_n\|_{L^1_tC_x^{\varepsilon}}
        &\lesssim\|\rho_n^{-1}(b_n\otimes m_n-m_n\otimes b_n)-\left(\rho^{-1}(b\otimes m-m\otimes b)\right)*_x\phi_{\lambda_n^{-1}}*_t\varphi_{\lambda_n^{-1}}\|_{L^1_tC_x^{\varepsilon}}\\
        &\quad+\lambda_n^{-2}\||\nabla|^{2\alpha_2-1}b_n\|_{L^1_tC_x^{\varepsilon}}\\
        &=:I_1+I_2.
    \end{aligned}
\end{equation}
By the standard mollification estimates we have
\begin{equation}\label{eq5.42}
    0<c_1\leq\rho_n\leq c_2.
\end{equation}
By a similar technique of \eqref{eq4.50} and \eqref{eq4.51}, we obtain that for $\alpha_1,\alpha_2\in(0,\frac{1}{2}]$:
\begin{align}
    & K_1\lesssim\lambda_n^{-2}\left(\sum_{N_1+N_2=2}\|\rho_n^{-1}\|_{C_tC_x^{N_1}}\|m_n\|_{C_tC_x^{N_2}}\right)^{\frac{\varepsilon}{2}}\lesssim\lambda_n^{-2+\varepsilon},\label{eq5.43}\\  &I_2\lesssim\lambda_n^{-2}\|b_n\|_{C_tC_x^2}^{\frac{\varepsilon}{2}}\lesssim\lambda_n^{-2+\varepsilon},\label{eq5.44}
\end{align}
and for $\alpha_1,\alpha_2\in(\frac{1}{2},1)$:
\begin{align}
    & K_1\lesssim\lambda_n^{-2}\left(\sum_{N_1+N_2=2}\|\rho_n^{-1}\|_{C_tC_x^{N_1}}\|m_n\|_{C_tC_x^{N_2}}\right)^{\frac{2\alpha_1-1+\varepsilon}{2}}\lesssim\lambda_n^{2\alpha_1-3+\varepsilon},\label{eq5.45}\\
   &I_2\lesssim\lambda_n^{-2}\|b_n\|_{C_tC_x^2}^{\frac{2\alpha_2-1+\varepsilon}{2}}\lesssim\lambda_n^{2\alpha_2-3+\varepsilon}.\label{eq5.46}
\end{align}
Thus, we conclude from \eqref{eq5.43}--\eqref{eq5.46} that
\begin{equation}\label{eq5.47}
    K_1\lesssim\lambda_n^{-1+\varepsilon},\quad I_2\lesssim\lambda_n^{-1+\varepsilon},
\end{equation}
and, similarly,
\begin{equation}\label{eq5.48}
    K_2\lesssim \lambda_n^{-1+\varepsilon}.
\end{equation}

Furthermore, by interpolation we derive that
\begin{equation}\label{eq5.49}
\begin{aligned}
K_3& \lesssim\left\|P\left(\rho_{n}\right)-P_n\right\|_{C_{t, x}}^{1-\varepsilon}\left\|P\left(\rho_n\right)-P_n\right\|_{C_t C_x^1}^{\varepsilon} \\
& \lesssim\left(\left\|P\left(\rho_n\right)-P\left(\rho\right)\right\|_{C_{t, x}}+\left\|P\left(\rho\right)-P_n\right\|_{C_{t, x}}\right)^{1-\varepsilon}\left(\left\|P\left(\rho\right)\right\|_{C_t C_x^1}+\left\|P_n\right\|_{C_t C_x^1}\right)^{\varepsilon} \\
& \lesssim\left(\left\|\rho_n-\rho\right\|_{C_{t, x}}+\lambda_n^{-\widetilde{\beta}}\left\|P\left(\rho\right)\right\|_{C_{t, x}^{\widetilde{\beta}}}\right)^{1-\varepsilon}\lambda_n^{\varepsilon}\\
&\lesssim\left(\lambda_n^{-\widetilde{\beta}}\left\|\rho\right\|_{C_{t, x}^{\widetilde{\beta}}}+\lambda_n^{-\widetilde{\beta}}\left\|\rho\right\|_{C_{t, x}^{\widetilde{\beta}}}\right)^{1-\varepsilon}\lambda_n^{\varepsilon}\\
&\lesssim\lambda_n^{-\widetilde{\beta}(1-\varepsilon)+\varepsilon}.
\end{aligned}
\end{equation}

Now we concern the commutator parts $K_4$, $K_5$ and $I_1$. by a direct calculation, we have
\begin{equation}\label{eq5.50}
\begin{aligned}
K_4 &\lesssim  \left\|\rho_n^{-1} m_n \otimes m_n-\left(\rho^{-1} m \otimes m\right) *_x \phi_{\lambda_n^{-1}} *_t \varphi_{\lambda_n^{-1}}\right\|_{C_{t, x}}^{1-\varepsilon} \\
&\quad \times\left\|\rho_n^{-1} m_n \otimes m_n-\left(\rho^{-1} m \otimes m\right) *_x \phi_{\lambda_n^{-1}} *_t \varphi_{\lambda_n^{-1}}\right\|_{C_t C_x^1}^{\varepsilon} \\
&=: K_{41}^{1-\varepsilon} \times K_{42}^{\varepsilon},
\end{aligned}
\end{equation}
Using the standard mollification estimate we obtain
\begin{equation}\label{eq5.51}
\begin{aligned}
K_{41} &\lesssim  \left\|\rho_n^{-1} m_n \otimes\left(m_n-m\right)\right\|_{C_{t, x}}+\left\|\rho_n^{-1}\left(m_n-m\right) \otimes m\right\|_{C_{t, x}}+\left\|\left(\rho_n^{-1}-\rho^{-1}\right) m \otimes m\right\|_{C_{t, x}} \\
& \quad+\left\|\rho^{-1} m \otimes m-\left(\rho^{-1} m \otimes m\right) *_x \phi_{\lambda_n^{-1}} *_t\varphi_{\lambda_n^{-1}}\right\|_{C_{t, x}} \\
&\lesssim  \left\|\rho_n^{-1}\right\|_{C_{t, x}}\left\|m_n-m\right\|_{C_{t, x}}\left(\left\|m_n\right\|_{C_{t, x}}+\|m\|_{C_{t, x}}\right)\\
&\quad+\left\|\rho_n^{-1}-\rho^{-1}\right\|_{C_{t, x}}\|m\|_{C_{t, x}}^2+\lambda_n^{-\widetilde{\beta}}\left\|\rho^{-1} m \otimes m\right\|_{C_{t, x}^{\widetilde{\beta}}} \\
&\lesssim  \lambda_n^{-\widetilde{\beta}}\left\|\rho_n^{-1}\right\|_{C_{t, x}}\|m\|_{C_{t, x}^{\widetilde{\beta}}}\|m\|_{C_{t, x}}+\lambda_n^{-\widetilde{\beta}}\|\rho\|_{C_{t, x}^{\widetilde{\beta}}}\|m\|_{C_{t, x}}^2+\lambda_n^{-\widetilde{\beta}}\left\|\rho^{-1}\right\|_{C_{t, x}^{\widetilde{\beta}}}\|m\|_{C_{t, x}^{\widetilde{\beta}}}^2 \\
&\lesssim  \lambda_n^{-\widetilde{\beta}},
\end{aligned}
\end{equation}
and
\begin{equation}\label{eq5.52}
\begin{aligned}
K_{42} & \lesssim \sum_{N_1+N_2+N_3=1}\left\|\rho_n^{-1}\right\|_{C_t C_x^{N_1}}\left\|m_n\right\|_{C_t C_x^{N_2}}\left\|m_n\right\|_{C_t C_x^{N_3}}+\lambda_n\left\|\rho^{-1} m \otimes m\right\|_{C_{t, x}} \\
& \lesssim \lambda_n .
\end{aligned}
\end{equation}
Therefore, combining \eqref{eq5.50}--\eqref{eq5.52} we get
\begin{equation}\label{eq5.53}
    K_4\lesssim\lambda_n^{-\widetilde{\beta}(1-\varepsilon)+\varepsilon}.
\end{equation}
Moreover, similarly, we also have
\begin{equation}\label{eq5.54}
    K_5\lesssim\lambda_n^{-\widetilde{\beta}(1-\varepsilon)+\varepsilon},\quad  I_1\lesssim\lambda_n^{-\widetilde{\beta}(1-\varepsilon)+\varepsilon}.
\end{equation}
Thus, plugging \eqref{eq5.47}--\eqref{eq5.49}, \eqref{eq5.53} and \eqref{eq5.54} into \eqref{5.40} and \eqref{5.41}, and using \eqref{eq5.36s} and \eqref{eq5.35s} we arrive at
\begin{equation}   \|R^m_n\|_{L^1_tC_x}+\|R^b_n\|_{L^1_tC_x}\lesssim\lambda_n^{-1+\varepsilon}+\lambda_n^{-\widetilde{\beta}(1-\varepsilon)+\varepsilon}\leq \delta_{n+2},
\end{equation}
which means that \eqref{eq2.9} is valid at level $n+1$.

Now we turn to the inductive estimates of \eqref{eq2.5}--\eqref{eq2.8}, by \eqref{eq5.42} we get
\begin{equation}\label{eq5.56}
    \frac{c_1}{2}-\lambda_{n+1}^{-\beta} \leq \rho_n \leq \frac{c_2}{2}+\lambda_{n+1}^{-\beta}.
\end{equation}
Moreover, for $1 \leq N \leq 4$ and $M=0,1$, by a direct calculation, we have
\begin{equation}\label{eq5.57}
    \left\|\partial_t^M \rho_n\right\|_{C_t C_x^N} \lesssim \lambda_n^{M+N}\|\rho\|_{C_{t, x}} \lesssim \lambda_n^{1+N} \leq \lambda_{n+1}^{\frac{\varepsilon}{4}}.
\end{equation}
Hence, \eqref{eq2.5} and \eqref{eq2.6} are verified at level $n+1$.

Moreover, we also see that
\begin{equation}\label{eq5.58}
    \left\|m_n\right\|_{C_{t, x}^N}+\left\|b_n\right\|_{C_{t, x}^N} \lesssim \lambda_n^N\left(\|m\|_{C_{t, x}}+\left\|b_n\right\|_{C_{t, x}^N}\right) \lesssim \lambda_n^N \ll \lambda_{n+1}^{2 N+2},
\end{equation}
which yields \eqref{eq2.7} at level $n+1$.

Finally, for the $C^1_{t,x}$-estimates \eqref{eq2.8}, in view of the Sobolev embedding $W_{t, x}^{1, 5} \hookrightarrow L_{t, x}^{\infty}$, we obtain
\begin{equation}\label{eq5.59}
\begin{aligned}
\left\|R^m_n\right\|_{C_{t, x}^1} 
&\lesssim \left\|P\left(\rho_n\right)\right\|_{W_{t, x}^{2, 5}}+\left\|P(\rho) *_x \phi_{\lambda_n^{-1} }*_t \varphi_{\lambda_n^{-1}}\right\|_{W_{t, x}^{2, 5}}+\left\|\lambda_n^{-2}|\nabla|^{2 \alpha_1-1}\left(\rho_n^{-1} m_n\right)\right\|_{W_{t, x}^{2, 5}} \\
&\quad +\left\|\lambda_n^{-2} {\rm div}\left(\rho_n^{-1} m_n\right)\right\|_{W_{t, x}^{2, 5}}+\left\|\rho_n^{-1} m_n \otimes m_n\right\|_{W_{t, x}^{2, 5}}+\left\|\left(\rho^{-1} m \otimes m\right) *_x \phi_{\lambda_n^{-1}}*_t \varphi_{\lambda_n^{-1}}\right\|_{W_{t, x}^{2, 5}} \\
&\quad+\left\| b_n \otimes b_n-\frac{1}{2}|b_n|^2\mathbb{I}\right\|_{W_{t, x}^{2, 5}}+\left\|\left(b \otimes b-\frac{1}{2}|b|^2\mathbb{I}\right) *_x \phi_{\lambda_n^{-1}}*_t \varphi_{\lambda_n^{-1}}\right\|_{W_{t, x}^{2, 5}} \\
&\lesssim \left\|P\left(\rho_n\right)\right\|_{C_{t, x}^2}+\lambda_n^2\|P(\rho)\|_{C_{t, x}}+\lambda_n^{-2}\left\||\nabla|^{2 \alpha-1}\left(\rho_n^{-1} m_n\right)\right\|_{W_{t, x}^{2, 5}}+\lambda_n^{-2}\left\|\rho_n^{-1} m_n\right\|_{C_{t, x}^3} \\
&\quad +\left\|\rho_n^{-1} m_n \otimes m_n\right\|_{C_{t, x}^2}+\lambda_n^2\left\|\rho^{-1} m \otimes m\right\|_{C_{t, x}}\\
&\quad+\left\| b_n \otimes b_n-\frac{1}{2}|b_n|^2\mathbb{I}\right\|_{C^2_{t,x}}+\lambda_n^2\left\|b \otimes b-\frac{1}{2}|b|^2\mathbb{I}\right\|_{C_{t,x}} \\
&\lesssim \left\|P\left(\rho_n\right)\right\|_{C_{t, x}^2}+\lambda_n^2\|P(\rho)\|_{C_{t, x}}+\lambda_n^{-2 }\left\||\nabla|^{2 \alpha-1}\left(\rho_n^{-1} m_n\right)\right\|_{W_{t, x}^{2, 5}}\\
&\quad+\lambda_n^{-2} \sum_{N_1+N_2 \leq 3}\left\|\rho_n^{-1}\right\|_{C_{t, x}^{N_1}}\left\|m_n\right\|_{C_{t, x}^{N_2}}+\sum_{N_1+N_2+N_3 \leq 2}\left\|\rho_n^{-1}\right\|_{C_{t, x}^{N_1}}\left\|m_n\right\|_{C_{t, x}^{N_2}}\left\|m_n\right\|_{C_{t, x}^{N_3}}\\
&\quad+\sum_{N_1+N_2 \leq 2}\left\|b_n\right\|_{C_{t, x}^{N_1}}\left\|b_n\right\|_{C_{t, x}^{N_2}}\\
&\quad+\lambda_n^2\left\|\rho^{-1} m \otimes m\right\|_{C_{t, x}}+\lambda_n^2\left\|b \otimes b-\frac{1}{2}|b|^2\mathbb{I}\right\|_{C_{t,x}},
\end{aligned}
\end{equation}
and
\begin{equation}\label{eq5.60}
\begin{aligned}
\left\|R^b_n\right\|_{C_{t, x}^1} 
&\lesssim\left\|\lambda_n^{-2}|\nabla|^{2 \alpha_2-1}\left(b_n\right)\right\|_{W_{t, x}^{2, 5}}+\left\|\rho_n^{-1} (b_n \otimes m_n-m_n\otimes b_n)\right\|_{W_{t, x}^{2, 5}}\\
&\quad+\left\|\left(\rho^{-1} (b\otimes m-m\otimes b)\right) *_x \phi_{\lambda_n^{-1}}*_t \varphi_{\lambda_n^{-1}}\right\|_{W_{t, x}^{2, 5}} \\
&\lesssim \lambda_n^{-2}\left\||\nabla|^{2 \alpha_2-1}\left(b_n\right)\right\|_{W_{t, x}^{2, 5}}+\left\|\rho_n^{-1} (b_n \otimes m_n-m_n\otimes b_n)\right\|_{C^2_{t,x}}\\
&\quad+\lambda_n^{2}\left\|\left(\rho^{-1} (b\otimes m-m\otimes b)\right) *_x \phi_{\lambda_n^{-1}}*_t \varphi_{\lambda_n^{-1}}\right\|_{C_{t,x}} \\
&\lesssim \lambda_n^{-2}\left\||\nabla|^{2 \alpha_2-1}\left(b_n\right)\right\|_{W_{t, x}^{2, 5}}+\sum_{N_1+N_2+N_3 \leq 2}\left\|\rho_n^{-1}\right\|_{C_{t, x}^{N_1}}\left\|b_n\right\|_{C_{t, x}^{N_2}}\left\|m_n\right\|_{C_{t, x}^{N_3}}\\
&\quad+\lambda_n^{2}\left\|\left(\rho^{-1} (b\otimes m-m\otimes b)\right) *_x \phi_{\lambda_n^{-1}}*_t \varphi_{\lambda_n^{-1}}\right\|_{C_{t,x}}.
\end{aligned}
\end{equation}
By \eqref{eq1.3} and standard mollification estimates, we have
\begin{equation}\label{eq5.61}
    \begin{aligned}
\left\|P\left(\rho_n\right)\right\|_{C_{t, x}^2} \lesssim & \left\|P\left(\rho_n\right)\right\|_{C_{t, x}}+\left\|P^{\prime}\left(\rho_n\right)\right\|_{C_{t, x}}\left(\left\|\rho_n\right\|_{C_t C_x^2}+\left\|\rho_n\right\|_{C_t^2 C_x}+\left\|\rho_n\right\|_{C_t^1 C_x^1}\right) \\
& +\left\|P^{\prime \prime}\left(\rho_n\right)\right\|_{C_{t, x}}\left(\left\|\rho_n\right\|_{C_t C_x^1}^2+\left\|\rho_n\right\|_{C_t^1 C_x}^2\right) \\
\lesssim & \lambda_n^2.
\end{aligned}
\end{equation}

For the case $\alpha_1,\alpha_2 \in\left(0, \frac{1}{2}\right]$, 
\begin{align}
    &\begin{aligned}
        \left\||\nabla|^{2 \alpha_1-1}\left(\rho_n^{-1} m_n\right)\right\|_{W_{t, x}^{2,5}} &\lesssim\left\|\rho_n^{-1} m_n\right\|_{C_{t, x}^2}\\ &\lesssim \sum_{N_1+N_2 \leq 2}\left\|\rho_n^{-1}\right\|_{C_{t, x}^{N_1}}\left\|m_n\right\|_{C_{t, x}^{N_2}}\\
        &\lesssim \sum_{N_1+N_2 \leq 2} \lambda_n^{N_1+N_2} \lesssim \lambda_n^2,
    \end{aligned}\label{eq5.62}\\
    &\left\||\nabla|^{2 \alpha_2-1}\left(b_n\right)\right\|_{W_{t, x}^{2,5}} \lesssim\left\|b_n\right\|_{C_{t, x}^2}\lesssim \lambda_n^2.\label{eq5.63}
\end{align}
On the other hand, if $\alpha_1,\alpha_2 \in\left(\frac{1}{2}, 1\right)$, an application of the interpolation inequality gives
\begin{equation}\label{eq5.64}
    \begin{aligned}
\left\||\nabla|^{2 \alpha_1-1}\left(\rho_n^{-1} m_n\right)\right\|_{W_{t, x}^{2, 5}}
& \lesssim\left\|\rho_n^{-1} m_n\right\|_{C_{t, x}^{2}}^{2-2 \alpha_1}\left\|\rho_n^{-1} m_n\right\|_{C_{t, x}^3}^{2 \alpha_1-1} \\
& \lesssim\left(\sum_{N_1+N_2 \leq 2}\left\|\rho_n^{-1}\right\|_{C_{t, x}^{N_1}}\left\|m_n\right\|_{C_{t, x}^{N_2}}\right)^{2-2 \alpha_1}\\
&\quad\times\left(\sum_{N_1+N_2 \leq 3}\left\|\rho_n^{-1}\right\|_{C_{t, x}^{N_1}}\left\|m_n\right\|_{C_{t, x}^{N_2}}\right)^{2 \alpha_1-1} \\
& \lesssim\left(\sum_{N_1+N_2 \leq 2} \lambda_n^{N_1+N_2}\right)^{2-2 \alpha_1}\left(\sum_{N_1+N_2 \leq 3} \lambda_n^{N_1+N_2}\right)^{2 \alpha_1-1} \\
&\lesssim \lambda_n^{2 \alpha_1+1},
\end{aligned}
\end{equation}
and
\begin{equation}\label{eq5.65}
\left\||\nabla|^{2 \alpha_2-1}\left(b_n\right)\right\|_{W_{t, x}^{2, 5}}
\lesssim\left\|b_n\right\|_{C_{t, x}^{2}}^{2-2 \alpha_2}\left\|b_n\right\|_{C_{t, x}^3}^{2 \alpha_2-1}\lesssim\lambda_n^{2\alpha_2+1}.
\end{equation}
Plugging \eqref{eq5.61}--\eqref{eq5.65}, we deduce that
\begin{equation}
   \|R^m_n\|_{C^1_{t,x}}+\|R^b_n\|_{C^1_{t,x}}\lesssim\lambda_n^2,
\end{equation}
which verifies \eqref{eq2.8} at level $n+1$.

Then, by the main iteration Theorem \ref{thm-main-iteration} we can obtain a sequence of relaxed solutions $\{(\rho_{n,q},m_{n,q},b_{n,q})\}_{q\geq0}$ to \eqref{eq5.37}. Let $q\rightarrow+\infty$ we obtain a weak solution $(\rho^{(n)},m^{(n)},b^{(n)})\in C_{t,x}\times (H^{\beta^{\prime}}_tC_x)^2$ to \eqref{eq1.4} with parameters $\kappa_n\nu^s$, $\kappa_n\nu^b$ and $\kappa_n\eta$ for some $0<\beta^{\prime}<{\rm min}\{\widetilde{\beta},\beta/(8+\beta)\}$.

Furthermore, we can deduce from \eqref{eq2.7}, \eqref{eq2.12} that
\begin{equation}\label{eq5.67}
\begin{aligned}
\left\|m^{(n)}-m\right\|_{H_t^{\beta^{\prime}} C_x} & \leq\left\|m^{(n)}-m_n\right\|_{H_t^{\beta^{\prime}} C_x}+\left\|m-m_n\right\|_{H_t^{\beta^{\prime}} C_x} \\
& \lesssim \sum_{q=n+1}^{\infty}\left\|m_{n, q+1}-m_{n, q}\right\|_{L_t^2 C_x}^{1-\beta^{\prime}}\left\|m_{n, q+1}-m_{n, q}\right\|_{C_{t, x}^1}^{\beta^{\prime}}+\left\|m-m_n\right\|_{C_t^{\beta^{\prime}} C_x} \\
& \lesssim \sum_{q=n+1}^{\infty} \lambda_{q+1}^{-\beta\left(1-\beta^{\prime}\right)} \lambda_{q+1}^{8 \beta^{\prime}}+\lambda_n^{-\left(\widetilde{\beta}-\beta^{\prime}\right)}\|m\|_{C_{t, x}^{\widetilde{\beta}}} \\
& \lesssim \frac{a^{n b\left(-\beta\left(1-\beta^{\prime}\right)+8 \beta\right)}}{a^{n b\left(\beta\left(1-\beta^{\prime}\right)-8 \beta\right)}-1}+a^{-\left(\widetilde{\beta}-\beta^{\prime}\right) b n} \leq \frac{1}{n},
\end{aligned}
\end{equation}
as well as
\begin{equation}\label{eq5.68}
    \left\|b^{(n)}-b\right\|_{H_t^{\beta^{\prime}} C_x} \lesssim\frac{1}{n},
\end{equation}
and, via \eqref{eq2.11},
\begin{equation}\label{eq5.69}
\begin{aligned}
\left\|\rho^{(n)}-\rho\right\|_{C_{t, x}} & \leq\left\|\rho^{(n)}-\rho_n\right\|_{C_{t, x}}+\left\|\rho-\rho_n\right\|_{C_{t, x}} \\
& \lesssim \sum_{q=n+1}^{\infty}\left\|\rho_{n, q+1}-\rho_{n, q}\right\|_{C_{t, x}}+\left\|\rho-\rho_n\right\|_{C_{t, x}} \\
& \lesssim \sum_{q=n+1}^{\infty} \delta_{q+2}^{\frac{1}{2}}+\lambda_n^{-\widetilde{\beta}}\|\rho\|_{C_{t, x}^{\widetilde{\beta}}} \\
& \lesssim \sum_{q=n+3}^{\infty} a^{-\beta b^q}+\lambda_n^{-\widetilde{\beta}} \lesssim \frac{a^{-\beta b(n+2)}}{a^{\beta b}-1}+\lambda_n^{-\widetilde{\beta}} \leq \frac{1}{n},
\end{aligned}
\end{equation}
where the last step is valid for $a$ large enough.

At last, letting $n\rightarrow+\infty$ we obtain the strong convergence \eqref{eq1.8s} and complete the proof of Theorem \ref{thm-vanishing limit}.
\hfill$\square$
\newpage

\bibliography{ref}
\bibliographystyle{acm}
\end{document}